\documentclass[final,12pt]{elsarticle}
\usepackage{geometry}
\usepackage{microtype}
\usepackage{amsmath}
\usepackage{amssymb}
\usepackage{amsfonts}
\usepackage{amsthm}
\usepackage{bm}
\usepackage{booktabs}
\usepackage{graphicx}
\usepackage{mathtools}
\usepackage{physics}
\usepackage{subcaption}
\usepackage[hidelinks, colorlinks, allcolors=blue, bookmarksdepth=2]{hyperref}
\usepackage{cleveref}
\usepackage{xspace}
\allowdisplaybreaks

\newtheorem{theorem}{Theorem}
\newtheorem{remark}{Remark}

\newtheorem{proposition}{Proposition}
\newtheorem{lemma}{Lemma}
\newtheorem{corollary}{Corollary}

\crefname{assumption}{Assumption}{Assumptions}
\crefname{remark}{Remark}{Remarks}
\crefname{theorem}{Theorem}{Theorems}
\crefname{proposition}{Proposition}{Propositions}
\crefname{lemma}{Lemma}{Lemmas}
\crefname{corollary}{Corollary}{Corollaries}
\crefname{section}{Section}{Sections}
\crefname{figure}{Fig.}{Figs.}
\crefname{table}{Table}{Tables}
\crefname{equation}{}{}
\crefname{appendix}{}{}

\newcommand{\etal}[0]{{et~al.\@}\xspace}
\newcommand{\eg}[0]{{e.g.\@}\xspace}
\newcommand{\ie}[0]{{i.e.\@}\xspace}

\newcommand{\der[2]}[]{\dfrac{{\rm d}#1}{{\rm d}#2}}%
\newcommand{\pder[2]}[]{\dfrac{\partial#1}{\partial#2}}
\newcommand{\fn}[1]{\mathcal{#1}}
\newcommand{\fnb}[1]{\bm{\mathcal{#1}}}
\renewcommand{\H}{\mathsf{H}}
\newcommand{\D}{\mathsf{D}}
\newcommand{\Q}{\mathsf{Q}}
\newcommand{\E}{\mathsf{E}}
\newcommand{\M}{\mathsf{M}}
\newcommand{\I}{\mathsf{I}}
\newcommand{\B}{\mathsf{B}}
\renewcommand{\P}{\mathsf{P}}
\newcommand{\R}{\mathsf{R}}
\newcommand{\N}{\mathsf{N}}

\renewcommand{\S}{\mathsf{S}}

\usepackage{xcolor}

\begin{document}

\hypersetup{pageanchor=false}
\begin{frontmatter}

\title{Entropy Stable Element-Local Continuous-Discontinuous Residual Correction Summation-By-Parts Discretization}

\author{Zelalem Arega Worku\corref{cor1}}
\ead{zelalem.worku@uwaterloo.ca}

\author{David C. Del Rey Fern{\'a}ndez}
\ead{ddelreyfernandez@uwaterloo.ca}

\cortext[cor1]{Corresponding author}

\affiliation{organization={Department of Applied Mathematics, University of Waterloo},
            addressline={200 University Avenue West},
            city={Waterloo},
            postcode={N2L 3G1},
            state={Ontario},
            country={Canada}}

\begin{abstract}
Residual correction provides an efficient alternative to two-point volume flux differencing for enforcing entropy stability in high-order discretizations of hyperbolic partial differential equations. We develop an element-local residual-correction summation-by-parts (SBP) method that combines continuous-SBP (C-SBP) and discontinuous-SBP (D-SBP) coupling on conforming meshes. The construction retains continuous assembly without simultaneous approximation terms (SATs) wherever suitable, restricts SATs to selected discontinuous facets, and permits the coupling to change in time as a foundation for adaptive discretizations. The scalar entropy defect and a graph-based distribution vector acting on conservative variables are computed entirely from element-local data, without neighboring states, auxiliary gradients, or correction SATs. The same correction applies to every facet partition, including the pure C-SBP and D-SBP limits. For fixed partitions, we establish conservation and entropy identities for each element residual, exact global conservation, entropy stability up to an explicit regularization defect, and design-order consistency. Changes in the partition are handled by a mass-weighted orthogonal projection that preserves conservative moments and convex admissible sets, does not increase entropy, and does not amplify error in the broken norm. Numerical experiments in one to three dimensions show that the hybrid method retains the smooth-solution accuracy of its D-SBP endpoint under repeated repartitioning. The nonlinear tests demonstrate robustness comparable to Hadamard-form two-point-flux discretizations, particularly when coupled with artificial dissipation.
\end{abstract}



\end{frontmatter}
\hypersetup{pageanchor=true}

\section{Introduction}\label{sec:introduction}

The state-of-the-art in high-order methods for nonlinear conservation laws
increasingly relies on discrete entropy stability to improve robustness.
The summation-by-parts (SBP) property is a key ingredient in many
entropy-stable formulations, since it provides a discrete analogue of
integration by parts and makes it possible to mimic continuous conservation
and entropy estimates. A widely used route to entropy stability is
two-point flux differencing
\cite{fisher2013high,gassner2016split,crean2018entropy}. This approach is
powerful but can be substantially more expensive than standard
divergence-form discretizations, especially at high polynomial degrees and
on elements where tensor-product structure is not available. It can also
be cumbersome to derive suitable two-point fluxes for new systems.
Entropy-split methods
\cite{harten1983symmetric,olsson1994energy,gerritsen1996designing,yee2000entropy,yee2023recent,worku2023entropy}
provide another path to entropy stability, but they are often
non-conservative. A third route, developed in the residual-correction
framework of Abgrall \etal
\cite{abgrall2018general,abgrall2022reinterpretation}, starts from a chosen
base residual and adds an algebraic correction whose moments enforce the
desired mathematical structures. This moment-based viewpoint is attractive because it separates the
construction of the base discretization from the enforcement of
additional residual identities, such as conservation and entropy
stability, allowing the base residual to be evaluated using efficient
divergence-form discretizations.

The cost of an entropy-stable discretization also depends on how neighboring
elements are coupled. Continuous SBP (C-SBP) is attractive on matching
facets because it avoids simultaneous approximation terms (SATs)
\cite{hicken2020entropy}, whose cost can be substantial for viscous systems
such as the Navier--Stokes equations. Discontinuous SBP (D-SBP) instead
couples neighboring elements through SATs (numerical fluxes) and
permits distinct interface traces. Local \(hp\)-adaptation
can produce nonmatching facet nodes for which direct C-SBP assembly is
not available. This motivates combining continuous and discontinuous
coupling, retaining C-SBP wherever suitable while allowing the C/D
partition to change in time. Such hybrid formulations have a substantial
history. Representative examples include the work
of Cangiani \etal \cite{cangiani2013stability}, who coupled continuous
and discontinuous finite-element regions for advection--diffusion--reaction problems. Abdi and Giraldo \cite{abdi2016efficient} recast continuous assembly in element-local storage and selected, at interface nodes, either direct stiffness summation or discontinuous Rusanov-flux coupling; they demonstrated regional mixing for the compressible Euler equations and observed that the
selection could be changed according to the flow regime. Kopera and Giraldo \cite{kopera2015mass}
established mass conservation for unified continuous/discontinuous
element-based Galerkin machinery on dynamically adaptive grids. Related
clustered constructions have been used for Maxwell's equations
\cite{angulo2015nodal}, while Kopriva and Gassner \cite{kopriva2021split}
proved stability, conservation, and constant preservation for a fixed
hybrid CG/DG spectral-element partition of linear hyperbolic systems.
The question addressed in this paper is how to combine this coupling
flexibility with an inexpensive, element-local residual correction for
nonlinear conservation laws, while preserving conservation and entropy stability for dynamic partitioning. 

To apply the same correction in C-SBP and D-SBP regions, we seek an
entropy defect and distribution vector computed entirely from
element-local data. In C-SBP discretizations, each element residual is
evaluated and corrected independently before contributions at shared
nodes are assembled into a global residual, allowing elementwise
parallel evaluation and correction as in D-SBP.
The divergence-form base residual is already conservative, and a correction
with zero conservative moment preserves this property. Existing D-SBP
correction mechanisms can introduce element coupling through either the
entropy defect or its distribution. In the constructions of
\cite{abgrall2018general,abgrall2022reinterpretation,gaburro2023high,mantri2024fully,alberti2024entropy},
the entropy defect contains both volume and SAT contributions and therefore
depends, as originally written, on neighboring traces. For compatible
entropy-conservative facet coupling, however, these trace-dependent
contributions cancel exactly, giving an algebraically equivalent
element-local implementation. In the entropy-correction artificial-viscosity
(ECAV) constructions \cite{chan2025artificial,van2026choice}, the entropy
defect is element-local, but the defect is distributed through viscous
volume and interface operators; the resulting viscous SATs and auxiliary
gradients couple neighboring elements. We examine the relation between these methods and the present
construction in \cref{sec:relation_existing_corrections}.

To design an element-local correction, we separate interface coupling from entropy correction,
leaving all inviscid SATs in the standard divergence-form base residual.
The entropy defect follows the construction of Chan
\cite{chan2025artificial}, while the correction distribution uses a
reference-element graph Laplacian acting on conservative variables,
motivated by the derivative-based corrections of Edoh
\cite{edoh2024conservative}. The correction has zero conservative
moment and requires no auxiliary gradients or correction SATs.
The reference-element distribution matrix can be precomputed;
evaluating the correction requires only local matrix-vector products
and scalar reductions. Neither the entropy defect nor the distribution
depends explicitly on the facet partition. The same correction therefore
applies to a time-dependent, facetwise C/D-SBP partition, with direct
assembly on C-SBP facets and conservative, entropy-stable numerical-flux
coupling on D-SBP facets.

For every fixed partition, we prove conservation and entropy identities
for each element residual, exact global conservation, entropy stability
up to the explicit regularization defect, design-order consistency, and
exact recovery of the pure C-SBP and D-SBP endpoints. Allowing the
partition to change additionally requires a transfer between the
corresponding partially assembled spaces. We transfer the state by a
mass-weighted orthogonal projection onto the space defined by the new
C-SBP facets. We prove that this transfer preserves the global conservative
moments and convex admissible sets, does not increase a convex entropy,
and is error non-amplifying in the broken norm. To complement the theoretical analysis, numerical experiments assess convergence, long-time entropy behavior, and robustness relative to split-form and Hadamard-form two-point-flux baselines. The targeted artificial dissipation of \cite{worku2026targeted} is added separately in the under-resolved tests. The analysis and experiments
are restricted to periodic, conforming inviscid problems; physical
boundaries, nonconforming and viscous extensions, and adaptive mesh
transfer remain future work. 

The remainder of the paper is organized
as follows. \Cref{sec:preliminaries} reviews entropy pairs and SBP notation,
and \cref{sec:correction_coefficients} defines the element-local correction
distribution. \Cref{sec:hybrid_cdsbp} constructs and analyzes the hybrid
method for a fixed C/D partition, while \cref{sec:dynamic_partition}
treats dynamic C/D repartitioning.
\Cref{sec:consistency_residual_correction} establishes design-order
consistency, \cref{sec:relation_existing_corrections} relates the formulation
to existing residual-correction methods, and
\cref{sec:numerical_experiments} presents the numerical results.

\section{Preliminaries}\label{sec:preliminaries}
\subsection{Entropy pairs and residual moments}
Let \(d\) denote the spatial dimension, and let \(n_c\) denote the number of
component variables in a system of partial differential equations (PDEs). On a compact and connected domain
\(\Omega\subset\mathbb{R}^{d}\), we consider
\begin{equation}\label{eq:hyperbolic_pde}
    \pder[\fnb{U}]{t}
    +
    \sum_{i=1}^{d}
    \pder[\fnb{F}^{(i)}(\fnb{U})]{x_i}
    =
    \bm{0},
\end{equation}
with periodic boundary conditions and with an admissible initial condition. The vector or scalar function,
\(\fnb{U}:\Omega\times[0,T]\to\mathbb{R}^{n_c}\), denotes the conservative state with \(n_c\) components, and
\(\fnb{F}^{(i)}:\mathbb{R}^{n_c}\to\mathbb{R}^{n_c}\) is the flux in the
\(x_i\)-direction.

For a strictly convex entropy \(\fn{S}(\fnb{U})\), the entropy variables are
\begin{equation}\label{eq:entropy_variables}
    \fnb{W}^{T}
    \coloneqq
    \pder[\fn{S}]{\fnb{U}}.
\end{equation}
We assume that entropy fluxes \(\fn{G}^{(i)}\) exist such that
\begin{equation}\label{eq:entropy_flux_compatibility}
    \fnb{W}^{T}
    \pder[\fnb{F}^{(i)}]{\fnb{U}}
    =
    \pder[\fn{G}^{(i)}]{\fnb{U}},
    \qquad i=1,\dots,d.
\end{equation}
Thus, smooth solutions satisfy
\begin{equation}\label{eq:entropy_balance_general}
    \pder[\fn{S}]{t}
    +
    \sum_{i=1}^{d}
    \pder[\fn{G}^{(i)}]{x_i}
    =
    0,
\end{equation}
whereas entropy solutions containing shocks satisfy the corresponding inequality in the sense of distributions \cite{dafermos2005hyperbolic}. We also use the entropy potential defined as
\begin{equation}\label{eq:entropy_potential}
    \psi^{(i)}
    \coloneqq
    \fnb{W}^{T}\fnb{F}^{(i)}
    -
    \fn{G}^{(i)},
\end{equation}
for which \(\partial\psi^{(i)}/\partial\fnb{W}=[\fnb{F}^{(i)}]^T\). These continuous identities provide the moments that the residual corrections enforce at the discrete level.

\subsection{Element notation and SBP identities}\label{sec:element_notation}
Let \(\mathcal{T}_{h} \coloneqq \{\{\Omega_k\}_{k=1}^{n_e}:\Omega=\cup_{k=1}^{n_e}\Omega_k\}\) be a conforming, non-overlapping tessellation of \(\Omega\) with \(n_e\) elements. We consider a family \(\{\mathcal{T}_h\}_{h\to0}\) of such tessellations and define the nominal mesh size by
\(
        h
        \coloneqq
        ({|\Omega|}/{n_e})^{1/d},
\)
where \(|\Omega|\) is the volume of the domain. For each element, let
\(h_k\coloneqq\operatorname{diam}(\Omega_k)\).
The boundary of \(\Omega_k\) is denoted by \(\Gamma_k\), and \(\Gamma^I\) denotes the set of interior facets, including periodic facet pairs. Every element \(\Omega_k\) contains \(n_p\) solution nodes, and each facet
\(\gamma\subset\Gamma_k\) contains \(n_\gamma\) facet nodes.  Whenever
element nodes are identified by an assembly map, \(n_g\) denotes the number
of resulting global solution nodes. We use
\(\bm{1}_k\in\mathbb{R}^{n_p}\), \(\bm{1}_\gamma\in
\mathbb{R}^{n_\gamma}\), \(\bm{1}_g\in\mathbb{R}^{n_g}\), and \(\I_m\in\mathbb{R}^{m\times m}\) for the
corresponding vectors of ones and identity matrix. The symbol
\(\bm{0}\) denotes a vector or matrix consisting of all zeros, with its size
clear from context. A nodal state vector on
\(\Omega_k\) is denoted by a bold small letter, \eg, \(\bm u_k\in\mathbb{R}^{n_c n_p}\).
A collection of element-local state vectors is called a broken state
when separate nodal copies are retained on adjacent elements, thereby
permitting discontinuities across element interfaces. If the corresponding
traces agree across every interface, the state will be referred to as
continuous. If \(\mathsf{X}\) acts on scalar nodal values, then
\(\overline{\mathsf{X}}\coloneqq\mathsf{X}\otimes\I_{n_c}\) denotes its componentwise extension to systems, where \(\otimes\) denotes the Kronecker product. Constant states are injected by
\begin{equation}\label{eq:identity_matrices}
    \overline{\I}_k
    \coloneqq
    \bm{1}_{k}\otimes\I_{n_c},
    \qquad
    \overline{\I}_g
    \coloneqq
    \bm{1}_{g}\otimes\I_{n_c},
\end{equation}
that is, \(\overline{\I}_k\bm c\) repeats a constant state
\(\bm c\in\mathbb{R}^{n_c}\) at every node of \(\Omega_k\). On a facet, we write \(\overline{\I}_{\gamma}\coloneqq\bm{1}_{\gamma}\otimes\I_{n_c}\). The vector \(\overline{\I}_g\bm c\) is the corresponding constant state in any assembled hybrid space.

On each physical element, the degree-\(p\) SBP derivative operator in the \(x_i\)-direction, denoted by \(\D_{x_i k}\), satisfies the identities
\begin{equation}\label{eq:local_sbp_operator}
    \D_{x_i k}
    =
    \H_k^{-1}\Q_{x_i k},
    \quad
    \Q_{x_i k}
    =
    \S_{x_i k}
    +
    \frac{1}{2}\E_{x_i k},
    \quad
    \S_{x_i k}
    =
    -\S_{x_i k}^{T},
    \quad
    \E_{x_i k}
    =
    \E_{x_i k}^{T}.
\end{equation}
Equivalently, we have \(\Q_{x_i k}+\Q_{x_i k}^{T}=\E_{x_i k}\), which is referred to as the SBP property. The matrix \(\H_k\) is a symmetric positive definite norm matrix, and in this paper, we only consider operators with diagonal-norm matrices. Writing
\[
        \H_k
        =
        \operatorname{diag}(w_{k,1},\ldots,w_{k,n_p}),
        \qquad
        w_{k,m}>0,
\]
we impose two assumptions on the mesh family. First, we assume shape
regularity. Specifically, let \(\fn{M}_k:\hat\Omega\to\Omega_k\) be a smooth invertible mapping from a fixed reference element, \(\hat{\Omega}\), let \(\mathsf{J}_k(\bm\xi)\coloneqq \partial\fn{M}_k(\bm\xi)/\partial\bm\xi\) denote its Jacobian matrix, and let \(\sigma_{\min}(\mathsf{J}_k(\bm\xi))\) and \(\sigma_{\max}(\mathsf{J}_k(\bm\xi))\) denote its smallest and largest singular values, respectively. Then, there exist constants \(c_J,C_J>0\), independent of \(k\) and \(h\), such that
\[
        c_Jh_k
        \le
        \sigma_{\min}\!\left(\mathsf{J}_k(\bm\xi)\right)
        \le
        \sigma_{\max}\!\left(\mathsf{J}_k(\bm\xi)\right)
        \le
        C_Jh_k,
        \qquad
        \bm\xi\in\hat\Omega.
\]
For the mapped diagonal-norm SBP operators considered here, the
physical and reference quadrature weights are related by
\(w_{k,m}=\hat w_m|\det(\mathsf{J}_k(\bm\xi_m))|\), where the reference
weights \(\hat w_m>0\) are fixed. The shape-regularity assumption therefore
gives constants \(c_s,C_s>0\), independent of \(k\), \(h\), and \(m\), such
that
\[
        c_s h_k^d
        \le
        w_{k,m}
        \le
        C_s h_k^d.
\]
Second, we assume that the mesh family is quasi-uniform in the sense
that there exist constants \(c_H,C_H>0\), independent of \(k\), \(h\), and
\(m\), such that
\begin{equation}\label{eq:quasi_uniformity}
        c_Hh^d
        \le
        w_{k,m}
        \le
        C_Hh^d,
        \qquad
        m\in\{1,\ldots,n_p\}.
\end{equation}
Comparing the two weight estimates gives
\begin{equation}\label{eq:local_global_mesh_size}
        c_h h
        \le
        h_k
        \le
        C_h h,
\end{equation}
where \(c_h,C_h>0\) are independent of \(k\) and \(h\).

The boundary matrix can be decomposed as
\begin{equation}\label{eq:local_sbp_boundary}
    \E_{x_i k}
    =
    \sum_{\gamma\subset\Gamma_k}
    \R_{\gamma k}^{T}
    \B_{\gamma}
    \N_{x_i\gamma k}
    \R_{\gamma k},
\end{equation}
where \(\R_{\gamma k}\) extrapolates the solution from the volume nodes to
the facet nodes, \(\B_\gamma\) contains the reference facet quadrature
weights, and \(\N_{x_i\gamma k}\) contains the surface-scaled outward normal
component in the \(x_i\)-direction. We assume exact restriction of constants, \ie,
\begin{equation}\label{eq:facet_constant_restriction}
    \R_{\gamma k}\bm1_k=\bm1_\gamma.
\end{equation}
In this work, we exclusively consider diagonal-norm, diagonal-\(\E\) SBP operators, for which \(\R_{\gamma k}\) is a Boolean nodal extraction because each facet node is collocated with a volume node. Consequently, restriction commutes with evaluation of every pointwise state function. Other C-SBP-compatible operators include those with dense boundary matrices on shared face nodes \cite{hicken2016multidimensional}. An entropy-stable extension to nonmatching or noncollocated traces requires compatible projection or mortar operators and is outside the present analysis.

A hybrid C/D-SBP discretization is specified by a partition of the interior facets,
\begin{equation}\label{eq:cdsbp_interface_partition}
    \Gamma^I
    =
    \Gamma_{\rm C}^I
    \cup
    \Gamma_{\rm D}^I,
\end{equation}
where facets in \(\Gamma_{\rm C}^I\) are coupled by strong assembly and
facets in \(\Gamma_{\rm D}^I\) are coupled weakly using SATs. Periodic facet pairs are included in \(\Gamma^I\). For a partition \(\chi\), let
\[
    \Gamma_{{\rm C},k}^{\chi}
    \coloneqq
    \{\gamma\subset\Gamma_k:\gamma\in\Gamma_{\rm C}^I\},
    \qquad
    \Gamma_{{\rm D},k}^{\chi}
    \coloneqq
    \{\gamma\subset\Gamma_k:\gamma\in\Gamma_{\rm D}^I\}.
\]
Matching element-local nodal copies are treated as one on C-SBP facets, while they are treated as separate entities on D-SBP facets. All copies connected by one or
more C-SBP facets are assigned the same global solution node. Let \(n_g^\chi\)
denote the number of resulting global solution nodes and let
\(\P_k^\chi\in\mathbb R^{n_p\times n_g^\chi}\) restrict a partially assembled
vector to \(\Omega_k\). Every row of \(\P_k^\chi\) has one nonzero entry equal
to one, and
\begin{equation}\label{eq:cdsbp_partial_assembly_map}
    \P_k^\chi\bm1_g^\chi=\bm1_k,
    \qquad
    \bm u_k=\overline{\P}_k^\chi\bm u_g^\chi.
\end{equation}
For every \(\gamma\in\Gamma^I\), the surface metrics of the adjacent elements
\(\Omega_k\) and \(\Omega_v\) satisfy
\begin{equation}\label{eq:cdsbp_paired_metric_compatibility}
    \N_{x_i\gamma v}=-\N_{x_i\gamma k},
    \qquad i=1,\ldots,d.
\end{equation}
This paired discrete-normal identity includes the periodic facet pairs. On a C-SBP facet, the traces also satisfy
\begin{equation}\label{eq:cdsbp_continuous_trace}
    \overline{\R}_{\gamma k}\overline{\P}_k^\chi\bm u_g^\chi
    =
    \overline{\R}_{\gamma v}\overline{\P}_v^\chi\bm u_g^\chi,
    \qquad
    \gamma=\Gamma_k\cap\Gamma_v\in\Gamma_{\rm C}^I.
\end{equation}
No such equality is imposed on a D-SBP facet. The partially assembled norm is
\begin{equation}\label{eq:cdsbp_assembled_norm}
    \H_g^\chi
    \coloneqq
    \sum_{\Omega_k\in\mathcal T_h}
    (\P_k^\chi)^T\H_k\P_k^\chi.
\end{equation}
It is diagonal and positive definite.  The all-continuous partition gives the
usual C-SBP assembly, while the all-discontinuous partition assigns a separate
global solution node to every element-local nodal copy and therefore gives the
block-diagonal broken norm.

For sufficiently smooth functions, the derivative operator, \(\D_{x_i k}\), satisfies
\begin{equation}\label{eq:D_accuracy}
    \left[\D_{x_i k}\bm v_k\right]_m
    =
    \pder[\fn{V}]{x_i}(\bm{x}_{k,m})
    +
    \fn{O}(h^p), \qquad \forall \,\fn{V}\in C^{p+1},
\end{equation}
where \(\bm v_k\) denotes the nodal restriction of \(\fn{V}\) on \(\Omega_k\) \cite{crean2018entropy}. The notation \(\fn{O}(h^p)\) is interpreted componentwise for vector-valued quantities. The reference volume
and facet quadrature rules underlying \(\H_k\) and \(\B_\gamma\) are exact
for polynomials of degrees at least \(2p-1\) and \(2p\), respectively. For
a prescribed set of volume nodes, the existence of a positive-weight
quadrature rule of degree at least \(2p-1\) is necessary for
the existence of a degree-\(p\) diagonal-norm SBP operator
\cite{hicken2013summation}. For the collocated operators considered here, \(\R_{\gamma k}\)
extracts the facet nodal values exactly.

For the all-continuous partition, we suppress the superscript \(\chi\) and
define the assembled operators
\begin{equation}\label{eq:csbp_oper_def}
    \Q_{x_i}
    \coloneqq
    \sum_{\Omega_k\in\mathcal{T}_h}\P_k^T\Q_{x_i k}\P_k,
    \qquad
    \D_{x_i}
    \coloneqq
    \H_g^{-1}\Q_{x_i},
\end{equation}
where \(\H_g\) is the all-continuous specialization of
\cref{eq:cdsbp_assembled_norm}. Further details on C-SBP assembly and general
background on multidimensional SBP operators can be found in
\cite{hicken2020entropy,hicken2016multidimensional,fernandez2014review,svard2014review,crean2018entropy}.

\section{Correction distribution vector}\label{sec:correction_coefficients}
We seek a distribution vector whose conservative moment vanishes and whose
entropy moment is one before regularization. To construct it, we choose an
element-local vector \(\bm a_k\) such that
\[
    \overline{\I}_k^T\bm a_k=\bm0,
    \qquad
    b_k\coloneqq\bm w_k^T\bm a_k>0
\]
for every nonconstant admissible nodal state. Then, we define the unregularized distribution vector as
\begin{equation}
    \bm\alpha_k \coloneqq \frac{\bm a_k}{b_k},
\end{equation}
which satisfies the required moments. The choice of
\(\bm a_k\) determines how the correction is distributed among the nodes of element \(\Omega_k\).

Let \(\bm u_k\) and \(\bm w_k\) contain the nodal conservative and entropy
variables on \(\Omega_k\), respectively, and define the element volume by
\begin{equation}\label{eq:V_k}
    V_k
    \coloneqq
    \bm{1}_k^T\H_k\bm{1}_k.
\end{equation}
Several existing choices motivate our distribution. Abgrall
\cite{abgrall2018general} used the arithmetic mean of the entropy variables;
Abgrall, {\"O}ffner, and Ranocha (AOR)
\cite{abgrall2022reinterpretation} used a mass-weighted mean. For diagonal
\(\H_k\), their nodal and mass-matrix-multiplied distribution matrices,
respectively, are
\begin{equation}\label{eq:mean_subtracted_distribution_matrices}
    \mathsf L^{\rm A}
    \coloneqq
    \I_{n_p}-\frac{\bm1_k\bm1_k^T}{n_p},
    \qquad
    \mathsf L_k^{\rm AOR}
    \coloneqq
    \H_k-\frac{\H_k\bm1_k\bm1_k^T\H_k}{V_k}.
\end{equation}
Both operators act on the nodal entropy variables and are complete-graph
Laplacians: every pair of element nodes is coupled, even if the SBP
derivative matrices are sparse. These distributions can nevertheless be applied in
\(\fn{O}(n_p)\) by computing the element mean once and performing
nodal operations, without forming the dense matrices. The ECAV distribution mechanism of
\cite{chan2025artificial} instead uses viscous volume and SAT terms that
couple neighboring elements. For comparison, Chan \cite[Sec.~5]{chan2025artificial}
also studies two element-local corrections weighted by
\(\partial\fnb U/\partial\fnb W\) evaluated at the element-mean state.
The mean-subtracted option couples every pair of element nodes, as in
\cref{eq:mean_subtracted_distribution_matrices}, and produced nonphysical
artifacts in a high-order two-dimensional Riemann problem; the derivative-based
option did not stabilize that test. 

These considerations motivate an element-local distribution acting on conservative variables. The construction proceeds in three steps:
\begin{enumerate}
    \item form a precomputable matrix from reference-element SBP derivatives;
    \item use this matrix to construct a graph Laplacian acting on
        conservative variables, avoiding differentiation of entropy variables
        and ensuring a positive entropy moment for nonconstant admissible
        states;
    \item regularize the volume-normalized entropy moment to handle
        constant and nearly constant states, while preserving the zero
        conservative moment and explicitly accounting for the resulting
        regularization defect.
\end{enumerate}

To form the precomputable matrix, let \(\D_{\xi_i}\) and \(\H_\xi\) be the scalar SBP derivative and norm matrices on the reference element. With
\(\hat V\coloneqq\bm1_k^T\H_\xi\bm1_k\), define
\begin{equation}\label{eq:scaled_reference_distribution_matrix}
    \hat{\M}
    \coloneqq
    \sum_{i=1}^{d}
    \D_{\xi_i}^{T}\H_{\xi}\D_{\xi_i},
    \qquad
    \mu_k
    \coloneqq
    \left(\frac{V_k}{\hat V}\right)^{\frac{d-2}{d}},
\end{equation}
where \(\hat{\M}\) can be precomputed as all its components are reference matrices. Applying it directly to conservative variables gives an element-local \(\bm a_k\) vector, but its positive semidefiniteness does not ensure a positive mixed moment \(b_k\). Applying it to entropy variables gives
a nonnegative quadratic moment, but differentiates those variables. For the
Euler equations, entropy variables can become large near low pressure, which
may produce large nodal corrections; related difficulties with
entropy-variable-jump dissipation have been reported in
\cite{derigs2017novel}. A conservative-variable diffusion based on entropy
variables may also require the state-dependent matrix
\(\partial\fnb U/\partial\fnb W\), whose nodal evaluation adds work. We
therefore use \(\hat{\M}\) to define nonnegative symmetric edge weights
\(\ell_{mn}\coloneqq|[\hat{\M}]_{mn}|\) for \(m\ne n\). The resulting graph
Laplacian is
\begin{equation}\label{eq:reference_distribution_graph}
    [\hat{\mathsf L}]_{mn}
    \coloneqq-\ell_{mn}\quad(m\ne n),
    \qquad
    [\hat{\mathsf L}]_{mm}
    \coloneqq\sum_{n\ne m}\ell_{mn}.
\end{equation}
Thus, \(\hat{\mathsf L}\bm1_k=\bm0\), and its off-diagonal sparsity pattern
is that of \(\hat{\M}\). We use reference operators for which the graph
with edges \(\ell_{mn}>0\) is connected; connectivity is checked when
\(\hat{\mathsf L}\) is constructed. The same fixed reference operator is used
on every element. By \cref{eq:quasi_uniformity,eq:V_k},
\(n_pc_Hh^d\le V_k\le n_pC_Hh^d\), so
\begin{equation}\label{eq:mu_scaling}
        c_\mu h^{d-2}
        \le
        \mu_k
        \le
        C_\mu h^{d-2}.
\end{equation}
The factor \(\mu_k\) gives the graph distribution the same mesh-size
scaling as \(\hat{\M}\) mapped to a physical element. We define
\begin{equation}\label{eq:coefficient_ab}
    \bm{a}_k
    \coloneqq
    \mu_k(\hat{\mathsf L}\otimes\I_{n_c})\bm{u}_k,
    \qquad
    b_k
    \coloneqq
    \bm{w}_k^T\bm{a}_k.
\end{equation}
The graph Laplacian is symmetric with zero row sums, so
\(\overline{\I}_k^T\bm a_k=\bm0\). Its edge weights also give the
following important property that \(b_k\) must satisfy.

\begin{proposition}\label{lem:graph_distribution_positive}
Suppose the admissible nodal states lie in a convex set on which the entropy
is twice differentiable and strictly convex. For the connected graph in
\cref{eq:reference_distribution_graph},
\(b_k>0\) if \(\bm u_k\) is nonconstant, and \(b_k=0\) if
\(\bm u_k\) is constant.
\end{proposition}
\begin{proof}
Using symmetry of the edge weights and applying direct algebra, one can show
\begin{equation}\label{eq:b_k_expanded}
    b_k
    =
    \mu_k\sum_{m<n}\ell_{mn}
    (\bm w_{k,m}-\bm w_{k,n})^T
    (\bm u_{k,m}-\bm u_{k,n}).
\end{equation}
For \(\Delta\bm u_{mn}\coloneqq\bm u_{k,m}-\bm u_{k,n}\), the
entropy-variable difference satisfies
\[
    (\bm w_{k,m}-\bm w_{k,n})^T\Delta\bm u_{mn}
    =
    \int_0^1
    \Delta\bm u_{mn}^T
    \fn S''(\bm u_{k,n}+s\Delta\bm u_{mn})
    \Delta\bm u_{mn}\,{\rm d}s,
\]
where \(\fn{S}'' = \partial\fnb{W}/\partial{\fnb{U}}\). Strict convexity makes this expression positive when
\(\Delta\bm u_{mn}\ne\bm0\). Thus, if \(b_k=0\), the state is constant on every node; the converse is also true due to the structure of \cref{eq:b_k_expanded}. Hence, it follows that if \(\bm u_k\) is nonconstant then \(b_k>0\). 
\end{proof}

This removes exact cancellation of \(b_k\) for a nonconstant state.
Nearly constant states can still have arbitrarily small \(b_k\), so the
correction distribution vector used in the schemes is regularized as follows.

\begin{equation}\label{eq:alpha_regularized}
    \bm{\alpha}_{k,\epsilon}
    \coloneqq
    \frac{b_k\bm{a}_k}{b_k^2+\epsilon_{\rm mach}V_k^2}
    =
    \frac{(b_k/V_k)(\bm a_k/V_k)}
    {(b_k/V_k)^2+\epsilon_{\rm mach}},
    \qquad
    \epsilon_{\rm mach}>0,
\end{equation}
where \(\epsilon_{\rm mach}\) denotes machine precision. The second form
shows that the regularization is applied to the volume-normalized entropy
moment \(b_k/V_k\). Since \(b_k\) has the scaling of an element-integrated
quantity, it decreases with the element volume for smooth states; division by
\(V_k\) prevents this mesh-size dependence alone from making
\(\epsilon_{\rm mach}\) dominate the normalization. Since
\(\overline{\I}_k^T\bm a_k=\bm0\), the correction distribution vector keeps
the conservative moment exact; however, it produces an entropy moment defect,
\ie,
\begin{equation}\label{eq:alpha_moments}
    \overline{\I}_k^T\bm{\alpha}_{k,\epsilon}
    =
    \bm{0},
    \qquad
    \bm{w}_k^T\bm{\alpha}_{k,\epsilon}
    =
    1-\eta_k,
    \qquad
    \eta_k
    \coloneqq
    \frac{\epsilon_{\rm mach}V_k^2}
    {b_k^2+\epsilon_{\rm mach}V_k^2}.
\end{equation}
When \(b_k\ne0\), the unregularized vector \(\bm a_k/b_k\)
has zero conservative moment and unit entropy moment. The bounded
regularization factor approaches \(1/b_k\) when
\(b_k^2\gg\epsilon_{\rm mach}V_k^2\) and vanishes at \(b_k=0\).
Regularization preserves the conservative moment and supplies the
entropy moment \((1-\eta_k)\delta_k\), where \(1-\eta_k\in[0,1]\);
thus, \(\eta_k\) measures the fraction left uncorrected.
Its effect on the entropy balance is made explicit in
\cref{sec:hybrid_cdsbp}.

\begin{remark}
    Residual corrections may also be designed to enforce conservation, kinetic-energy preservation, or other moment constraints. Such additional corrections are not considered in this paper.
\end{remark}

\section{Hybrid continuous--discontinuous SBP residual correction}
\label{sec:hybrid_cdsbp}
\label{sec:csbp_correction}
\label{sec:dsbp_correction}

We now construct a residual-correction discretization in which strong C-SBP
assembly and weak D-SBP coupling coexist on the same conforming mesh. The
facet partition affects only the partial assembly map and the set of facets on
which SATs are evaluated. The element entropy defect and correction
distribution have no explicit dependence on the facet classification when
evaluated at a fixed broken state.

Let \(\bm f_k^{(i)}\), \(\bm g_k^{(i)}\), and \(\bm\psi_k^{(i)}\) collect the
nodal values of \(\fnb F^{(i)}\), \(\fn G^{(i)}\), and \(\psi^{(i)}\) on
\(\Omega_k\), respectively. For an element facet, define the physical normal
flux, entropy flux, and entropy potential by
\begin{align}
    \bm f_{\gamma k}^{n}
    &\coloneqq
    \sum_{i=1}^{d}
    \overline{\N}_{x_i\gamma k}
    \overline{\R}_{\gamma k}\bm f_k^{(i)},
    \quad
    \bm g_{\gamma k}^{n}
    \coloneqq
    \sum_{i=1}^{d}
    \N_{x_i\gamma k}\R_{\gamma k}\bm g_k^{(i)},
    \quad
    \bm\psi_{\gamma k}^{n}
    \coloneqq
    \sum_{i=1}^{d}
    \N_{x_i\gamma k}\R_{\gamma k}\bm\psi_k^{(i)}.
    \label{eq:cdsbp_normal_fluxes}
\end{align}
For a facet \(\gamma=\Gamma_k\cap\Gamma_v\in\Gamma_{\rm D}^I\), we set
\(\bm u_{\gamma s}\coloneqq\overline{\R}_{\gamma s}\bm u_s\),
\(\bm f_{\gamma s}^{(i)}\coloneqq
\overline{\R}_{\gamma s}\bm f_s^{(i)}\), and
\(\bm w_{\gamma s}\coloneqq\overline{\R}_{\gamma s}\bm w_s\).  A
conservative numerical flux, written using the outward normal of
\(\Omega_k\), satisfies
\begin{equation}\label{eq:cdsbp_conservative_flux}
    \bm f_{\gamma v}^{*}=-\bm f_{\gamma k}^{*}.
\end{equation}
The local Lax--Friedrichs (LLF) flux used in the numerical experiments is
\begin{equation}\label{eq:new_lf_flux}
    \bm f_{\gamma k}^{\mathrm{LF}}
    \coloneqq
    \frac{1}{2}
    \sum_{i=1}^{d}
    \overline{\N}_{x_i\gamma k}
    \left(\bm f_{\gamma k}^{(i)}+\bm f_{\gamma v}^{(i)}\right)
    +
    \frac{1}{2}\Lambda_{\gamma k}
    \left(\bm u_{\gamma k}-\bm u_{\gamma v}\right),
\end{equation}
where \(\Lambda_{\gamma k}=\Lambda_{\gamma v}\) is block diagonal.  With
\[
    \bm n_{\gamma k}^{(m)}
    \coloneqq
    \left(
        [\N_{x_1\gamma k}]_{mm},\ldots,
        [\N_{x_d\gamma k}]_{mm}
    \right)^T,
\]
its \(m\)-th diagonal block is
\[
    [\Lambda_{\gamma k}]_{mm}
    =
    \|\bm n_{\gamma k}^{(m)}\|_2
    \lambda_{\max,\gamma}^{(m)}\I_{n_c}.
\]
For the Euler equations, \(\lambda_{\max,\gamma}^{(m)}\) is the
Guermond--Popov non-iterative upper bound for the corresponding unit-normal
Riemann problem \cite{guermond2016fast}.  Unlike the less expensive Davis
estimate \cite{davis1988simplified}, this choice is guaranteed to bound the
maximum absolute wave speed for ideal gases with
\(1<\gamma_{\rm gas}\le5/3\). More generally, we assume that the selected numerical flux satisfies the facetwise Tadmor condition 
\begin{equation}\label{eq:cdsbp_entropy_condition}
    \varepsilon_*^\gamma
    \coloneqq
    (\bm w_{\gamma k}-\bm w_{\gamma v})^T
    \overline{\B}_\gamma\bm f_{\gamma k}^{*}
    -
    \bm1_\gamma^T\B_\gamma
    (\bm\psi_{\gamma k}^{n}+\bm\psi_{\gamma v}^{n})
    \ge0.
\end{equation}
For LLF, this condition follows when the wave-speed estimate is an upper
bound \cite{chen2017entropy}.

For a fixed partition \(\chi\), let \(\bm u_g^\chi\) and
\(\bm w_g^\chi\) denote the partially assembled conservative and entropy
variable vectors, respectively, and let
\(\bm s_g^\chi\in\mathbb R^{n_g^\chi}\) collect the corresponding nodal
entropy values, \([\bm s_g^\chi]_j\coloneqq
\fn S([\bm u_g^\chi]_j)\). We set
\(\bm u_k=\overline{\P}_k^\chi\bm u_g^\chi\) and
\(\bm w_k=\overline{\P}_k^\chi\bm w_g^\chi\). The hybrid
residual-correction discretization is
\begin{align}
    \der[\bm u_g^\chi]{t}
    & =
    (\overline{\H}_g^\chi)^{-1}
    \sum_{\Omega_k\in\mathcal T_h}
    (\overline{\P}_k^\chi)^T
    \widetilde{\bm r}_{k,\epsilon}^{\chi},
    \label{eq:cdsbp_scheme}\\
    \widetilde{\bm r}_{k,\epsilon}^{\chi}
    &\coloneqq
    \bm q_k^\chi+\bm c_{k,\epsilon}-\bm d_{k,\epsilon},
    \label{eq:cdsbp_corrected_residual}\\
    \bm q_k^\chi
    &\coloneqq
    -\sum_{i=1}^{d}\overline{\Q}_{x_i k}\bm f_k^{(i)}
    +
    \sum_{\gamma\in\Gamma_{{\rm D},k}^{\chi}}
    \overline{\R}_{\gamma k}^T\overline{\B}_\gamma
    (\bm f_{\gamma k}^{n}-\bm f_{\gamma k}^{*}),
    \label{eq:cdsbp_base_residual}\\
    \delta_k
    &\coloneqq
    -\sum_{\gamma\subset\Gamma_k}
    \bm1_\gamma^T\B_\gamma\bm g_{\gamma k}^{n}
    +
    \sum_{i=1}^{d}\bm w_k^T\overline{\Q}_{x_i k}\bm f_k^{(i)},
    \label{eq:cdsbp_entropy_defect}\\
    \sigma_k^{\rm ec}
    &\coloneqq1,
    \qquad
    \sigma_k^{\rm ed}
    \coloneqq(1-\operatorname{sgn}(\delta_k))/2,
    \qquad
    \sigma_k\in\{\sigma_k^{\rm ec},\sigma_k^{\rm ed}\},
    \label{eq:cdsbp_correction_mode}\\
    \bm c_{k,\epsilon}
    &\coloneqq
    \sigma_k\bm\alpha_{k,\epsilon}\delta_k,
    \label{eq:cdsbp_correction}\\
    \bm d_{k,\epsilon}
    &\coloneqq
    \bm\alpha_{k,\epsilon}\varepsilon_k,
    \qquad \varepsilon_k\ge0,
    \label{eq:cdsbp_dissipation}
\end{align}
where \(\bm q_k^\chi\), \(\bm c_{k,\epsilon}\), and
\(\bm d_{k,\epsilon}\) are, respectively, the mass-matrix-multiplied base
residual, correction, and dissipation vectors. The entropy defect in
\cref{eq:cdsbp_entropy_defect} and the vector \(\bm\alpha_{k,\epsilon}\) in
\cref{eq:alpha_regularized} are evaluated from element-local data and have no
explicit dependence on \(\chi\), although a D-to-C transfer may change their
values through the transferred state \(\bm u_k\). To establish conservation and entropy
stability, we use the following properties of the correction and dissipation
vectors:
\begin{align}
    \overline{\I}_k^T\bm c_{k,\epsilon}
    &=\bm0,
    &
    \overline{\I}_k^T\bm d_{k,\epsilon}
    &=\bm0,
    \label{eq:cdsbp_conservative_moments}\\
    \bm w_k^T\bm c_{k,\epsilon}
    &=\sigma_k(1-\eta_k)\delta_k,
    &
    \bm w_k^T\bm d_{k,\epsilon}
    &=(1-\eta_k)\varepsilon_k\ge0,
    \label{eq:cdsbp_entropy_moments}
\end{align}
which follow directly from \cref{eq:alpha_moments}.

\begin{remark}\label{rem:cdsbp_independent_ad}
For simplicity, the formulation and exact entropy identities below use
the same distribution vector for the correction and dissipation.
Alternatively, the dissipation distribution may be constructed by
replacing \(\hat{\mathsf L}\) with \(\hat{\M}\) in
\cref{eq:coefficient_ab} and evaluating its entropy moment and regularized
normalization in \cref{eq:alpha_regularized} independently. The correction
and its regularization factor \(\eta_k\) remain unchanged. This
derivative-based choice is the default in \cite{worku2026targeted}
and is used for AD in the numerical experiments. Both choices satisfy
\(\overline{\I}_k^T\bm d_{k,\epsilon}=\bm0\) and
\(\bm w_k^T\bm d_{k,\epsilon}\ge0\), so the conservation statements and
entropy upper bounds established below hold for either choice.
The two choices exhibit a problem-dependent trade-off between resolution
and robustness \cite{worku2026targeted}. Under the stated convexity,
admissibility, and graph-connectivity assumptions, the graph distribution
additionally ensures that its entropy moment vanishes only for constant states.
An optional safeguard combines the two choices by using the
derivative-based distribution unless
\((b_k^{\rm d})^2<C_b\epsilon_{\rm mach}V_k^2\), where
\(b_k^{\rm d}\) denotes its entropy moment and \(C_b>1\).
In that case, the graph-Laplacian distribution is used, with its
corresponding entropy moment and regularized normalization.
This safeguard is not used in the numerical experiments reported here.
\end{remark}

\begin{theorem}\label{thm:cdsbp_balance}
Consider the hybrid discretization in \cref{eq:cdsbp_scheme} on a periodic
conforming mesh. Assume that the numerical flux is conservative and satisfies
\cref{eq:cdsbp_entropy_condition} on every D-SBP facet. For either correction
mode, the regularized element residual satisfies the conservation identity
\begin{equation}\label{eq:cdsbp_element_conservation}
    \overline{\I}_k^T\widetilde{\bm r}_{k,\epsilon}^{\chi}
    =
    -\sum_{\gamma\in\Gamma_{{\rm C},k}^{\chi}}
    \overline{\I}_\gamma^T\overline{\B}_\gamma\bm f_{\gamma k}^{n}
    -\sum_{\gamma\in\Gamma_{{\rm D},k}^{\chi}}
    \overline{\I}_\gamma^T\overline{\B}_\gamma\bm f_{\gamma k}^{*},
\end{equation}
and the entropy identity
\begin{align}
    \bm w_k^T\widetilde{\bm r}_{k,\epsilon}^{\chi}
    ={}&
    -\sum_{\gamma\in\Gamma_{{\rm C},k}^{\chi}}
    \bm1_\gamma^T\B_\gamma\bm g_{\gamma k}^{n}
    +
    \sum_{\gamma\in\Gamma_{{\rm D},k}^{\chi}}
    \left(
        \bm1_\gamma^T\B_\gamma\bm\psi_{\gamma k}^{n}
        -\bm w_{\gamma k}^T\overline{\B}_\gamma
         \bm f_{\gamma k}^{*}
    \right)
    \nonumber\\
    &-
    [1-\sigma_k(1-\eta_k)]\delta_k
    -(1-\eta_k)\varepsilon_k.
    \label{eq:cdsbp_element_entropy}
\end{align}
Defining \(\overline{\I}_g^\chi\coloneqq\bm1_g^\chi\otimes\I_{n_c}\), the hybrid discretization is globally conservative,
\begin{equation}\label{eq:cdsbp_global_conservation}
    (\overline{\I}_g^\chi)^T\overline{\H}_g^\chi
    \der[\bm u_g^\chi]{t}=\bm0,
\end{equation}
and its entropy rate satisfies
\begin{align}
    (\bm1_g^\chi)^T\H_g^\chi
    \der[\bm s_g^\chi]{t}
    ={}&
    -\sum_{\gamma\in\Gamma_{\rm D}^I}\varepsilon_*^\gamma
    -\sum_{\Omega_k\in\mathcal T_h}
    [1-\sigma_k(1-\eta_k)]\delta_k
    -
    \sum_{\Omega_k\in\mathcal T_h}
    (1-\eta_k)\varepsilon_k
    \nonumber\\
    \le{}&
    -\!\!\sum_{\Omega_k:\,\delta_k<0}\eta_k\delta_k.
    \label{eq:cdsbp_global_entropy}
\end{align}
Thus, for every fixed partition, the method is exactly conservative and
entropy stable up to the same explicit regularization defect as its pure
C-SBP and D-SBP endpoints.
\end{theorem}

\begin{proof}
Since constants are differentiated exactly and \(\Q_{x_i k}+\Q_{x_i k}^T=\E_{x_i k}\), we have
\[
    \overline{\I}_k^T
    \left(-\sum_{i=1}^{d}\overline{\Q}_{x_i k}\bm f_k^{(i)}\right)
    =
    -\sum_{\gamma\subset\Gamma_k}
    \overline{\I}_\gamma^T\overline{\B}_\gamma
    \bm f_{\gamma k}^{n}.
\]
On each D-SBP facet, the physical-flux term in the SAT in
\cref{eq:cdsbp_base_residual} cancels the corresponding term in this sum,
leaving the physical flux on the C-SBP facets and the numerical flux on the
D-SBP facets. The zero conservative moments in \cref{eq:cdsbp_conservative_moments} then prove \cref{eq:cdsbp_element_conservation}. To analyze entropy stability, we use \cref{eq:cdsbp_base_residual,eq:cdsbp_entropy_defect} to write
\begin{align}
    \bm w_k^T\bm q_k^\chi
    ={}&
    -\delta_k
    -\sum_{\gamma\subset\Gamma_k}
    \bm1_\gamma^T\B_\gamma\bm g_{\gamma k}^{n}
    +\sum_{\gamma\in\Gamma_{{\rm D},k}^{\chi}}
    \bm w_{\gamma k}^T\overline{\B}_\gamma
    (\bm f_{\gamma k}^{n}-\bm f_{\gamma k}^{*})
    \nonumber\\
    ={}&
    -\delta_k
    -\sum_{\gamma\in\Gamma_{{\rm C},k}^{\chi}}
    \bm1_\gamma^T\B_\gamma\bm g_{\gamma k}^{n}
    +\sum_{\gamma\in\Gamma_{{\rm D},k}^{\chi}}
    \left(
        \bm1_\gamma^T\B_\gamma\bm\psi_{\gamma k}^{n}
        -\bm w_{\gamma k}^T\overline{\B}_\gamma
        \bm f_{\gamma k}^{*}
    \right),
    \label{eq:cdsbp_base_entropy_moment}
\end{align}
where the second equality follows from the nodal entropy-potential identity
\(\bm w_{\gamma k}^T\overline{\B}_\gamma\bm f_{\gamma k}^{n}
-\bm1_\gamma^T\B_\gamma\bm g_{\gamma k}^{n}
=\bm1_\gamma^T\B_\gamma\bm\psi_{\gamma k}^{n}\).
Adding the entropy moment of \(\bm c_{k,\epsilon}\) and subtracting that of
\(\bm d_{k,\epsilon}\) using \cref{eq:cdsbp_entropy_moments} proves
\cref{eq:cdsbp_element_entropy}.

To show global conservation, we use the definition of \(\overline{\I}_g^\chi\) and the constant-injection property in \cref{eq:cdsbp_partial_assembly_map} and write
\[
    (\overline{\I}_g^\chi)^T\overline{\H}_g^\chi
    \der[\bm u_g^\chi]{t}
    =
    \sum_{\Omega_k\in\mathcal T_h}\overline{\I}_k^T
    \widetilde{\bm r}_{k,\epsilon}^{\chi}.
\]
On a C-SBP facet, the physical flux is single valued and the two outward
normals are opposite. On a D-SBP facet,
\cref{eq:cdsbp_conservative_flux} gives the corresponding cancellation.
Combining these results proves the global conservation statement in
\cref{eq:cdsbp_global_conservation}. For the global entropy-stability analysis, we multiply \cref{eq:cdsbp_scheme} by \((\bm w_g^\chi)^T\overline{\H}_g^\chi\), use \(\bm w_k=\overline{\P}_k^\chi\bm w_g^\chi\), and apply the chain rule to the temporal term to find
\[
    (\bm w_g^\chi)^T\overline{\H}_g^\chi
    \der[\bm u_g^\chi]{t}
    =
    (\bm1_g^\chi)^T\H_g^\chi
    \der[\bm s_g^\chi]{t}
    =
    \sum_{\Omega_k\in\mathcal T_h}
    \bm w_k^T\widetilde{\bm r}_{k,\epsilon}^{\chi}.
\]
The physical entropy fluxes cancel pairwise on C-SBP facets. For a D-SBP facet
\(\gamma=\Gamma_k\cap\Gamma_v\), numerical-flux conservation gives
\(\bm f_{\gamma v}^*=-\bm f_{\gamma k}^*\). Adding the two element-side
contributions in \cref{eq:cdsbp_element_entropy} therefore gives
\begin{align*}
    \bm1_\gamma^T\B_\gamma
    (\bm\psi_{\gamma k}^{n}+\bm\psi_{\gamma v}^{n})
    -(\bm w_{\gamma k}-\bm w_{\gamma v})^T
    \overline{\B}_\gamma\bm f_{\gamma k}^{*}
    =
    -\varepsilon_*^\gamma\le0.
\end{align*}
Using \cref{eq:cdsbp_corrected_residual,eq:cdsbp_base_entropy_moment,eq:cdsbp_entropy_moments}
and the facet cancellations above, we obtain
\begin{align*}
    \sum_{\Omega_k\in\mathcal T_h}
    \bm w_k^T\widetilde{\bm r}_{k,\epsilon}^{\chi}
    ={}&
    \sum_{\Omega_k\in\mathcal T_h}
    \left(
        \bm w_k^T\bm q_k^\chi
        +\bm w_k^T\bm c_{k,\epsilon}
        -\bm w_k^T\bm d_{k,\epsilon}
    \right)\\
    ={}&
    -\sum_{\gamma\in\Gamma_{\rm D}^I}\varepsilon_*^\gamma
    -\sum_{\Omega_k\in\mathcal T_h}
    [1-\sigma_k(1-\eta_k)]\delta_k
    -\sum_{\Omega_k\in\mathcal T_h}
    (1-\eta_k)\varepsilon_k.
\end{align*}
This proves the equality in \cref{eq:cdsbp_global_entropy}. In the
entropy-conservative correction mode the coefficient of \(\delta_k\) is
\(\eta_k\). In the entropy-dissipative mode it is one for \(\delta_k>0\)
and \(\eta_k\) for \(\delta_k<0\). Dropping all nonpositive contributions
gives the stated bound in \cref{eq:cdsbp_global_entropy}.
\end{proof}

\begin{remark}
For an element with \(\delta_k<0\), the upper bound in \cref{eq:cdsbp_global_entropy} contains the positive term \(-\eta_k\delta_k\). By \cref{eq:alpha_moments}, \(\eta_k\) approaches one when \(b_k^2\ll\epsilon_{\rm mach}V_k^2\), so the estimate may allow nearly all of \(-\delta_k\) to remain. The same issue arises when denominators are regularized in other residual-correction methods \cite{abgrall2018general,chan2025artificial}. For an exactly constant element state, \(\delta_k=0\). A branch intended to avoid division by very small \(b_k\) must also handle nearly constant states, typically using a tolerance; that tolerance may turn off the correction for a nonconstant state and leave a nonzero entropy-producing defect uncorrected. Thus, both a regularized denominator and a tolerance-based branch can leave a positive entropy contribution. Guaranteeing entropy nonincrease requires sufficient dissipation, already present or added, to offset that contribution.
\end{remark}

\begin{corollary}\label{cor:cdsbp_endpoints}
The all-C and all-D classifications recover the pure endpoint schemes
exactly. 
\end{corollary}

\begin{proof}
If \(\Gamma_{\rm D}^I=\varnothing\), then
\(\bm q_k^\chi=-\sum_i\overline{\Q}_{x_i k}\bm f_k^{(i)}\) and
\cref{eq:cdsbp_scheme} is the pure C-SBP method. If
\(\Gamma_{\rm C}^I=\varnothing\), the assembled norm is the broken norm and
\begin{align*}
    \bm q_k^\chi
    &=
    -\sum_i\overline{\Q}_{x_i k}\bm f_k^{(i)}
    +\sum_{\gamma\subset\Gamma_k}
    \overline{\R}_{\gamma k}^T\overline{\B}_\gamma
    (\bm f_{\gamma k}^{n}-\bm f_{\gamma k}^{*})
    =
    \sum_i\overline{\Q}_{x_i k}^T\bm f_k^{(i)}
    -\sum_{\gamma\subset\Gamma_k}
    \overline{\R}_{\gamma k}^T\overline{\B}_\gamma
    \bm f_{\gamma k}^{*},
\end{align*}
which is the mass-matrix-multiplied D-SBP base residual. Since
\(\bm c_{k,\epsilon}\) and \(\bm d_{k,\epsilon}\) have no explicit dependence
on \(\chi\), the complete all-C and all-D discretizations recover the
corresponding pure residual-correction schemes.
\end{proof}

\section{Conservative and entropy-stable dynamic repartitioning}
\label{sec:dynamic_partition}

The C/D facet partition may change between accepted time steps, for example in
response to an element sensor that tracks solution features. A C-to-D change
removes a continuity constraint but does not alter the element-local nodal
values, since any newly independent copies inherit their common pre-change
value. A D-to-C change imposes a new continuity constraint and
therefore requires the matched copies to be assigned a common state. For a partition \(\chi\), define
\begin{equation}\label{eq:cdsbp_broken_operators}
    \H_b
    \coloneqq
    \operatorname{diag}(\H_1,\ldots,\H_{n_e}),
    \qquad
    \P^\chi
    \coloneqq
    \begin{bmatrix}
        \P_1^\chi\\[-0.2em]
        \vdots\\[-0.2em]
        \P_{n_e}^\chi
    \end{bmatrix}.
\end{equation}
The partially assembled norm satisfies
\begin{equation}\label{eq:cdsbp_broken_assembled_norm}
    \H_g^\chi
    =
    (\P^\chi)^T\H_b\P^\chi.
\end{equation}
For a time-dependent partition, let
\(\bm u_b(t)\coloneqq\overline{\P}^{\chi(t)}\bm u_g(t)\). At a
repartitioning time \(t_j\), we use the shorthand
\[
    \chi^\pm\coloneqq\chi(t_j^\pm),
    \qquad
    \bm u_g^\pm\coloneqq\bm u_g(t_j^\pm),
    \qquad
    \bm u_b^\pm\coloneqq\bm u_b(t_j^\pm),
\]
where the superscripts \(-\) and \(+\) denote the values immediately before
and after the transfer, respectively.

The transfer applies the same norm-weighted accumulate-and-invert operation
used in the C-SBP spatial discretization to \(\bm u_b^-\), but with the
assembly map of the new partition, \ie,
\begin{align}
    \bm u_g^+
    &\coloneqq
    (\overline{\H}_g^{\chi^+})^{-1}
    (\overline{\P}^{\chi^+})^T
    \overline{\H}_b\bm u_b^-,
    \label{eq:cdsbp_state_transfer}\\
    \bm u_b^+
    &=
    \Pi^{\chi^+}\bm u_b^-,
    \label{eq:cdsbp_projected_state}
\end{align}
where 
\begin{equation}\label{eq:cdsbp_broken_projection}
    \Pi^{\chi^+}
    \coloneqq
    \overline{\P}^{\chi^+}
    (\overline{\H}_g^{\chi^+})^{-1}
    (\overline{\P}^{\chi^+})^T
    \overline{\H}_b.
\end{equation}
The identities
\[
    (\Pi^{\chi^+})^2=\Pi^{\chi^+},
    \qquad
    \operatorname{range}(\Pi^{\chi^+})
    =\operatorname{range}(\overline{\P}^{\chi^+}),
    \qquad
    (\Pi^{\chi^+})^T\overline{\H}_b
    =\overline{\H}_b\Pi^{\chi^+}
\]
show that \(\Pi^{\chi^+}\) is an orthogonal projection with respect to
\(\overline{\H}_b\) onto
\(\operatorname{range}(\overline{\P}^{\chi^+})\), the space of broken states
satisfying the new C-SBP continuity constraints. It has the same algebraic
form as the direct-stiffness-summation projection analyzed by Kopera and
Giraldo \cite[Appendix~B]{kopera2015mass}.

To make the transfer more transparent, we introduce the nodal form of the
operation in \cref{eq:cdsbp_broken_projection}. For each new global solution
node \(j\), the entries \([\P_k^{\chi^+}]_{mj}=1\) select all local nodal
copies \((k,m)\) assigned to \(j\). Their assembled weight is
\begin{align}
    \omega_j^+
    &\coloneqq
    \sum_{\Omega_k\in\mathcal T_h}
    \sum_{m=1}^{n_p}
    [\P_k^{\chi^+}]_{mj}w_{k,m}
    =
    [\H_g^{\chi^+}]_{jj}.
    \label{eq:cdsbp_transferred_nodal_weight}
\end{align}
The transferred value is their mass-weighted average,
\begin{equation}\label{eq:cdsbp_transferred_nodal_state}
    [\bm u_g^+]_j
    =
    \frac{1}{\omega_j^+}
    \sum_{\Omega_k\in\mathcal T_h}
    \sum_{m=1}^{n_p}
    [\P_k^{\chi^+}]_{mj}
    w_{k,m}[\bm u_b^-]_{k,m}.
\end{equation}
These matrices in \cref{eq:cdsbp_state_transfer,eq:cdsbp_projected_state} are used only for analysis and need not be formed explicitly.
In the implementation, \(\P_k^{\chi^+}\) is represented by the new
local-to-global map, and the sums above are evaluated directly over the copies
sharing each global index. A C-to-D change removes the corresponding
connection when this map is constructed; any copies thereby separated receive
different indices but retain their common pre-transfer value. All D-to-C
connections are added before the sums are evaluated, so matched copies and any
copies connected to them through existing or simultaneously added C-SBP
facets receive the same index. A constraint between copies that are already
connected is redundant, and copies not belonging to an affected index remain
unchanged.

To assess the entropy change caused by repartitioning, we compare the discrete
entropy immediately before and after each transfer. For a broken state
\(\bm u_b\), we define the total entropy by
\begin{equation}\label{eq:cdsbp_discrete_entropy}
    S_h(\bm u_b)
    \coloneqq
    \sum_{\Omega_k\in\mathcal T_h}
    \sum_{m=1}^{n_p}
    w_{k,m}
    \fn S([\bm u_b]_{k,m}).
\end{equation}
When \(\bm u_b=\overline{\P}^\chi\bm u_g^\chi\), this definition is
equivalent to
\begin{equation}\label{eq:cdsbp_entropy_representations}
    S_h(\bm u_b)
    =
    (\bm1_g^\chi)^T\H_g^\chi\bm s_g^\chi.
\end{equation}

\begin{theorem}\label{thm:cdsbp_state_transfer}
Let \(\chi^-\) and \(\chi^+\) be any two C/D partitions of the same
conforming mesh. The transfer in
\cref{eq:cdsbp_state_transfer} satisfies the following properties.

\begin{enumerate}
\item[(i)] \emph{Conservation.} For each new global index \(j\), the associated
conservative moment is preserved:
\begin{equation}\label{eq:cdsbp_transfer_index_conservation}
    \omega_j^+[\bm u_g^+]_j
    =
    \sum_{\Omega_k\in\mathcal T_h}
    \sum_{m=1}^{n_p}
    [\P_k^{\chi^+}]_{mj}w_{k,m}[\bm u_b^-]_{k,m},
    \qquad
    j=1,\ldots,n_g^{\chi^+}.
\end{equation}
Summing over all new global indices shows that every globally conserved
component is preserved:
\begin{equation}\label{eq:cdsbp_transfer_conservation}
    \sum_{\Omega_k\in\mathcal T_h}
    \sum_{m=1}^{n_p}
    w_{k,m}[\bm u_b^+]_{k,m}
    =
    \sum_{\Omega_k\in\mathcal T_h}
    \sum_{m=1}^{n_p}
    w_{k,m}[\bm u_b^-]_{k,m}.
\end{equation}

\item[(ii)] \emph{Entropy stability.} The transfer does not increase the
discrete entropy:
\begin{equation}\label{eq:cdsbp_transfer_entropy}
    S_h(\bm u_b^+)\le S_h(\bm u_b^-).
\end{equation}

\item[(iii)] \emph{Admissibility preservation.} If the admissible set is
convex, admissible pre-transfer nodal states remain admissible after transfer.

\item[(iv)] \emph{Error non-amplification.} For every reference state
\(\bm z_b\in\operatorname{range}(\overline{\P}^{\chi^+})\), the transfer
leaves \(\bm z_b\) unchanged and does not increase the error measured in the
\(\overline{\H}_b\)-norm:
\begin{equation}\label{eq:cdsbp_transfer_error_nonamplification}
    \Pi^{\chi^+}\bm z_b=\bm z_b,
    \quad
    \|\bm u_b^+-\bm z_b\|_{\overline{\H}_b}
    \le
    \|\bm u_b^--\bm z_b\|_{\overline{\H}_b},
    \quad
    \|\bm v\|_{\overline{\H}_b}^2
    \coloneqq\bm v^T\overline{\H}_b\bm v.
\end{equation}
In particular, a pure C-to-D change leaves the element-local nodal values
unchanged: \(\bm u_b^+=\bm u_b^-\).
\end{enumerate}
\end{theorem}

\begin{proof}
For each new global index \(j\), the coefficients
\(
    [\P_k^{\chi^+}]_{mj}w_{k,m}/\omega_j^+
\)
are nonnegative, and by \cref{eq:cdsbp_transferred_nodal_weight}, we have
\[
    \sum_{\Omega_k\in\mathcal T_h}
    \sum_{m=1}^{n_p}
    \frac{[\P_k^{\chi^+}]_{mj}w_{k,m}}{\omega_j^+}
    =
    \frac{\omega_j^+}{\omega_j^+}
    =1.
\]
Thus, \cref{eq:cdsbp_transferred_nodal_state} expresses
\([\bm u_g^+]_j\) as a convex combination of the pre-transfer nodal states.
Multiplying \cref{eq:cdsbp_transferred_nodal_state} by \(\omega_j^+\) proves
\cref{eq:cdsbp_transfer_index_conservation}. Summing this identity over
\(j=1,\ldots,n_g^{\chi^+}\) gives
\begin{align*}
    \sum_{j=1}^{n_g^{\chi^+}}
    \omega_j^+[\bm u_g^+]_j
    &=
    \sum_{j=1}^{n_g^{\chi^+}}
    \sum_{\Omega_k\in\mathcal T_h}
    \sum_{m=1}^{n_p}
    [\P_k^{\chi^+}]_{mj}
    w_{k,m}[\bm u_b^-]_{k,m}\\
    &=
    \sum_{\Omega_k\in\mathcal T_h}
    \sum_{m=1}^{n_p}
    w_{k,m}[\bm u_b^-]_{k,m}
    \underbrace{\sum_{j=1}^{n_g^{\chi^+}}
    [\P_k^{\chi^+}]_{mj}}_{=1}\\
    &=
    \sum_{\Omega_k\in\mathcal T_h}
    \sum_{m=1}^{n_p}
    w_{k,m}[\bm u_b^-]_{k,m}.
\end{align*}
The left-hand side is the post-transfer conservative quantity, which, using
\(\bm u_b^+=\overline{\P}^{\chi^+}\bm u_g^+\) and
\cref{eq:cdsbp_transferred_nodal_weight}, can be written as
\begin{align*}
    \sum_{\Omega_k\in\mathcal T_h}
    \sum_{m=1}^{n_p}
    w_{k,m}[\bm u_b^+]_{k,m}
    &=
    \sum_{\Omega_k\in\mathcal T_h}
    \sum_{m=1}^{n_p}
    w_{k,m}
    \sum_{j=1}^{n_g^{\chi^+}}
    [\P_k^{\chi^+}]_{mj}[\bm u_g^+]_j\\
    &=
    \sum_{j=1}^{n_g^{\chi^+}}
    \left(
        \sum_{\Omega_k\in\mathcal T_h}
        \sum_{m=1}^{n_p}
        [\P_k^{\chi^+}]_{mj}w_{k,m}
    \right)[\bm u_g^+]_j\\
    &=
    \sum_{j=1}^{n_g^{\chi^+}}
    \omega_j^+[\bm u_g^+]_j.
\end{align*}
Combining the last two identities proves
\cref{eq:cdsbp_transfer_conservation}.

To prove \((ii)\), we apply the convexity of \(\fn S\) to the convex
combination in \cref{eq:cdsbp_transferred_nodal_state} and use Jensen's inequality to write
\begin{align*}
    \fn S([\bm u_g^+]_j)
    &=
    \fn S\!\left(
        \sum_{\Omega_k\in\mathcal T_h}
        \sum_{m=1}^{n_p}
        \frac{[\P_k^{\chi^+}]_{mj}w_{k,m}}{\omega_j^+}
        [\bm u_b^-]_{k,m}
    \right)
    \le
    \sum_{\Omega_k\in\mathcal T_h}
    \sum_{m=1}^{n_p}
    \frac{[\P_k^{\chi^+}]_{mj}w_{k,m}}{\omega_j^+}
    \fn S([\bm u_b^-]_{k,m}).
\end{align*}
Multiplying by \(\omega_j^+\) and summing over
\(j=1,\ldots,n_g^{\chi^+}\) therefore gives
\begin{align*}
    \sum_{j=1}^{n_g^{\chi^+}}
    \omega_j^+\fn S([\bm u_g^+]_j)
    &\le
    \sum_{j=1}^{n_g^{\chi^+}}
    \sum_{\Omega_k\in\mathcal T_h}
    \sum_{m=1}^{n_p}
    [\P_k^{\chi^+}]_{mj}w_{k,m}
    \fn S([\bm u_b^-]_{k,m})\\
    &=
    \sum_{\Omega_k\in\mathcal T_h}
    \sum_{m=1}^{n_p}
    w_{k,m}\fn S([\bm u_b^-]_{k,m})
    \sum_{j=1}^{n_g^{\chi^+}}
    [\P_k^{\chi^+}]_{mj}\\
    &=S_h(\bm u_b^-).
\end{align*}
The left-hand side is the post-transfer discrete entropy. Since
\(\bm u_b^+=\overline{\P}^{\chi^+}\bm u_g^+\), each post-transfer local
copy takes the value assigned to its new global index, and hence
\begin{align*}
    S_h(\bm u_b^+)
    &=
    \sum_{\Omega_k\in\mathcal T_h}
    \sum_{m=1}^{n_p}
    w_{k,m}\fn S([\bm u_b^+]_{k,m})
    =
    \sum_{\Omega_k\in\mathcal T_h}
    \sum_{m=1}^{n_p}
    w_{k,m}
    \sum_{j=1}^{n_g^{\chi^+}}
    [\P_k^{\chi^+}]_{mj}\fn S([\bm u_g^+]_j)\\
    &=
    \sum_{j=1}^{n_g^{\chi^+}}
    \left(
        \sum_{\Omega_k\in\mathcal T_h}
        \sum_{m=1}^{n_p}
        [\P_k^{\chi^+}]_{mj}w_{k,m}
    \right)\fn S([\bm u_g^+]_j)\\
    &=
    \sum_{j=1}^{n_g^{\chi^+}}
    \omega_j^+\fn S([\bm u_g^+]_j).
\end{align*}
Combining the last two identities proves
\cref{eq:cdsbp_transfer_entropy}. If all pre-transfer states
belong to a convex admissible set, each convex combination
\([\bm u_g^+]_j\) belongs to the same set. Scattering these values to the
element-local copies therefore proves the admissibility preservation statement in \((iii)\).

Finally, since \(\Pi^{\chi^+}\) is an orthogonal projection with respect to
\(\overline{\H}_b\), we have \(\Pi^{\chi^+}\bm z_b=\bm z_b\) for all \(\bm z_b \in \operatorname{range}(\overline{\P}^{\chi^{+}})\). Moreover, since \(\bm u_{b}^+ = \Pi^{\chi^+} \bm u_b^{-}\), \(\bm u_b^{-} - \bm u_b^{+}\) is orthogonal to this column space, while \(\bm u_b^+-\bm z_b\) lies in it. Therefore, we have
\begin{align*}
    \|\bm u_b^--\bm z_b\|_{\overline{\H}_b}^2
    &=
    \|\bm u_b^+-\bm z_b\|_{\overline{\H}_b}^2
    +
    \|\bm u_b^--\bm u_b^+\|_{\overline{\H}_b}^2
    \ge
    \|\bm u_b^+-\bm z_b\|_{\overline{\H}_b}^2,
\end{align*}
which proves \cref{eq:cdsbp_transfer_error_nonamplification}. A pure C-to-D
change gives \(\bm u_b^+=\Pi^{\chi^+}\bm u_b^-=\bm u_b^-\). Consequently,
the error non-amplification estimate holds with equality.
\end{proof}

\begin{remark}\label{rem:cdsbp_switch_accuracy}
At a repartitioning time, let \(\bm z_b\) be the nodal
sampling of a smooth exact solution. Because the exact solution has matching
traces across every C-SBP facet in the new partition,
\(\bm z_b\in\operatorname{range}(\overline{\P}^{\chi^+})\). Hence,
\cref{eq:cdsbp_transfer_error_nonamplification} gives
\[
    \Pi^{\chi^+}\bm z_b=\bm z_b,
    \qquad
    \|\bm u_b^+-\bm z_b\|_{\overline{\H}_b}
    \le
    \|\bm u_b^--\bm z_b\|_{\overline{\H}_b}.
\]
Thus, a single transfer preserves the sampled exact solution and does not
increase the spatial error.
\end{remark}

\begin{remark}\label{rem:cdsbp_switch_locality}
Consider a D-to-C transfer that replaces two nodal states
\(\bm u_{\rm L}^-\) and \(\bm u_{\rm R}^-\), with weights
\(w_{\rm L},w_{\rm R}>0\), by the common state
\[
    \overline{\bm u}
    =
    \frac{w_{\rm L}\bm u_{\rm L}^-+w_{\rm R}\bm u_{\rm R}^-}
         {w_{\rm L}+w_{\rm R}}.
\]
The corresponding changes in the element conservative moments are
\[
    \Delta_{\rm L}
    \coloneqq
    w_{\rm L}(\overline{\bm u}-\bm u_{\rm L}^-)
    =
    \frac{w_{\rm L}w_{\rm R}}{w_{\rm L}+w_{\rm R}}
    (\bm u_{\rm R}^--\bm u_{\rm L}^-),
    \qquad
    \Delta_{\rm R}
    \coloneqq
    w_{\rm R}(\overline{\bm u}-\bm u_{\rm R}^-)
    =-\Delta_{\rm L}.
\]
Thus, the transfer produces an instantaneous conservative exchange between
the two elements: their individual moments need not remain unchanged, but the
equal-and-opposite changes give an exact local balance and preserve the total
quantity. For a merge of more than two copies, the corresponding changes
likewise sum to zero. The entropy inequality applies to the combined weighted
entropy of the merged copies and need not hold for each element separately.
After the transfer, the fixed-partition residual identities in
\cref{thm:cdsbp_balance} apply under the new partition.
\end{remark}

\begin{corollary}\label{cor:cdsbp_dynamic_balance}
Let \(0<t_1<\cdots<t_{N_s}<T\) be the only repartitioning times. Suppose that
\cref{eq:cdsbp_scheme} is evolved with a fixed C/D partition between
successive repartitioning times, that \cref{eq:cdsbp_state_transfer} is
applied at each \(t_j\), and that all nodal states remain admissible throughout
\([0,T]\). Define the entropy change at the \(j\)-th transfer by
\begin{equation}\label{eq:cdsbp_transfer_entropy_change}
    \Delta S_{{\rm tr},j}
    \coloneqq
    S_h(\bm u_b^+)
    -S_h(\bm u_b^-).
\end{equation}
Defining \( \Delta S_h^{0,T}\coloneqq S_h(\bm u_b(T))-S_h(\bm u_b(0))\), the piecewise semidiscrete entropy balance satisfies
\begin{align}
    \Delta S_h^{0,T}
    &=
    \sum_{j=1}^{N_s}\Delta S_{{\rm tr},j}
    -\int_0^T
    \sum_{\gamma\in\Gamma_{\rm D}^I(t)}
    \varepsilon_*^\gamma(t)\,\dd t
    \nonumber\\
    &\quad -\int_0^T\sum_{\Omega_k\in\mathcal T_h}
    \left(
        [1-\sigma_k(1-\eta_k)]\delta_k
        +(1-\eta_k)\varepsilon_k
    \right)\dd t
    \label{eq:cdsbp_dynamic_entropy_equality}
    \\
    & \le
    \sum_{j=1}^{N_s}\Delta S_{{\rm tr},j}
    -\int_0^T\sum_{\Omega_k:\,\delta_k<0}
      \eta_k\delta_k\,\dd t
    \le
    -\int_0^T\sum_{\Omega_k:\,\delta_k<0}
      \eta_k\delta_k\,\dd t.
    \label{eq:cdsbp_dynamic_entropy_inequality}
\end{align}
Thus, repartitioning adds no positive entropy contribution. 
\end{corollary}

\begin{proof}
Set \(t_0=0\) and \(t_{N_s+1}=T\), and let \(\chi_j\) denote the fixed
partition on \([t_j,t_{j+1})\), with the final interval understood to include
\(T\). Using \cref{eq:cdsbp_entropy_representations}, multiplying
\cref{eq:cdsbp_scheme} by
\((\bm w_g^{\chi_j})^T\overline{\H}_g^{\chi_j}\) and applying the chain rule
in time gives the fixed-partition entropy identity in
\cref{eq:cdsbp_global_entropy} for \(t\in(t_j,t_{j+1})\):
\begin{align*}
    \der[S_h(\bm u_b(t))]{t}
    ={}&
    -\sum_{\gamma\in\Gamma_{\rm D}^I(t)}
    \varepsilon_*^\gamma(t)
    -\sum_{\Omega_k\in\mathcal T_h}
    \left(
        [1-\sigma_k(1-\eta_k)]\delta_k
        +(1-\eta_k)\varepsilon_k
    \right).
\end{align*}
At each repartitioning time \(t_j\),
\cref{thm:cdsbp_state_transfer}, specifically
\cref{eq:cdsbp_transfer_entropy}, ensures that the entropy jump satisfies
\[
    \Delta S_{{\rm tr},j}
    =
    S_h(\bm u_b(t_j^+))-S_h(\bm u_b(t_j^-))
    \le0,
\]
where \(-\) and \(+\) denote states immediately before and after the
transfer, respectively. Adding the entropy changes over the
fixed-partition intervals and the instantaneous transfer jumps gives
\[
    \Delta S_h^{0,T}
    =
    \sum_{j=0}^{N_s}\int_{t_j}^{t_{j+1}}
    \der[S_h(\bm u_b(t))]{t}\,\dd t
    +
    \sum_{j=1}^{N_s}\Delta S_{{\rm tr},j},
\]
where the time derivative is taken within each fixed-partition interval
and excludes the instantaneous jumps. Substituting the equality in \cref{eq:cdsbp_global_entropy} into these integrals and combining them over \([0,T]\) proves \cref{eq:cdsbp_dynamic_entropy_equality}; using the upper bound in the same equation and
\(\Delta S_{{\rm tr},j}\le0\) proves \cref{eq:cdsbp_dynamic_entropy_inequality}.
\end{proof}

\section{Consistency}\label{sec:consistency_residual_correction}
The purpose of this section is to establish the consistency of the algebraic
correction for an arbitrary fixed C/D partition. We first determine how the
correction distribution vector scales with mesh refinement. We then combine
this result with an estimate of the entropy defect to bound the correction
contribution after assembly and application of the inverse norm.

\subsection{Scaling of the correction distribution vector}
The regularized distribution vector \(\bm\alpha_{k,\epsilon}\) in
\cref{eq:alpha_regularized} depends on the vector \(\bm a_k\)
and its entropy moment \(b_k=\bm w_k^T\bm a_k\). The next two results first bound these quantities
for a smooth state and then establish the resulting mesh scaling of
\(\bm\alpha_{k,\epsilon}\).

\begin{lemma}\label{lem:correction_terms_scaling}
    Let \(\bm u_k\) and \(\bm w_k\) be the nodal restrictions of a smooth
    admissible conservative state and its entropy variables. Then, there are
    constants \(C_a,C_b>0\), independent of \(k\) and \(h\), such that
    \[
        \|\bm a_k\|_2
        \le
        C_a h^{d-1},
        \qquad
        |b_k|
        \le
        C_b h^d.
    \]
    Consequently, \(\|\bm a_k\|_2=\fn{O}(h^{d-1})\), \(|b_k|=\fn{O}(h^d)\), and \({b_k}/{V_k}=\fn{O}(1)\). 
\end{lemma}

\begin{proof}
    Choose a point in \(\Omega_k\), and let \(\bm u_{0,k}\) and
    \(\bm w_{0,k}\) be the corresponding values of the smooth
    conservative and entropy-variable fields. Smoothness and
    \cref{eq:local_global_mesh_size} give
    \[
        \|\bm u_k-\overline{\I}_k\bm u_{0,k}\|_2
        =
        \fn{O}(h),
        \qquad
        \|\bm w_k-\overline{\I}_k\bm w_{0,k}\|_2
        =
        \fn{O}(h).
    \]
    Since \(\hat{\mathsf L}\bm1_k=\bm0\), the upper bound for
    \(\mu_k\) in \cref{eq:mu_scaling} and the fixed reference matrix
    \(\hat{\mathsf L}\) give
    \[
        \bm a_k
        =
        \mu_k(\hat{\mathsf L}\otimes\I_{n_c})
        (\bm u_k-\overline{\I}_k\bm u_{0,k})
        =
        \fn{O}(h^{d-1}).
    \]
    Furthermore, \(\overline{\I}_k^T\bm a_k=\bm0\), and hence
    \[
        b_k
        =
        (\bm w_k-\overline{\I}_k\bm w_{0,k})^T\bm a_k
        =
        \fn{O}(h^d).
    \]
    Finally, \(V_k\ge n_pc_Hh^d\) by
    \cref{eq:quasi_uniformity,eq:V_k}, so \(b_k/V_k=\fn{O}(1)\).
\end{proof}

\begin{theorem}\label{thm:reference_normalization}
    Under the assumptions of \(\cref{lem:correction_terms_scaling}\), for fixed \(\epsilon_{\rm mach}>0\), independent of \(k\) and \(h\), we
    have
    \[
        \|\bm\alpha_{k,\epsilon}\|_2
        \le
        \frac{C_a}
        {2n_pc_H\sqrt{\epsilon_{\rm mach}}}h^{-1},
    \]
    and hence \(\bm\alpha_{k,\epsilon}=\fn{O}(h^{-1})\).
\end{theorem}

\begin{proof}
    If \(b_k=0\), then \(\bm\alpha_{k,\epsilon}=\bm0\), and the result
    follows immediately. If \(b_k\ne0\), the arithmetic--geometric mean
    inequality gives
    \[
        b_k^2+\epsilon_{\rm mach}V_k^2
        \ge
        2|b_k|V_k\sqrt{\epsilon_{\rm mach}}.
    \]
    Therefore, using \cref{lem:correction_terms_scaling} and
    \(V_k\ge n_pc_Hh^d\), we obtain
    \[
        \begin{aligned}
        \|\bm\alpha_{k,\epsilon}\|_2
        &=
        \frac{|b_k|}{b_k^2+\epsilon_{\rm mach}V_k^2}
        \|\bm a_k\|_2
        \le
        \frac{\|\bm a_k\|_2}
        {2V_k\sqrt{\epsilon_{\rm mach}}}
        \le
        \frac{C_a}
        {2n_pc_H\sqrt{\epsilon_{\rm mach}}}h^{-1}.
        \end{aligned} \qedhere
    \]
\end{proof}

\subsection{Consistency of the corrected discretizations}
We next combine the distribution bound with a high-order estimate
of the entropy defect to establish consistency of the corrected
discretization. The optional
dissipation vector \(\bm d_{k,\epsilon}\) is a separate modeling term and is
not included in the consistency analysis below. The targeted
artificial-dissipation operator used in the numerical experiments is
design-order accurate for smooth solutions under the parameter scaling in \cite{worku2026targeted}.

\begin{lemma}\label{lem:entropy_defect_scaling}
    Consider a shape-regular, quasi-uniform family of affine elements and
    degree-\(p\), diagonal-\(\E\) SBP operators, with \(p\ge1\), whose volume
    and facet quadrature rules are exact for polynomials of total degrees at
    least \(2p-1\) and \(2p\), respectively. Let \(\bm u_k\) be the nodal
    restriction of a smooth admissible solution, and assume that the
    corresponding entropy variables, physical fluxes, and entropy fluxes have
    bounded derivatives through order \(2p+1\) on the computational domain.
    Then, the entropy defect in \cref{eq:cdsbp_entropy_defect} satisfies
    \begin{equation}\label{eq:entropy_defect_superconvergence}
        \delta_k
        =
        \fn{O}\!\left(h^{2p+d}\right).
    \end{equation}
\end{lemma}

\begin{proof}
    Chan \cite[Sec.~4.4, Eq.~(27)]{chan2025artificial} obtains
    \(\fn{O}(h^{2p+2+d})\) under exact integration for a related DG volume
    entropy residual. Using the SBP properties and the stated quadrature
    exactness, we obtain the discrete estimate \(\fn{O}(h^{2p+d})\); its proof is provided in
    \ref{app:entropy_defect_scaling}.
\end{proof}

The estimate in \cref{lem:entropy_defect_scaling} is an elementwise
functional superconvergence result. Indeed, \(\delta_k\) is the difference
between discrete entropy-production and facet entropy-flux functionals that
approximate the same exact element entropy balance. Since a quasi-uniform mesh
contains \(\fn{O}(h^{-d})\) elements, the local rate
\(\fn{O}(h^{2p+d})\) corresponds to the \(\fn{O}(h^{2p})\) global rate
expected for superconvergent functionals of degree-\(p\) SBP discretizations
\cite{hicken2011superconvergent,worku2021simultaneous,worku2023entropy}.

\begin{theorem}\label{thm:cdsbp_consistency}
On a shape-regular, quasi-uniform family of affine meshes, consider the hybrid
discretization in \cref{eq:cdsbp_scheme} with degree-\(p\), diagonal-\(\E\) SBP
operators, any fixed C/D facet partition, and the LLF flux in
\cref{eq:new_lf_flux} on D-SBP facets. Let the exact solution be smooth and
admissible. Then the base operator after partial assembly is degree-\(p\)
consistent at every global solution node. Moreover, the assembled correction
term satisfies, componentwise,
\begin{equation}\label{eq:cdsbp_correction_scaling}
    (\overline{\H}_g^\chi)^{-1}
    \sum_{\Omega_k\in\mathcal T_h}
    (\overline{\P}_k^\chi)^T\bm c_{k,\epsilon}
    =\fn{O}(h^{2p-1}).
\end{equation}
Consequently, the assembled regularized correction term does not reduce the
\(\fn{O}(h^p)\) consistency order of the base discretization.
\end{theorem}

\begin{proof}
The volume term in the element base residual \(\bm q_k^\chi\) is
\[
    -\sum_{i=1}^{d}\overline{\Q}_{x_i k}\bm f_k^{(i)}
    =
    -\overline{\H}_k
    \sum_{i=1}^{d}\overline{\D}_{x_i k}\bm f_k^{(i)},
\]
whose pointwise consistency follows from \cref{eq:D_accuracy}. On a D-SBP
facet, the diagonal-\(\E\) restriction extracts the values at matching facet
nodes. For the nodal sampling of the exact solution,
\[
    \bm u_{\gamma k}=\bm u_{\gamma v},
    \qquad
    \bm f_{\gamma k}^{*}=\bm f_{\gamma k}^{n},
\]
so the D-SBP SAT vanishes. For an SBP extrapolation operator to the facet
nodes (not shared with the volume nodes) with componentwise trace error \(\fn{O}(h^{p+1})\), the two traces would instead differ by \(\fn{O}(h^{p+1})\); the resulting flux difference is
\(\fn{O}(h^{p+d})\) and contributes \(\fn{O}(h^p)\) after application of the
inverse norm. On a conforming, shape-regular mesh, the number of
element-local nodal copies assigned to any global solution node is bounded
independently of \(h\). Partial assembly therefore forms an \(\H\)-weighted
average of a bounded number of degree-\(p\) approximations and retains the
same order. Applying the entropy-defect estimate in \cref{lem:entropy_defect_scaling} and using \cref{thm:reference_normalization}, we obtain
\[
    \bm c_{k,\epsilon}
    =\sigma_k\bm\alpha_{k,\epsilon}\delta_k
    =\fn{O}(h^{2p+d-1}).
\]
Each diagonal entry of \(\H_g^\chi\) scales as \(h^d\), and hence
\((\overline{\H}_g^\chi)^{-1}=\fn{O}(h^{-d})\). Partial assembly sums only a
bounded number of element-local corrections and therefore preserves the
\(\fn{O}(h^{2p+d-1})\) scaling of \(\bm c_{k,\epsilon}\). Thus,
\cref{eq:cdsbp_correction_scaling} follows, and since
\(2p-1\ge p\) for \(p\ge1\), the correction retains the design order.
\end{proof}

\section{Relation to existing entropy-correction constructions}
\label{sec:relation_existing_corrections}

The residual-correction construction of AOR
\cite{abgrall2022reinterpretation}, the ECAV method of Chan
\cite{chan2025artificial}, and the present method can be compared by
separately analyzing the scalar entropy defect and distribution vector. In particular, in the diagonal-\(\E\) SBP setting considered in this work, and with the compatible AOR facet coupling described below, all three constructions use the same element-local scalar defect, as shown below in \cref{eq:shared_correction_defect}. Their correction mechanisms differ in how the defect is distributed, activated, and regularized.

\subsection{Element-local form of the compatible AOR correction}

Let \(\bm q_k^{\rm D}\) denote the all-D specialization of
\(\bm q_k^\chi\) in \cref{eq:cdsbp_base_residual}. A SAT-dependent correction may instead be formed from the defect as
\begin{equation}\label{eq:sat_dependent_defect}
    \delta_k^{\rm SAT}
    \coloneqq
    -\bm w_k^T\bm q_k^{\rm D}
    -\sum_{\gamma\subset\Gamma_k}
    \bm1_\gamma^T\B_\gamma\bm g_{\gamma k}^{\star},
\end{equation}
where \(\bm g_{\gamma k}^{\star}\) is a consistent numerical entropy flux,
written using the outward normal of \(\Omega_k\). Applying
\cref{eq:entropy_potential} at the facet nodes and using the normal-flux
definitions in \cref{eq:cdsbp_normal_fluxes} we can write
\begin{equation}\label{eq:entropy_potential_identity}
    \bm w_{\gamma k}^T\overline{\B}_\gamma
    \bm f_{\gamma k}^{n}
    =
    \bm1_\gamma^T\B_\gamma
    \left(\bm g_{\gamma k}^{n}+\bm\psi_{\gamma k}^{n}\right).
\end{equation}
Expanding \(-\bm w_k^T\bm q_k^{\rm D}\) in
\cref{eq:sat_dependent_defect} and using the element-local defect in
\cref{eq:cdsbp_entropy_defect} gives
\begin{align}
    \delta_k^{\rm SAT}
    &=
    \delta_k
    +
    \sum_{\gamma\subset\Gamma_k}\mathcal E_{\gamma k}^{\star},
    \label{eq:sat_defect_decomposition}
    \\
    \mathcal E_{\gamma k}^{\star}
    &\coloneqq
    \bm w_{\gamma k}^T\overline{\B}_\gamma\bm f_{\gamma k}^{*}
    -\bm1_\gamma^T\B_\gamma\bm g_{\gamma k}^{\star}
    -\bm1_\gamma^T\B_\gamma\bm\psi_{\gamma k}^{n}. \label{eq:sat_defect_decomposition_leftover}
\end{align}
Thus, all explicit dependence of \(\delta_k^{\rm SAT}\) on the
neighboring trace is contained in \(\mathcal E_{\gamma k}^{\star}\). Now, suppose that \(\bm f_{\gamma k}^{*}\) is a two-point
entropy-conservative facet flux, so equality holds in
\cref{eq:cdsbp_entropy_condition}. Following AOR
\cite[Eq.~(30)]{abgrall2022reinterpretation}, choose the compatible numerical
entropy flux satisfying
\begin{equation}\label{eq:sat_ec_entropy_flux}
    \bm1_\gamma^T\B_\gamma\bm g_{\gamma k}^{\star,\rm EC}
    \coloneqq
    \frac12(\bm w_{\gamma k}+\bm w_{\gamma v})^T
    \overline{\B}_\gamma\bm f_{\gamma k}^{*}
    -\frac12\bm1_\gamma^T\B_\gamma
    (\bm\psi_{\gamma k}^{n}-\bm\psi_{\gamma v}^{n}),
\end{equation}
where the sign before \(\bm\psi_{\gamma v}^{n}\) accounts for its definition
with the outward normal of \(\Omega_v\). For this compatible flux pair, define
\begin{equation}\label{eq:aor_entropy_defect}
    \delta_k^{\rm AOR}
    \coloneqq
    -\bm w_k^T\bm q_k^{\rm D}
    -\sum_{\gamma\subset\Gamma_k}
    \bm1_\gamma^T\B_\gamma
    \bm g_{\gamma k}^{\star,\rm EC}
    =
    \delta_k + \sum_{\gamma\subset \Gamma_k}\mathcal E_{\gamma k}^{\star,\rm EC},
\end{equation}
Substituting
\cref{eq:sat_ec_entropy_flux} into
\cref{eq:sat_defect_decomposition_leftover}, simplifying, and applying
\cref{eq:cdsbp_entropy_condition} gives
\begin{align}
    \fn{E}_{\gamma k}^{\star,\rm EC} &=
    \frac12(\bm w_{\gamma k}-\bm w_{\gamma v})^T
    \overline{\B}_\gamma\bm f_{\gamma k}^{*}
    -\frac12\bm1_\gamma^T\B_\gamma
    (\bm\psi_{\gamma k}^{n}+\bm\psi_{\gamma v}^{n})
    =
    0,
    \label{eq:sat_ec_cancellation}
\end{align}
which after substituting in \cref{eq:sat_defect_decomposition} yields
\begin{equation}\label{eq:sat_ec_local_defect}
    \delta_k^{\rm AOR}=\delta_k.
\end{equation}

To identify the remaining difference, we use the AOR distribution matrix in
\cref{eq:mean_subtracted_distribution_matrices} to define
\begin{align}
    \bm a_k^{\rm AOR}
    &\coloneqq
    (\mathsf L_k^{\rm AOR}\otimes\I_{n_c})\bm w_k,
    \qquad
    b_k^{\rm AOR}
    \coloneqq
    \bm w_k^T\bm a_k^{\rm AOR}.
    \label{eq:aor_distribution}
\end{align}
When \(b_k^{\rm AOR}\ne0\), the AOR correction vector corresponding to the correction in \cref{eq:cdsbp_correction} is
\[
    \bm c_k^{\rm AOR}
    =
    \frac{\bm a_k^{\rm AOR}}{b_k^{\rm AOR}} \delta_k^{\rm AOR}.
\]
The present method instead uses the conservative-variable distribution vector,
\(\bm\alpha_{k,\epsilon}\), given in \cref{eq:alpha_regularized}.

\begin{remark}\label{rem:aor_element_local_correction}
For collocated facet and volume nodes and the compatible
entropy-conservative flux pair in \cref{eq:sat_ec_entropy_flux},
\cref{eq:sat_ec_local_defect} gives an element-local form of the
AOR correction, \ie, both \(\delta_k\) and \(\bm a_k^{\rm AOR}\) use only the
state on \(\Omega_k\); no numerical entropy flux is needed to compute the
correction, and \(\bm f_{\gamma k}^{*}\) remains in the base D-SBP
residual.
\end{remark}

\subsection{Relation to entropy-correction artificial viscosity}

The ECAV construction of Chan \cite{chan2025artificial} begins with a volume entropy residual written using the discrete entropy projection. Specializing that construction to a diagonal-\(\E\) SBP discretization, the projected entropy variables are represented by \(\bm w_k\), and the corresponding entropy-projected conservative state is represented by \(\bm u_k\). In the present notation, the volume entropy residual of \cite{chan2025artificial} then becomes
\begin{equation}\label{eq:chan_entropy_defect}
    \delta_k^{\rm ECAV}
    \coloneqq
    -\sum_{i=1}^{d}
    (\bm f_k^{(i)})^T\overline{\Q}_{x_i k}\bm w_k
    +\sum_{\gamma\subset\Gamma_k}
    \bm1_\gamma^T\B_\gamma\bm\psi_{\gamma k}^{n}.
\end{equation}
For each coordinate direction, the SBP identity gives
\[
    -(\bm f_k^{(i)})^T\overline{\Q}_{x_i k}\bm w_k
    =
    \bm w_k^T\overline{\Q}_{x_i k}\bm f_k^{(i)}
    -
    \bm w_k^T\overline{\E}_{x_i k}\bm f_k^{(i)}.
\]
Summing over the coordinate directions and using the entropy-potential identity \cref{eq:entropy_potential_identity} gives
\begin{align}
    \delta_k^{\rm ECAV}
    &=
    \sum_{i=1}^{d}
    \bm w_k^T\overline{\Q}_{x_i k}\bm f_k^{(i)}
    -\sum_{\gamma\subset\Gamma_k}
    \bm w_{\gamma k}^T\overline{\B}_\gamma
    \bm f_{\gamma k}^{n}
    +\sum_{\gamma\subset\Gamma_k}
    \bm1_\gamma^T\B_\gamma\bm\psi_{\gamma k}^{n}
    \notag\\
    &=
    \sum_{i=1}^{d}
    \bm w_k^T\overline{\Q}_{x_i k}\bm f_k^{(i)}
    -\sum_{\gamma\subset\Gamma_k}
    \bm1_\gamma^T\B_\gamma\bm g_{\gamma k}^{n}
    =
    \delta_k.
    \label{eq:chan_local_defect_equivalence}
\end{align}
Therefore, we can write
\begin{equation}\label{eq:shared_correction_defect}
    \delta_k^{\rm AOR}
    =
    \delta_k^{\rm ECAV}
    =
    \delta_k,
\end{equation}
establishing the equivalence of the scalar entropy defect between the three methods.

As a result, the methods differ mainly due to the fact that AOR uses the distribution in \cref{eq:aor_distribution}; ECAV applies the entropy-producing part of \(\delta_k\) with a distribution vector constructed using viscous volume and SAT terms; and the present method distributes that part through \(\sigma_k^{\rm ed}\bm\alpha_{k,\epsilon}\). The regularization also differs in that AOR sets its
correction to zero for constant entropy variables, ECAV uses a fixed tolerance in its viscosity coefficient \cite{chan2025artificial}, and \cref{eq:alpha_moments} carries the present regularization factor \(1-\eta_k\) explicitly. In the present formulation, an
entropy-conservative facet flux gives \(\varepsilon_*^\gamma=0\) in
\cref{eq:cdsbp_entropy_condition}, while an entropy-stable flux such as LLF
gives \(\varepsilon_*^\gamma\ge0\), without changing \(\delta_k\) or
\(\bm\alpha_{k,\epsilon}\). Both remain independent of neighboring traces and
the C/D facet partition. The sensitivity associated with noncompatible SATs, \ie, SATs for which the numerical flux \(\bm f_{\gamma k}^{*}\) and the chosen numerical entropy flux \(\bm g_{\gamma k}^{\star}\) do not yield the cancellation in \cref{eq:sat_ec_cancellation}, is analyzed in \cref{app:sat_dependent_correction}.

\section{Numerical experiments}\label{sec:numerical_experiments}
The numerical experiments assess accuracy, long-time behavior, and robustness
of the hybrid method.  We use ``C-RCM'', ``D-RCM'', and ``CD-RCM'' for the
all-C, all-D, and hybrid C/D residual-correction discretizations,
respectively.  Unless stated otherwise, the RCM calculations use the
entropy-dissipative correction mode, \(\sigma_k=\sigma_k^{\rm ed}\).  The Euler
comparison baseline is the Hadamard-form discretization with the
entropy-conservative two-point flux of Ranocha \cite{ranocha2018comparison}.
We denote its C-SBP form by
``C-2pt'' and its D-SBP form with local Lax--Friedrichs (LLF) interface
dissipation by ``D-2pt''.  The LLF wave speed is the Guermond--Popov bound
\cite{guermond2016fast}, unless stated otherwise. All RCM calculations use the correction distribution in \cref{eq:alpha_regularized} with the components given in \cref{eq:coefficient_ab}. Schemes with the suffix ``AD'' use the volume artificial dissipation of
\cite{worku2026targeted} with its default parameters.

The prescribed regions and sensors used below act
elementwise and label each element C or D, denoting continuous or discontinuous
coupling, respectively. In this setup, an interior facet is D-SBP if either adjacent element is labeled D and is C-SBP otherwise. The labels are updated only after an accepted time step and are held fixed through all stages and rejected retries of the following time step.

All calculations use diagonal-norm, diagonal-\(\E\) SBP operators.  The
one-dimensional meshes use Legendre--Gauss--Lobatto operators, while the
Euler calculations use diagonal-\(\E\) operators constructed from the
degree-\(2p\) triangular and tetrahedral quadrature rules in
\cite{worku2024quadrature}.  Each coordinate direction is divided into
\(n_{\rm edge}\) segments; splitting each Cartesian cell into two triangles
in 2D or six tetrahedra in 3D gives \(n_{\rm edge}\),
\(2n_{\rm edge}^2\), and \(6n_{\rm edge}^3\) elements in 1D, 2D, and 3D,
respectively.  The vortex convergence calculations use standard fourth-order Runge-Kutta (RK4) method with a CFL number of \(0.1\), while the Burgers, Kelvin--Helmholtz, and Riemann
calculations use the adaptive third-order, four-stages strong-stability-preserving Runge-Kutta (SSPRK43) method \cite{fekete2022embedded,kraaijevanger1991contractivity,ranocha2022optimized}
from \texttt{OrdinaryDiffEq.jl} \cite{rackauckas2017differentialequations}
with absolute and relative tolerances \(10^{-7}\).

For an exact solution, the error is measured in the broken norm
\begin{equation}\label{eq:cdsbp_error_norm}
    \|\bm e\|_{\H}^{2}
    =
    \sum_{\Omega_k\in\mathcal T_h}
    \sum_{j=1}^{n_c}
    \bm e_{k,j}^{T}\H_k\bm e_{k,j}.
\end{equation}
At output times, we also monitor the conservation defect
\begin{equation}\label{eq:cdsbp_conservation_error}
    \mathcal E_{\rm con}(t)
    =
    \sum_{j=1}^{n_c}
    \left|
    \sum_{\Omega_k\in\mathcal T_h}
    \bm1_k^T\H_k\der[\bm u_{k,j}]{t}
    \right|
\end{equation}
and the normalized change in discrete mathematical entropy,
\begin{equation}\label{eq:cdsbp_entropy_change}
    \Delta S_h(t)
    =
    \frac{S_h(\bm u_b(t))-S_h(\bm u_b(0))}
    {\left|S_h(\bm u_b(0))\right|}.
\end{equation}
This normalization is used when
\(\left|S_h(\bm u_b(0))\right|\) is nonzero; otherwise, we report only the
numerator. In the numerical experiments, we use
\(\fn S(\fn U)=\fn U^2/2\) and \(\fn G(\fn U)=\fn U^3/3\)
for the Burgers equation. For all Euler experiments, we set
\(\gamma_{\rm gas}=1.4\) and use
\(\fn S=-\rho\log(p/\rho^{\gamma_{\rm gas}})/(\gamma_{\rm gas}-1)\)
and \(\fn G^{(i)}=v_i\fn S\), where \(v_i\) denotes the
\(i\)-th velocity component.
Since interface SATs are evaluated only on D-SBP facets, we
use the D-facet fraction
\(\phi_{\rm D}(t)=|\Gamma_{\rm D}^{I}(t)|/|\Gamma^{I}|\), with
\(0\le\phi_{\rm D}\le1\), as a proxy for the associated interface
computation. Smaller values indicate fewer SAT evaluations. When a single
value is reported for a complete run, \(\overline\phi_{\rm D}\) denotes the
average over all right-hand-side evaluations.

\subsection{Two- and three-dimensional isentropic vortex}
\label{sec:numerics_isentropic_vortex}
The isentropic-vortex problems
\cite{erlebacher1997interaction,williams2013nodal} are used to assess smooth-solution
accuracy.  In two dimensions, \(\Omega=[0,20]\times[-5,5]\), and the exact
solution is
\begin{equation}\label{eq:euler2d_vortex}
\begin{aligned}
    \rho
    &=\left(1-
    \frac{\alpha^2(\gamma_{\rm gas}-1)}{16\gamma_{\rm gas}\pi^2}
    \exp\!\left(2(1-r^2)\right)\right)^{1/(\gamma_{\rm gas}-1)},
    \qquad 
    p=\rho^{\gamma_{\rm gas}},\\
    u&=1-\frac{\alpha}{2\pi}(y-y_c)\exp(1-r^2),
    \qquad
    v=\frac{\alpha}{2\pi}(x-x_c-t)\exp(1-r^2),
\end{aligned}
\end{equation}
where
\(r^2=(x-x_c-t)^2+(y-y_c)^2\), \(\alpha=3\), and
\((x_c,y_c)=(5,0)\).  In three dimensions,
\(\Omega=[-10,10]^3\), and
\begin{equation}\label{eq:euler3d_vortex}
\begin{aligned}
    \rho
    &=\left(1-\frac{2}{25}(\gamma_{\rm gas}-1)
    \exp\!\left(1-(y-t)^2-x^2\right)\right)^{1/(\gamma_{\rm gas}-1)},\\
    E
    &=\frac{\rho^{\gamma_{\rm gas}}}{\gamma_{\rm gas}(\gamma_{\rm gas}-1)}
    +\frac{\rho}{2}(u^2+v^2+w^2),\\
    u&=-\frac{2}{5}(y-t)
    \exp\!\left(\frac{1-(y-t)^2-x^2}{2}\right),\\
    v&=1+\frac{2}{5}x
    \exp\!\left(\frac{1-(y-t)^2-x^2}{2}\right),
    \qquad 
    w=0.
\end{aligned}
\end{equation}
Periodic boundary conditions are imposed in every coordinate direction.

Let \(\bm x_k^c\) and \(r_k\) be the centroid and bounding radius of
\(\Omega_k\).  In 2D, an element is labeled D when
\begin{equation}\label{eq:vortex_2d_d_region}
    \operatorname{dist}_{\rm per}
    \left(\bm x_k^c,(x_c+t,y_c)\right)
    \le R_{\rm D}+r_k,
    \qquad R_{\rm D}=2.
\end{equation}
Thus, the D region is a circle slightly larger than the vortex
core.  The 3D solution in \cref{eq:euler3d_vortex} is independent of \(z\)
and is therefore a vortex tube.  Its D region is the corresponding moving
cylinder,
\begin{equation}\label{eq:vortex_3d_d_region}
    \operatorname{dist}^{xy}_{\rm per}
    \left((x_k^c,y_k^c),(0,t)\right)
    \le R_{\rm D}+r_k,
    \qquad R_{\rm D}=2.
\end{equation}

The convergence study tests whether the moving C/D partition and associated
state transfers retain smooth-solution accuracy. The long-time study examines
the effects of repeated repartitioning on conservation, entropy, and solution
error. No volume artificial dissipation is added. The convergence runs use the
entropy-dissipative correction mode, \(\sigma_k=\sigma_k^{\rm ed}\), whereas the
long-time runs use the entropy-conservative mode,
\(\sigma_k=\sigma_k^{\rm ec}\). The long-time runs use relaxation RK4
\cite{ketcheson2019relaxation,ranocha2020relaxation} to reduce temporal
contamination of the entropy histories.

\begin{table}[!t]
\centering
\scriptsize
\setlength{\tabcolsep}{2.2pt}
\renewcommand{\arraystretch}{1.14}
\caption{Isentropic-vortex convergence in 2D at \(t_f=1\).}
\label{tab:isentropic_vortex_cdsbp_convergence_2d}
\begin{tabular}{c r cc@{\hspace{8pt}}cc@{\hspace{8pt}}cc@{\hspace{8pt}}cc@{\hspace{8pt}}cc}
\toprule
& & \multicolumn{2}{c}{C-2pt} & \multicolumn{2}{c}{C-RCM} & \multicolumn{2}{c}{D-2pt} & \multicolumn{2}{c}{D-RCM} & \multicolumn{2}{c}{CD-RCM} \\
\cmidrule(lr){3-4}\cmidrule(lr){5-6}\cmidrule(lr){7-8}\cmidrule(lr){9-10}\cmidrule(lr){11-12}
\(p\) & \(n_{\rm edge}\) & error & rate & error & rate & error & rate & error & rate & error & rate \\
\midrule
1 & 32 & 3.553e-01 & -- & 3.291e-01 & -- & 2.118e-01 & -- & 2.177e-01 & -- & 2.183e-01 & -- \\
 & 64 & 7.981e-02 & 2.15 & 7.906e-02 & 2.06 & 5.876e-02 & 1.85 & 6.109e-02 & 1.83 & 6.132e-02 & 1.83 \\
 & 128 & 1.465e-02 & 2.45 & 1.630e-02 & 2.28 & 1.568e-02 & 1.91 & 1.653e-02 & 1.89 & 1.660e-02 & 1.88 \\
 & 256 & 3.351e-03 & 2.13 & 3.707e-03 & 2.14 & 4.133e-03 & 1.92 & 4.423e-03 & 1.90 & 4.435e-03 & 1.90 \\
\addlinespace[0.25em]
2 & 32 & 8.497e-02 & -- & 9.070e-02 & -- & 3.371e-02 & -- & 3.813e-02 & -- & 3.835e-02 & -- \\
 & 64 & 1.273e-02 & 2.74 & 1.562e-02 & 2.54 & 5.116e-03 & 2.72 & 6.161e-03 & 2.63 & 6.243e-03 & 2.62 \\
 & 128 & 2.475e-03 & 2.36 & 3.073e-03 & 2.35 & 7.303e-04 & 2.81 & 9.017e-04 & 2.77 & 9.304e-04 & 2.75 \\
 & 256 & 6.547e-04 & 1.92 & 7.145e-04 & 2.10 & 9.590e-05 & 2.93 & 1.218e-04 & 2.89 & 1.329e-04 & 2.81 \\
\addlinespace[0.25em]
3 & 32 & 1.226e-02 & -- & 1.758e-02 & -- & 5.436e-03 & -- & 7.793e-03 & -- & 7.807e-03 & -- \\
 & 64 & 1.059e-03 & 3.53 & 1.542e-03 & 3.51 & 4.419e-04 & 3.62 & 6.229e-04 & 3.65 & 6.254e-04 & 3.64 \\
 & 128 & 1.060e-04 & 3.32 & 1.456e-04 & 3.40 & 3.502e-05 & 3.66 & 4.938e-05 & 3.66 & 4.990e-05 & 3.65 \\
 & 256 & 8.228e-06 & 3.69 & 1.272e-05 & 3.52 & 2.589e-06 & 3.76 & 3.658e-06 & 3.75 & 3.737e-06 & 3.74 \\
\addlinespace[0.25em]
4 & 32 & 1.790e-03 & -- & 2.936e-03 & -- & 7.904e-04 & -- & 1.336e-03 & -- & 1.337e-03 & -- \\
 & 64 & 9.617e-05 & 4.22 & 1.429e-04 & 4.36 & 3.601e-05 & 4.46 & 5.553e-05 & 4.59 & 5.573e-05 & 4.58 \\
 & 128 & 5.387e-06 & 4.16 & 8.064e-06 & 4.15 & 1.547e-06 & 4.54 & 2.337e-06 & 4.57 & 2.366e-06 & 4.56 \\
 & 256 & 2.442e-07 & 4.46 & 3.629e-07 & 4.47 & 6.033e-08 & 4.68 & 9.053e-08 & 4.69 & 9.277e-08 & 4.67 \\
\bottomrule
\end{tabular}
\end{table}

\begin{table}[!t]
\centering
\scriptsize
\setlength{\tabcolsep}{2.2pt}
\renewcommand{\arraystretch}{1.14}
\caption{Isentropic-vortex convergence in 3D at \(t_f=1\).}
\label{tab:isentropic_vortex_cdsbp_convergence_3d}
\begin{tabular}{c r cc@{\hspace{8pt}}cc@{\hspace{8pt}}cc@{\hspace{8pt}}cc@{\hspace{8pt}}cc}
\toprule
& & \multicolumn{2}{c}{C-2pt} & \multicolumn{2}{c}{C-RCM} & \multicolumn{2}{c}{D-2pt} & \multicolumn{2}{c}{D-RCM} & \multicolumn{2}{c}{CD-RCM} \\
\cmidrule(lr){3-4}\cmidrule(lr){5-6}\cmidrule(lr){7-8}\cmidrule(lr){9-10}\cmidrule(lr){11-12}
\(p\) & \(n_{\rm edge}\) & error & rate & error & rate & error & rate & error & rate & error & rate \\
\midrule
1 & 10 & 3.613e+00 & -- & 3.572e+00 & -- & 2.394e+00 & -- & 2.429e+00 & -- & 2.481e+00 & -- \\
 & 15 & 1.875e+00 & 1.62 & 1.861e+00 & 1.61 & 1.286e+00 & 1.53 & 1.352e+00 & 1.45 & 1.384e+00 & 1.44 \\
 & 20 & 1.081e+00 & 1.91 & 1.046e+00 & 2.00 & 6.709e-01 & 2.26 & 7.051e-01 & 2.26 & 7.263e-01 & 2.24 \\
 & 25 & 6.896e-01 & 2.01 & 6.921e-01 & 1.85 & 4.235e-01 & 2.06 & 4.598e-01 & 1.92 & 4.739e-01 & 1.91 \\
\addlinespace[0.25em]
2 & 10 & 2.311e+00 & -- & 2.295e+00 & -- & 1.259e+00 & -- & 1.282e+00 & -- & 1.289e+00 & -- \\
 & 15 & 1.114e+00 & 1.80 & 1.100e+00 & 1.81 & 5.611e-01 & 1.99 & 5.968e-01 & 1.89 & 6.038e-01 & 1.87 \\
 & 20 & 6.132e-01 & 2.07 & 6.224e-01 & 1.98 & 2.863e-01 & 2.34 & 3.033e-01 & 2.35 & 3.114e-01 & 2.30 \\
 & 25 & 3.933e-01 & 1.99 & 4.036e-01 & 1.94 & 1.720e-01 & 2.28 & 1.817e-01 & 2.30 & 1.892e-01 & 2.23 \\
\addlinespace[0.25em]
3 & 10 & 7.382e-01 & -- & 8.124e-01 & -- & 4.674e-01 & -- & 5.041e-01 & -- & 5.051e-01 & -- \\
 & 15 & 2.752e-01 & 2.43 & 3.129e-01 & 2.35 & 1.368e-01 & 3.03 & 1.521e-01 & 2.96 & 1.534e-01 & 2.94 \\
 & 20 & 1.238e-01 & 2.78 & 1.481e-01 & 2.60 & 5.683e-02 & 3.05 & 6.846e-02 & 2.77 & 6.920e-02 & 2.77 \\
 & 25 & 6.110e-02 & 3.16 & 7.484e-02 & 3.06 & 2.631e-02 & 3.45 & 3.245e-02 & 3.35 & 3.300e-02 & 3.32 \\
\addlinespace[0.25em]
4 & 10 & 3.011e-01 & -- & 3.778e-01 & -- & 1.662e-01 & -- & 2.181e-01 & -- & 2.183e-01 & -- \\
 & 15 & 7.186e-02 & 3.53 & 1.024e-01 & 3.22 & 4.036e-02 & 3.49 & 5.808e-02 & 3.26 & 5.819e-02 & 3.26 \\
 & 20 & 2.156e-02 & 4.19 & 3.133e-02 & 4.12 & 1.192e-02 & 4.24 & 1.708e-02 & 4.26 & 1.716e-02 & 4.25 \\
 & 25 & 8.944e-03 & 3.94 & 1.285e-02 & 3.99 & 4.690e-03 & 4.18 & 6.666e-03 & 4.22 & 6.719e-03 & 4.20 \\
\bottomrule
\end{tabular}
\end{table}

\Cref{tab:isentropic_vortex_cdsbp_convergence_2d,tab:isentropic_vortex_cdsbp_convergence_3d}
show decreasing errors under mesh refinement for all five methods. On the
finest meshes, the observed rates are approximately in the \(p\)- to
\((p+1)\)-order range. CD-RCM and D-RCM have closely matching errors and rates at every degree
and resolution, despite the moving C/D partition. Thus, the hybrid
coupling and state transfer retain the smooth-solution accuracy of D-RCM
in these tests. C-RCM also achieves errors and rates comparable to C-2pt,
without using two-point volume fluxes.

\begin{figure}[!t]
\centering
\begin{subfigure}[t]{0.485\textwidth}
\centering
\includegraphics[width=0.99\linewidth]{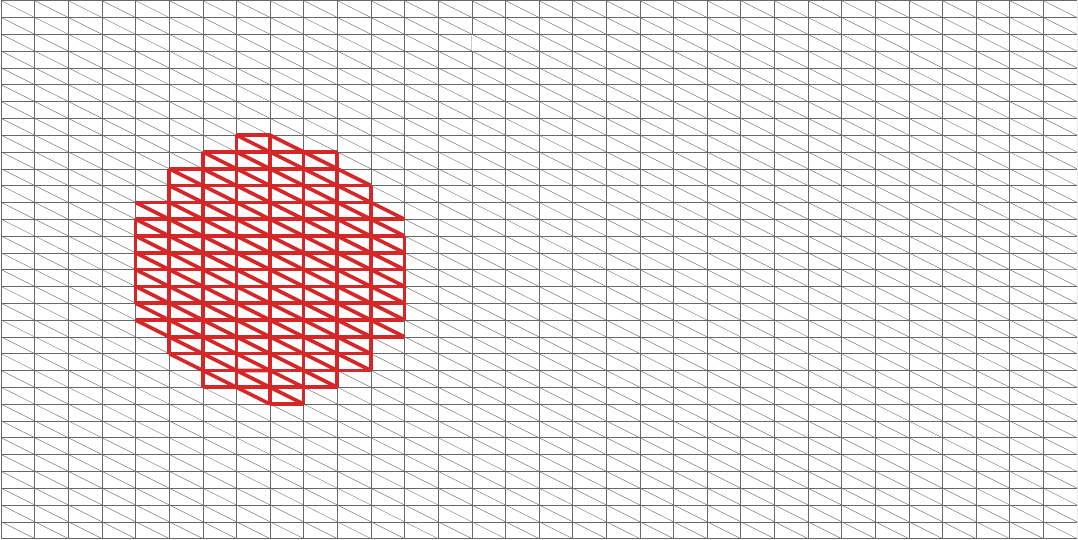}
\caption{Facet partition, \(t=0\)}
\label{fig:vortex_2d_mesh_t0}
\end{subfigure}
\hfill
\begin{subfigure}[t]{0.485\textwidth}
\centering
\includegraphics[width=0.99\linewidth]{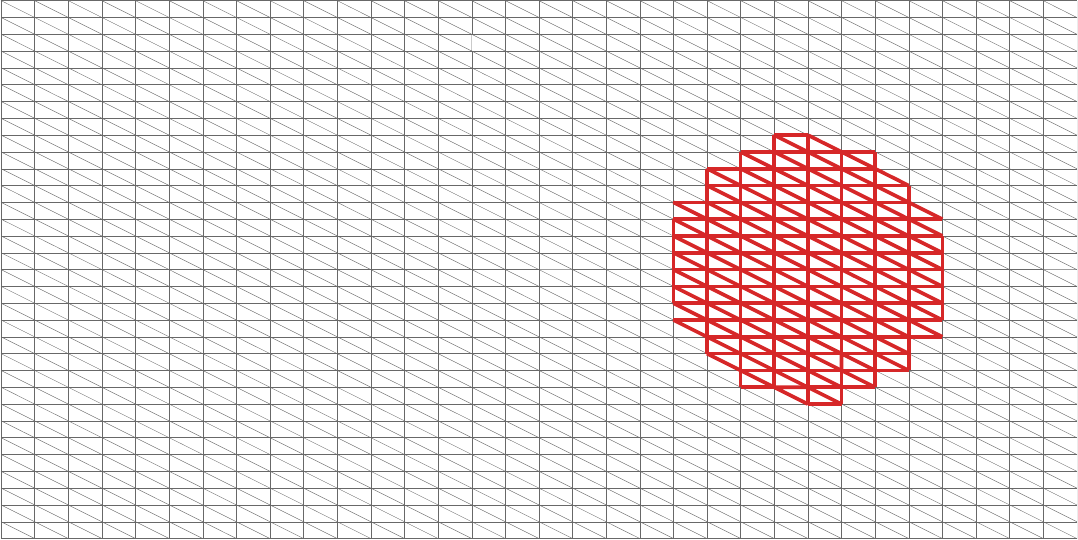}
\caption{Facet partition, \(t=10\)}
\label{fig:vortex_2d_mesh_t10}
\end{subfigure}

\begin{subfigure}[t]{0.485\textwidth}
\centering
\includegraphics[width=\linewidth]{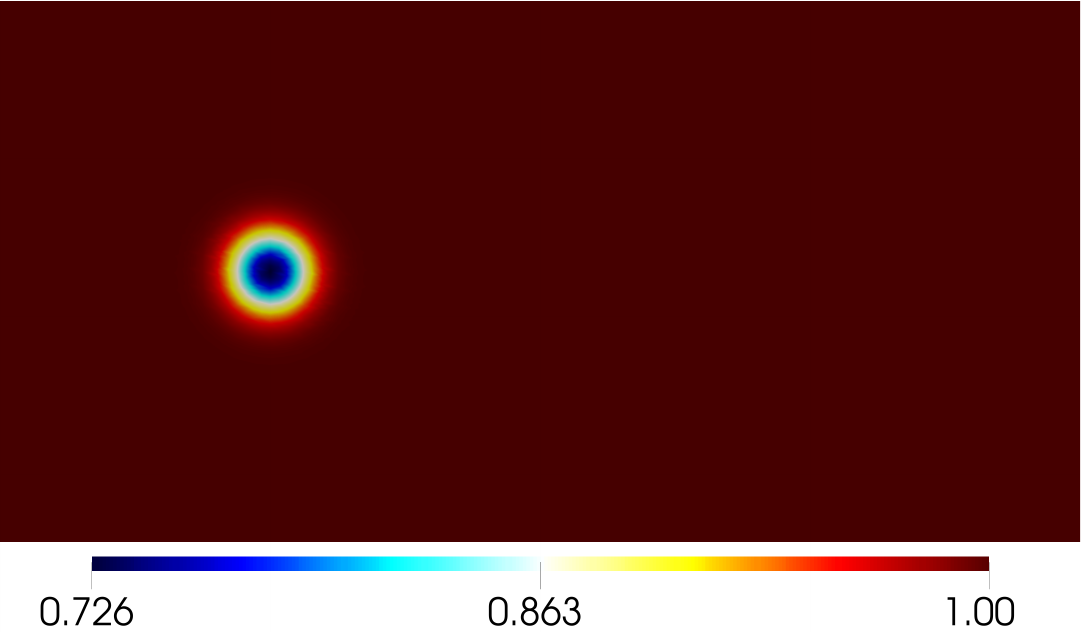}
\caption{Density, \(t=0\)}
\label{fig:vortex_2d_density_t0}
\end{subfigure}
\hfill
\begin{subfigure}[t]{0.485\textwidth}
\centering
\includegraphics[width=\linewidth]{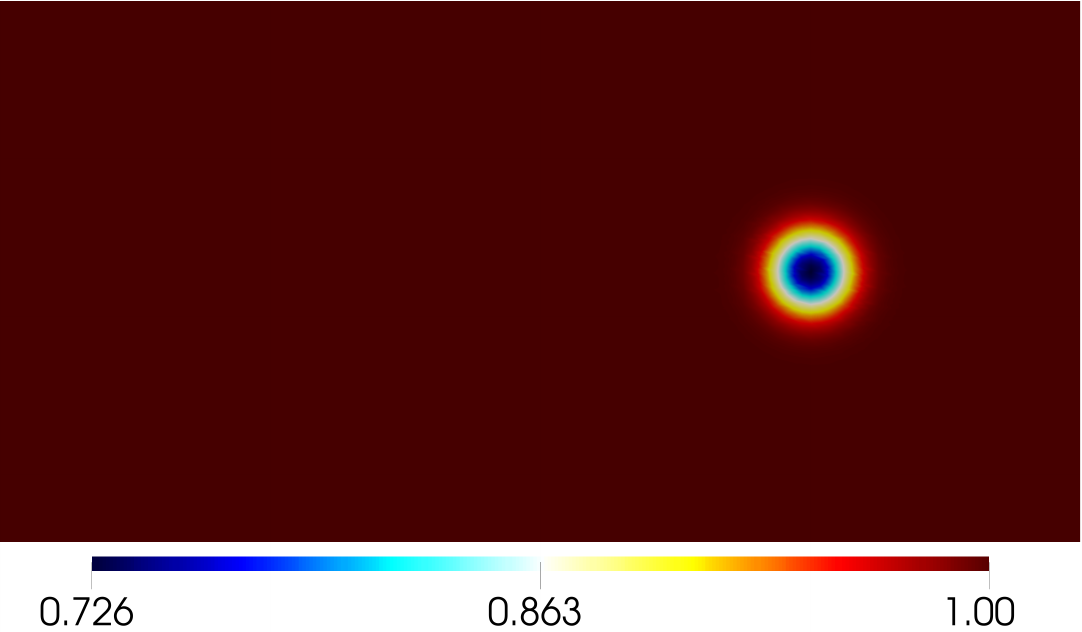}
\caption{Density, \(t=10\)}
\label{fig:vortex_2d_density_t10}
\end{subfigure}
\caption{Hybrid facet partition and density for the two-dimensional
isentropic vortex with \(p=4\) and \(n_{\rm edge}=32\). D-SBP facets are red.}
\label{fig:isentropic_vortex_mesh_snapshots}
\end{figure}

\begin{figure}[!t]
\centering
\begin{subfigure}[t]{0.48\textwidth}
\centering
\includegraphics[width=\linewidth]{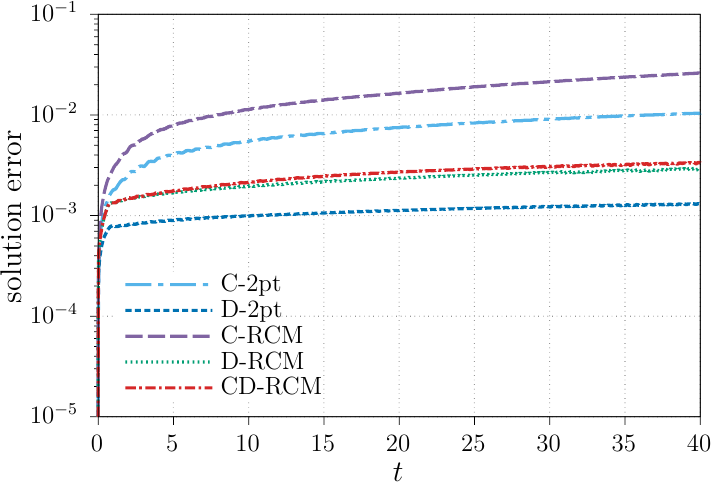}
\caption{2D, \(n_{\rm edge}=32\)}
\label{fig:vortex_2d_long_time_error}
\end{subfigure}
\hfill
\begin{subfigure}[t]{0.48\textwidth}
\centering
\includegraphics[width=\linewidth]{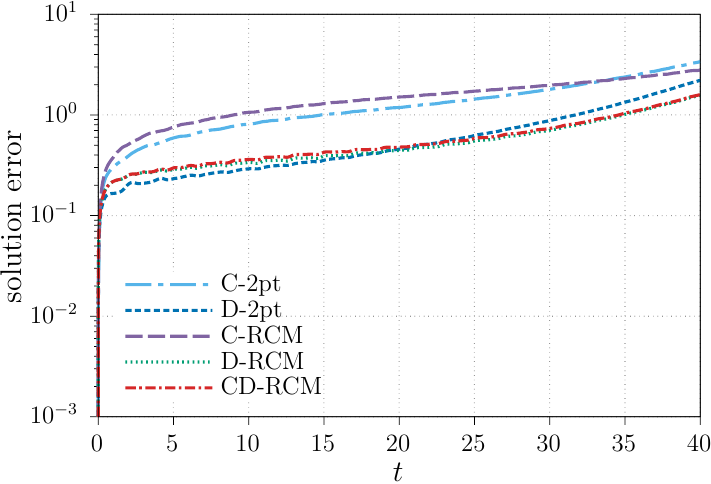}
\caption{3D, \(n_{\rm edge}=10\)}
\label{fig:vortex_3d_long_time_error}
\end{subfigure}
\caption{Long-time isentropic-vortex solution-error histories for \(p=4\).}
\label{fig:isentropic_vortex_long_time_error}
\end{figure}

\begin{figure}[!t]
\centering
\begin{subfigure}[t]{0.48\textwidth}
\centering
\includegraphics[width=\linewidth]{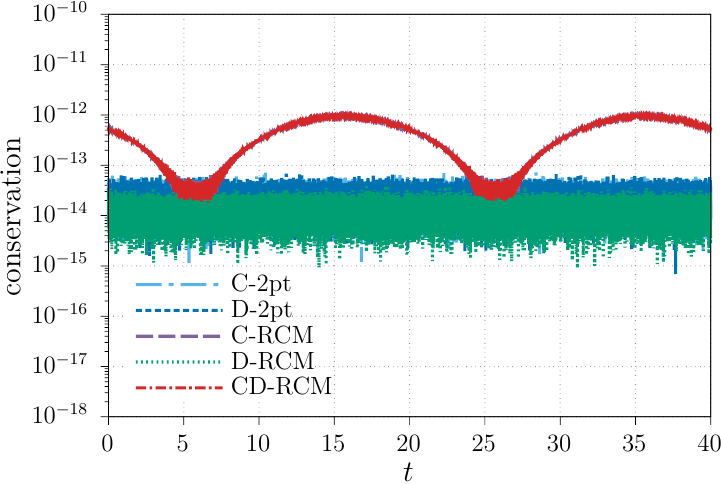}
\caption{2D, \(n_{\rm edge}=32\)}
\label{fig:vortex_2d_long_time_conservation}
\end{subfigure}
\hfill
\begin{subfigure}[t]{0.48\textwidth}
\centering
\includegraphics[width=\linewidth]{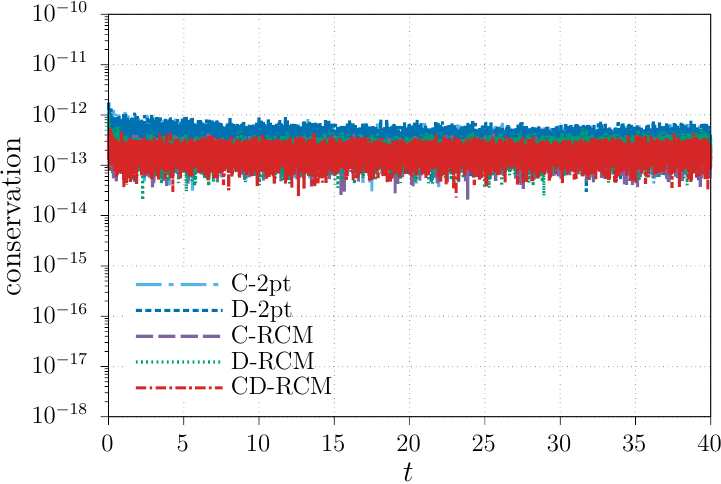}
\caption{3D, \(n_{\rm edge}=10\)}
\label{fig:vortex_3d_long_time_conservation}
\end{subfigure}
\caption{Long-time isentropic-vortex conservation defects for \(p=4\).}
\label{fig:isentropic_vortex_long_time_conservation}
\end{figure}

\begin{figure}[!t]
\centering
\begin{subfigure}[t]{0.48\textwidth}
\centering
\includegraphics[width=\linewidth]{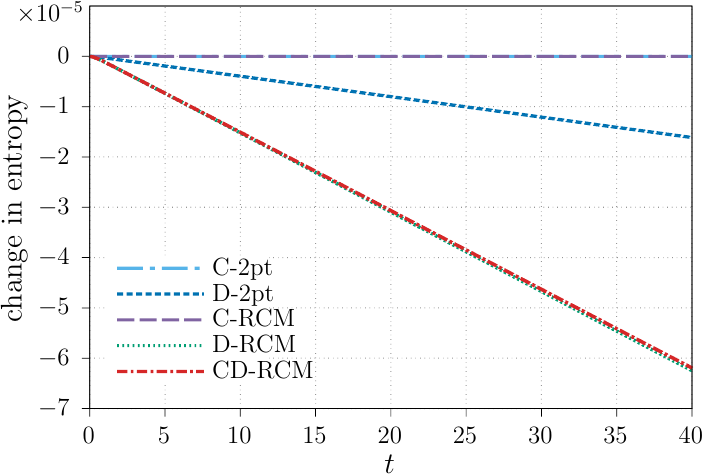}
\caption{2D, \(n_{\rm edge}=32\)}
\label{fig:vortex_2d_long_time_entropy}
\end{subfigure}
\hfill
\begin{subfigure}[t]{0.48\textwidth}
\centering
\includegraphics[width=\linewidth]{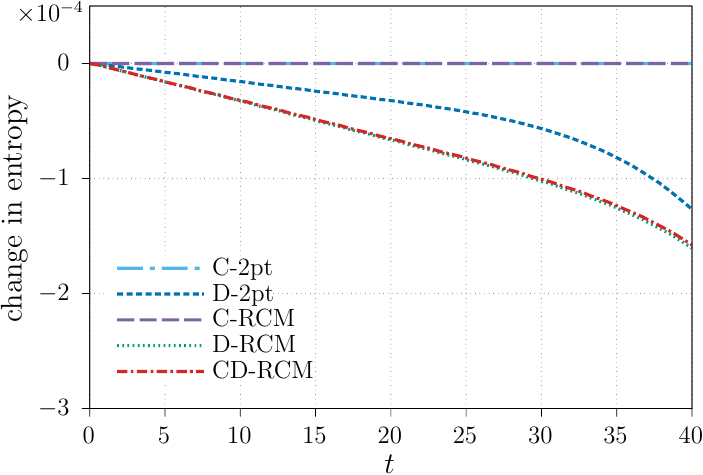}
\caption{3D, \(n_{\rm edge}=10\)}
\label{fig:vortex_3d_long_time_entropy}
\end{subfigure}
\caption{Long-time entropy changes for the isentropic-vortex problems with \(p=4\).}
\label{fig:isentropic_vortex_long_time_entropy}
\end{figure}

\Cref{fig:isentropic_vortex_mesh_snapshots} shows that the D-SBP region
follows the vortex core while the rest of the mesh retains C-SBP coupling,
without visible switching artifacts. The smooth histories in
\cref{fig:isentropic_vortex_long_time_error} show no abrupt error growth
associated with the repeated partition changes. The CD-RCM error histories closely follow those of D-RCM. Both RCM schemes with D-SBP facets have smaller long-time errors than C-RCM in
these tests. Over \(0\le t\le40\), the conservation
defects in \cref{fig:isentropic_vortex_long_time_conservation} remain below
approximately \(10^{-12}\) in both 2D and 3D. The entropy histories in
\cref{fig:isentropic_vortex_long_time_entropy} remain essentially unchanged
for C-2pt and C-RCM. Methods with D-SBP facets have a nonpositive entropy
change, consistent with LLF dissipation and, for CD-RCM, entropy-nonincreasing
D-to-C transfers. Finally, similar to the error profiles, the entropy histories of the CD-RCM method closely follow those of the D-RCM.

\subsection{One-dimensional inviscid Burgers equation}
\label{sec:numerics_burgers}
We solve
\begin{equation}\label{eq:burgers_numerical_problem}
    \pder[\fn{U}]{t}+\frac{1}{2}\pder[\fn{U}^2]{x}=0,
    \qquad \Omega=[0,1],
\end{equation}
with periodic boundary conditions and the initial condition
\begin{equation}\label{eq:burgers_numerical_initial_condition}
    \fn{U}(x,0)=2+\sin(2\pi x).
\end{equation}
The shock forms at \(t_s=1/(2\pi)\). For CD-RCM, element \(\Omega_k\) is
labeled D after an accepted step when the solution smoothness sensor, \(\theta_k\) (defined in \cite{worku2026targeted}), satisfies
\begin{equation}\label{eq:burgers_hybrid_sensor}
    \theta_k\ge 10^{-2},
\end{equation}
and relabeled C when \(\theta_k\le 10^{-3}\). We compare solutions of the C-split-AD, C-RCM-AD, D-split, D-RCM, and CD-RCM methods.
Volume AD of \cite{worku2026targeted} is used only with the C-split and C-RCM schemes, as they produce highly oscillatory solutions without it.

\begin{table}[!t]
\centering
\scriptsize
\setlength{\tabcolsep}{3.8pt}
\renewcommand{\arraystretch}{1.15}
\caption{Accepted SSPRK43 step counts for the Burgers problem at
\(t_f=1.0\). The final column is the mean D-facet fraction for CD-RCM.}
\label{tab:burgers_hybrid_timesteps}
\begin{tabular}{l l l l l l@{\hspace{12pt}}l l}
\toprule
\(p\) & \(n_{\rm edge}\) & C-split-AD & C-RCM-AD & D-split & D-RCM & CD-RCM & \(\overline{\phi}_{\rm D}\) \\
\midrule
2 & 25 & 5297 & 5193 & 5978 & 5841 & 6206 & 0.306 \\
 & 50 & 9282 & 9098 & 10449 & 10204 & 10841 & 0.166 \\
 & 100 & 16375 & 16087 & 18377 & 17942 & 19088 & 0.084 \\
\addlinespace[0.25em]
3 & 25 & 8263 & 8200 & 9192 & 8995 & 9223 & 0.215 \\
 & 50 & 14517 & 14387 & 16139 & 15804 & 16141 & 0.123 \\
 & 100 & 25741 & 25507 & 28475 & 27874 & 28479 & 0.068 \\
\addlinespace[0.25em]
4 & 25 & 11340 & 11231 & 12470 & 12247 & 12315 & 0.186 \\
 & 50 & 20059 & 19864 & 21951 & 21529 & 21604 & 0.098 \\
 & 100 & 35548 & 35168 & 38806 & 38071 & 38162 & 0.054 \\
\addlinespace[0.25em]
5 & 25 & 13854 & 13870 & 15752 & 15521 & 15591 & 0.165 \\
 & 50 & 24311 & 24356 & 27778 & 27353 & 27432 & 0.091 \\
 & 100 & 42974 & 43069 & 49185 & 48378 & 48491 & 0.049 \\
\bottomrule
\end{tabular}
\end{table}

\begin{figure}[!t]
\centering
\begin{subfigure}[t]{0.485\textwidth}
\centering
\includegraphics[width=\linewidth]{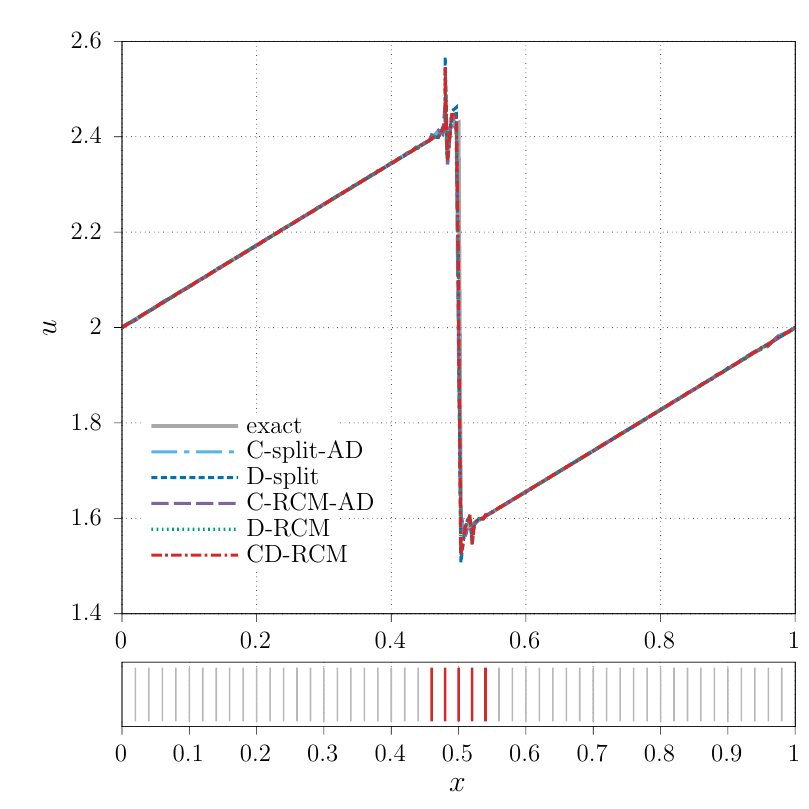}
\caption{\(t=1.0\)}
\label{fig:burgers_solution_t1p0}
\end{subfigure}
\hfill
\begin{subfigure}[t]{0.485\textwidth}
\centering
\includegraphics[width=\linewidth]{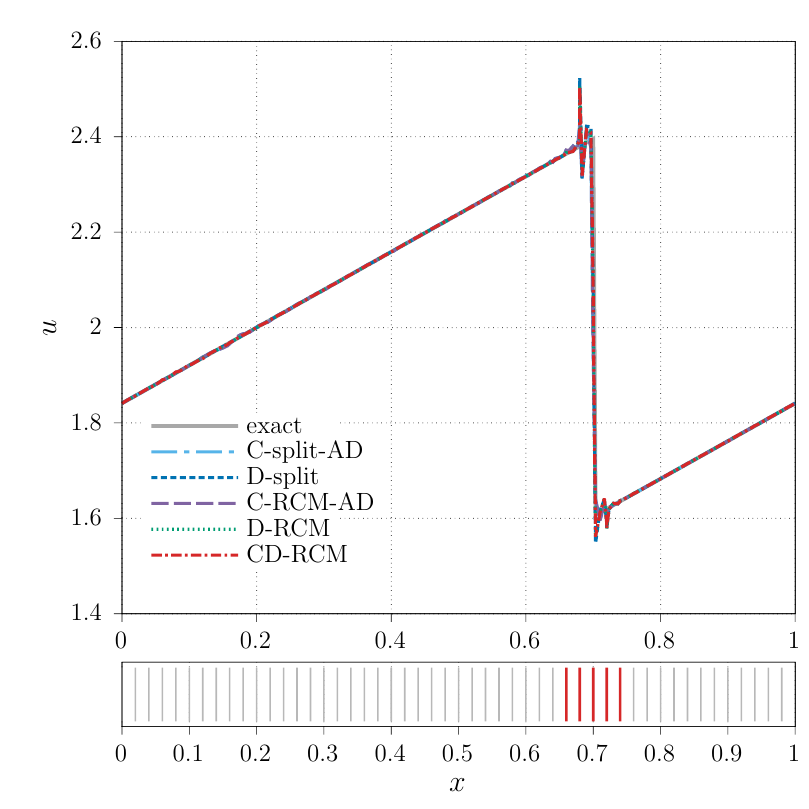}
\caption{\(t=1.1\)}
\label{fig:burgers_solution_t1p1}
\end{subfigure}
\caption{Post-shock Burgers solutions and CD-RCM facet classifications
for \(p=4\) and \(n_{\rm edge}=50\). D-SBP facets are shown in red.}
\label{fig:burgers_hybrid_solution}
\end{figure}

In \cref{fig:burgers_hybrid_solution}, all methods reproduce the
moving shock location and agree closely with the exact entropy solution
away from the front. Local overshoots and undershoots remain near the
shock. The CD-RCM profile closely follows D-RCM, while its D-SBP
facets remain confined to a narrow band that moves with the shock. Compared to the C-split and C-RCM schemes without AD (not shown), CD-RCM exhibits smaller oscillations near the shock. 

In \cref{tab:burgers_hybrid_timesteps}, CD-RCM uses accepted-step counts
within about \(7\%\) of D-RCM across the reported degrees and meshes.
Its mean D-facet fraction ranges from \(0.049\) to \(0.306\) and decreases
under mesh refinement at every degree. Thus, the hybrid scheme retains
similar shock resolution and adaptive stepping to D-RCM while restricting
SAT evaluations to a small fraction of the interfaces.

\subsection{Kelvin--Helmholtz instability}
\label{sec:numerics_khi}
The Kelvin--Helmholtz instability is a demanding robustness test for
high-order methods as vortex roll-up generates increasingly fine,
under-resolved flow structures. We solve the two-dimensional Euler equations on
\(\Omega=[-1,1]^2\) with periodic boundary conditions.  Let
\[
    B(y)=\tanh(15y+7.5)-\tanh(15y-7.5).
\]
Following \cite{rueda2022subcell}, the primitive initial condition is
\begin{equation}\label{eq:khi_initial_condition}
    \rho=\frac12+\frac34 B(y),
    \quad
    u=\frac12\left(B(y)-1\right),
    \quad
    v=\frac1{10}\sin(2\pi x),
    \quad p=1.
\end{equation}

For CD-RCM, we use the same element classification as in the Burgers
problem; \ie, element \(\Omega_k\) is labeled D when \(\theta_k\ge10^{-2}\),
relabeled C when \(\theta_k\le10^{-3}\), and left unchanged between the
thresholds. We classify a run as failed if any nodal density or pressure
becomes nonpositive or nonfinite, or if the adaptive time step satisfies
\(\Delta t\le10^{-8}\) for ten consecutive steps.
We compare C-2pt, D-2pt, C-RCM, D-RCM, and CD-RCM in two sets of
runs, without and with volume AD, using \(p=2,\ldots,8\),
\(n_{\rm edge}=32\), and \(t_f=15\). The correction mode
\(\sigma_k=\sigma_k^{\rm ed}\), LLF coupling, sensor thresholds, mesh,
time integrator, tolerances, and failure criteria are identical in the two
sets. The no-AD table reports the last accepted time \(t_{\rm last}\), with
\(t_{\rm last}=15\) for a completed run. The AD table additionally reports
the accepted-step count \(N_t\) for completed runs and the mean D-facet
fraction \(\overline{\phi}_{\rm D}\).

\begin{table}[!t]
\centering
\scriptsize
\setlength{\tabcolsep}{3.0pt}
\renewcommand{\arraystretch}{1.15}
\caption{Kelvin--Helmholtz robustness with volume AD of
\cite{worku2026targeted} on \(n_{\rm edge}=32\) mesh. Each entry reports
\((t_{\rm last}\,|\,N_t\,|\,\overline{\phi}_{\rm D})\), where
\(\overline{\phi}_{\rm D}\) is the mean D-facet fraction. Accepted-step
counts, \(N_t\), are reported for completed runs. Bold entries did not reach
\(t_f=15\).}
\label{tab:khi_hybrid_robustness}
\begin{tabular}{l l l l l l}
\toprule
\(p\) & C-2pt-AD & C-RCM-AD & D-2pt-AD & D-RCM-AD & CD-RCM-AD \\
\midrule
2 & (15 | 8186 | 0) & (15 | 8115 | 0) & (15 | 8433 | 1) & (15 | 8361 | 1) & (15 | 14791 | 0.69) \\
3 & (15 | 9808 | 0) & (15 | 10133 | 0) & (15 | 10405 | 1) & (15 | 10187 | 1) & (15 | 19643 | 0.08) \\
4 & (15 | 17899 | 0) & (15 | 18014 | 0) & (15 | 18668 | 1) & (15 | 19237 | 1) & (15 | 42147 | 0.34) \\
5 & (15 | 23070 | 0) & (15 | 25072 | 0) & (15 | 27268 | 1) & (15 | 26542 | 1) & (15 | 36808 | 0.03) \\
6 & (15 | 35878 | 0) & (15 | 38328 | 0) & (15 | 34170 | 1) & (15 | 36237 | 1) & (15 | 58376 | 0.11) \\
7 & (15 | 39309 | 0) & (15 | 41255 | 0) & (15 | 39085 | 1) & (15 | 39206 | 1) & (15 | 52209 | 0.02) \\
8 & (15 | 54985 | 0) & \textbf{(4.21 | -- | 0)} & \textbf{(10.8 | -- | 1)} & \textbf{(4.69 | -- | 1)} & \textbf{(4.19 | -- | --)} \\
\bottomrule
\end{tabular}
\end{table}  

\begin{table}[!t]
\centering
\scriptsize
\setlength{\tabcolsep}{3.0pt}
\renewcommand{\arraystretch}{1.15}
\caption{Kelvin--Helmholtz robustness without volume AD on \(n_{\rm edge}=32\) mesh. Entries report the last accepted time before failure. D-SBP facets retain LLF dissipation.}
\label{tab:khi_hybrid_robustness_noad}
\begin{tabular}{l l l l l l}
\toprule
\(p\) & C-2pt & C-RCM & D-2pt & D-RCM & CD-RCM \\
\midrule
2 & 2.89 & 2.25 & 3.31 & 2.53 & 3.29 \\
3 & 2.9 & 1.55 & 3.74 & 1.77 & 1.76 \\
4 & 3.06 & 1.65 & 3.69 & 2.04 & 1.77 \\
5 & 3.03 & 1.42 & 3.04 & 1.48 & 1.71 \\
6 & 3.14 & 1.8 & 3.61 & 1.84 & 1.59 \\
7 & 3.05 & 1.64 & 3.11 & 1.36 & 1.65 \\
8 & 2.85 & 1.43 & 3.13 & 1.76 & 1.42 \\
\bottomrule
\end{tabular}
\end{table}

The KHI results indicate that the RCM schemes can achieve robustness
comparable to the Hadamard-form two-point-flux discretizations, particularly
when coupled with AD of \cite{worku2026targeted}. In
\cref{tab:khi_hybrid_robustness}, all five schemes reach \(t_f=15\) for
\(p=2,\ldots,7\). At \(p=8\), only C-2pt-AD completes; CD-RCM-AD fails
at \(t\approx4.19\), close to C-RCM-AD and earlier than the two
discontinuous endpoints. Without volume AD, every calculation in
\cref{tab:khi_hybrid_robustness_noad} fails before \(t=4\), with D-2pt
reaching the largest \(t_{\rm last}\) at every degree.

For the completed hybrid runs,
\(\overline\phi_{\rm D}\in[0.02,0.69]\), so many interfaces avoid SAT
evaluations. However, CD-RCM-AD requires more accepted steps than D-RCM-AD
at every reported degree with a completed run. The smaller D-facet fraction therefore does not by itself establish a reduction in total computational cost. Less frequent C/D partition updates may reduce transfer overhead, but their effect on accuracy, robustness, and overall cost requires further study and is left for future work.

\Cref{fig:khi_hybrid_solution} shows CD-RCM-AD with \(p=4\) and
\(n_{\rm edge}=64\). At \(t=3.7\), D-SBP facets follow the rolled-up
shear layers while smoother regions retain C-SBP coupling. By
\(t=6.7\), finer flow structures and the D-SBP regions are more widespread,
but substantial parts of the mesh still use C-SBP. The partition thus
responds to the evolving flow without changing the underlying mesh.

\begin{figure}[!t]
\centering
\begin{subfigure}[t]{0.485\textwidth}
\centering
\includegraphics[width=\linewidth]{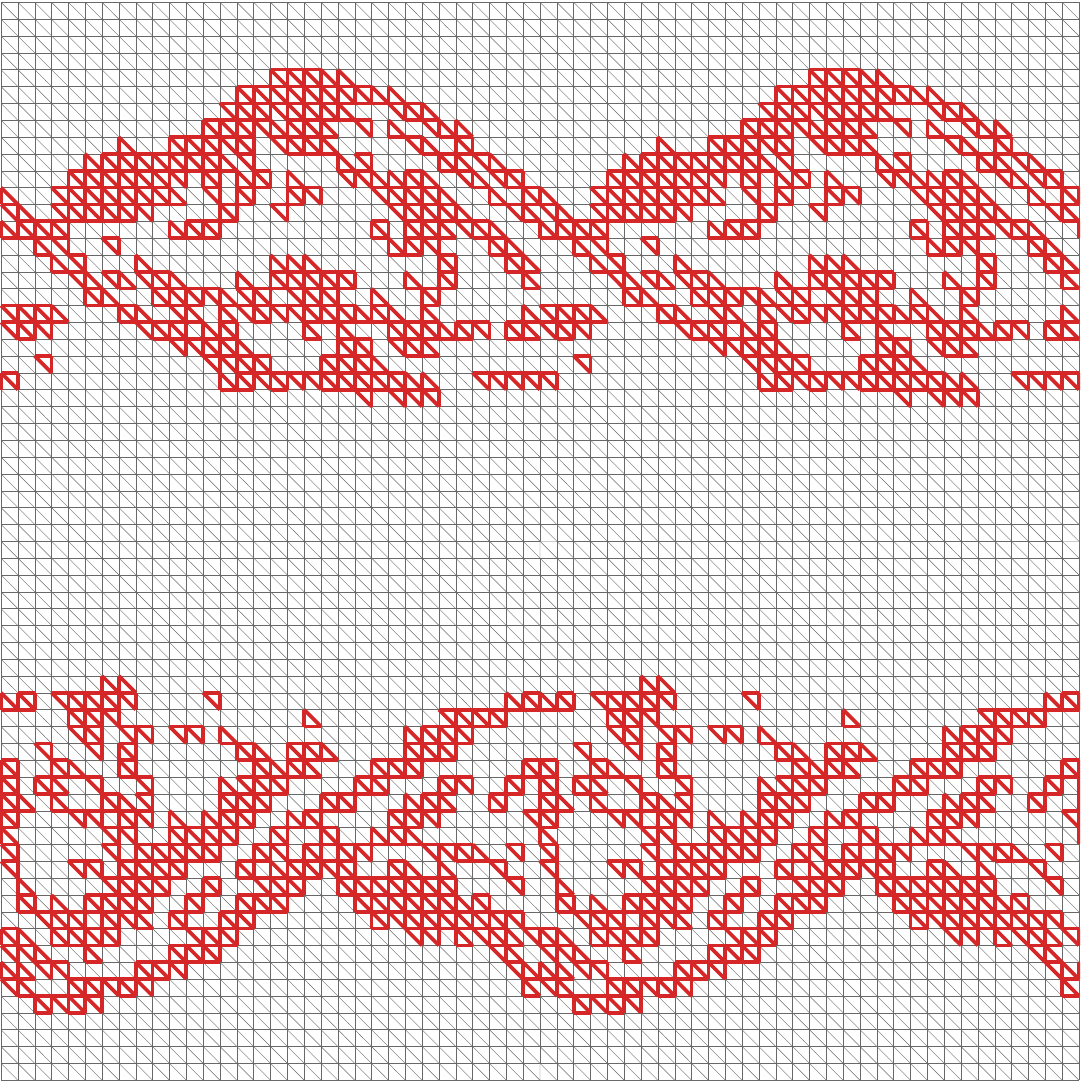}
\caption{Facet partition, \(t=3.7\)}
\label{fig:khi_mesh_t3p7}
\end{subfigure}
\hfill
\begin{subfigure}[t]{0.485\textwidth}
\centering
\includegraphics[width=\linewidth]{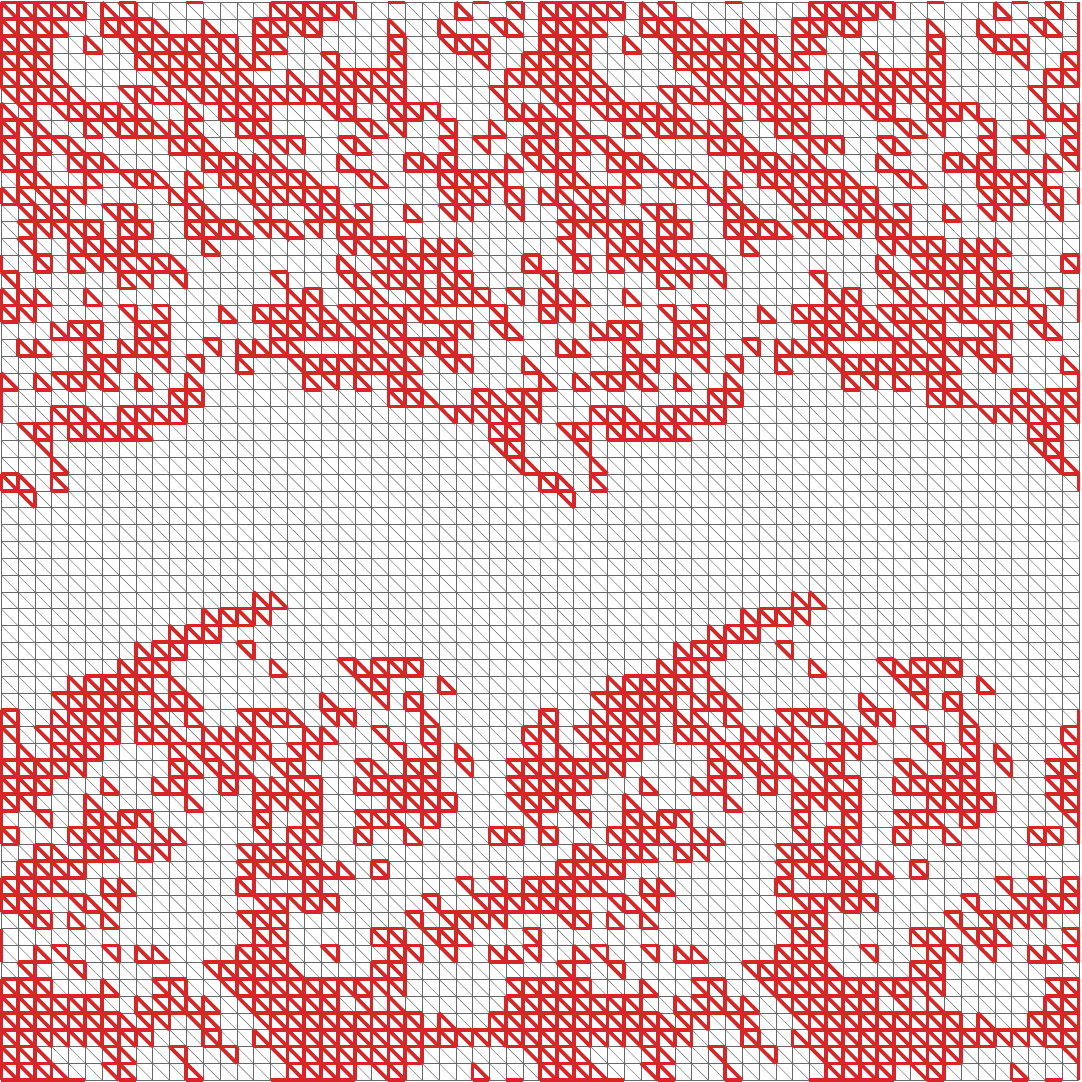}
\caption{Facet partition, \(t=6.7\)}
\label{fig:khi_mesh_t6p7}
\end{subfigure}

\begin{subfigure}[t]{0.485\textwidth}
\centering
\includegraphics[width=\linewidth]{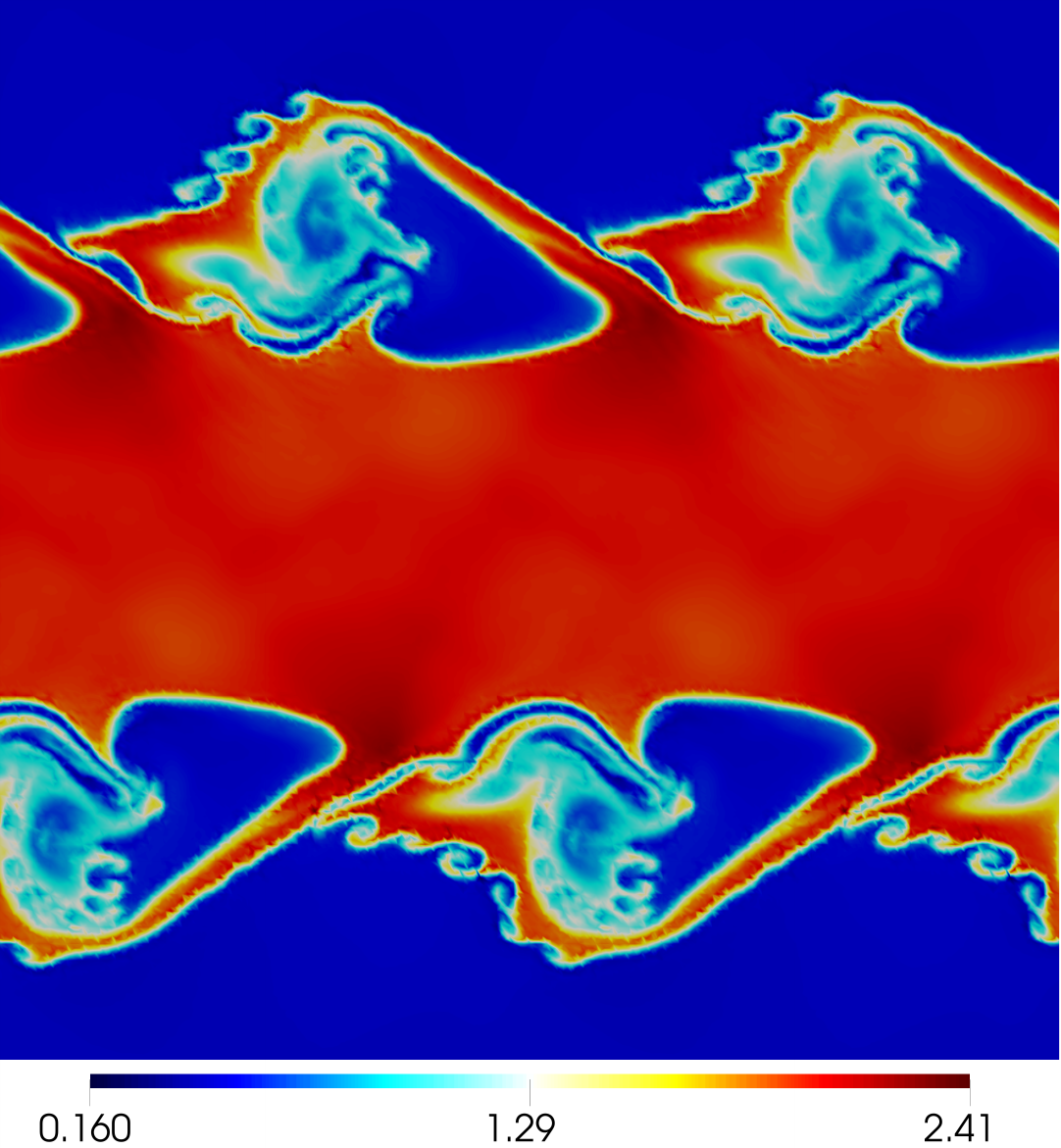}
\caption{Density, \(t=3.7\)}
\label{fig:khi_density_t3p7}
\end{subfigure}
\hfill
\begin{subfigure}[t]{0.485\textwidth}
\centering
\includegraphics[width=\linewidth]{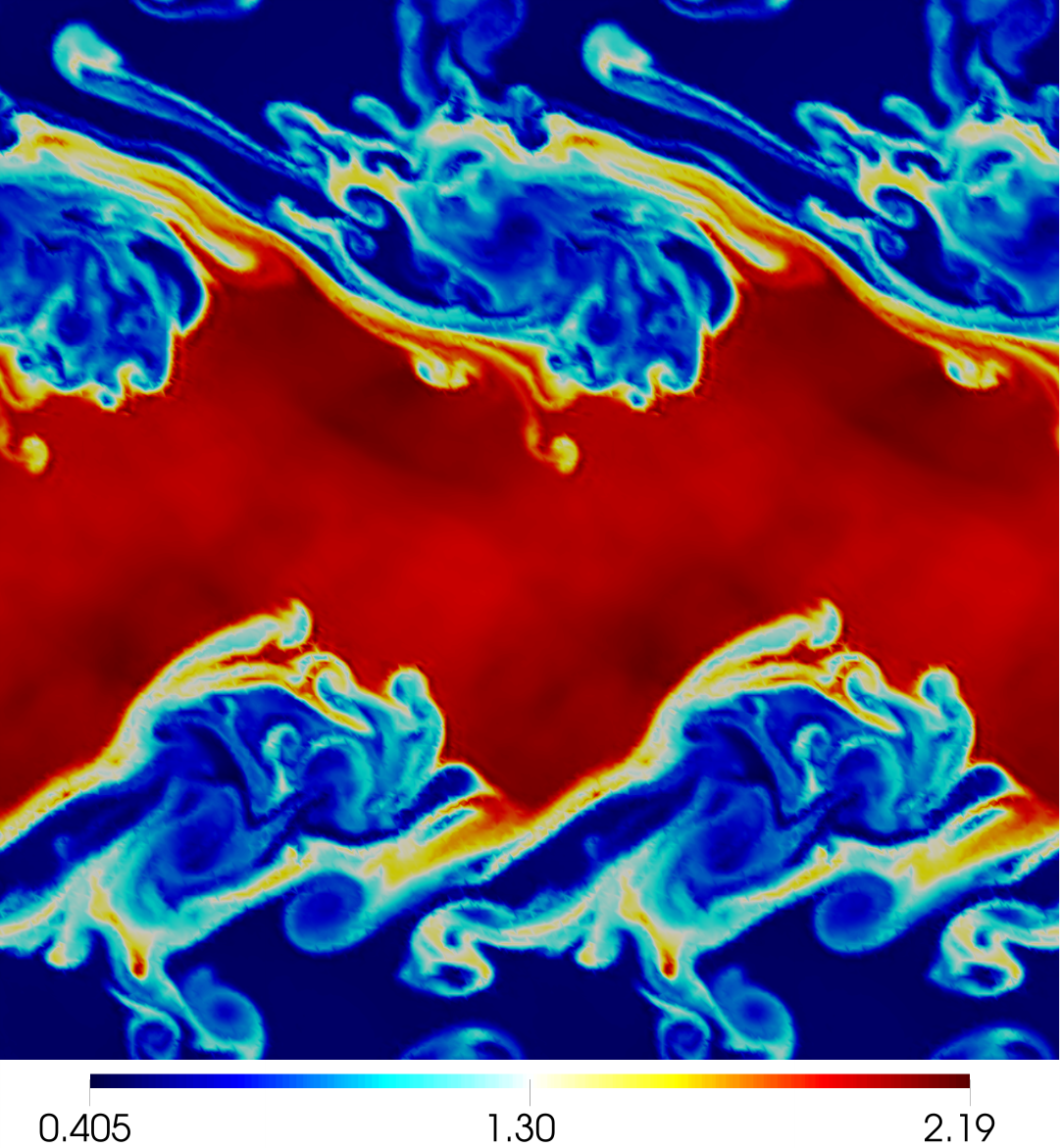}
\caption{Density, \(t=6.7\)}
\label{fig:khi_density_t6p7}
\end{subfigure}
\caption{CD-RCM-AD density and evolving facet partition for the
Kelvin--Helmholtz problem with \(p=4\) and \(n_{\rm edge}=64\). D-SBP facets
are shown in red.}
\label{fig:khi_hybrid_solution}
\end{figure}

We use the same \(p=4\), \(n_{\rm edge}=64\) CD-RCM-AD calculation to study the impact of 
the solution transfer between D and C facets on conservation and entropy stability. At every D-to-C transfer time
\(t_i\), we record
\begin{align}
    \mathcal E_{\rm con,tr}
    &\coloneqq
    \max_i\sum_{j=1}^{n_c}
    \left|
    \sum_{\Omega_k\in\mathcal T_h}
    \bm1_k^T\H_k
    \left(\bm u_{k,j}^{+}-\bm u_{k,j}^{-}\right)
    \right|,
    \label{eq:khi_transfer_conservation}\\
    \Delta S_{{\rm tr},i}
    &=
    S_h(\bm u_b(t_i^+))-S_h(\bm u_b(t_i^-)),
    \label{eq:khi_transfer_entropy_change}\\
    \Delta S_{\rm tr}(t)
    &\coloneqq
    \sum_{t_i\le t}
    \Delta S_{{\rm tr},i}.
    \label{eq:khi_transfer_entropy}
\end{align}
The first quantity checks conservation through repartitioning, while the
remaining two quantify the entropy change caused by each D-to-C projection
and its cumulative contribution, separately from spatial LLF and AD.

\begin{figure}[!t]
\centering
\begin{subfigure}[t]{0.485\textwidth}
\centering
\includegraphics[width=\linewidth]{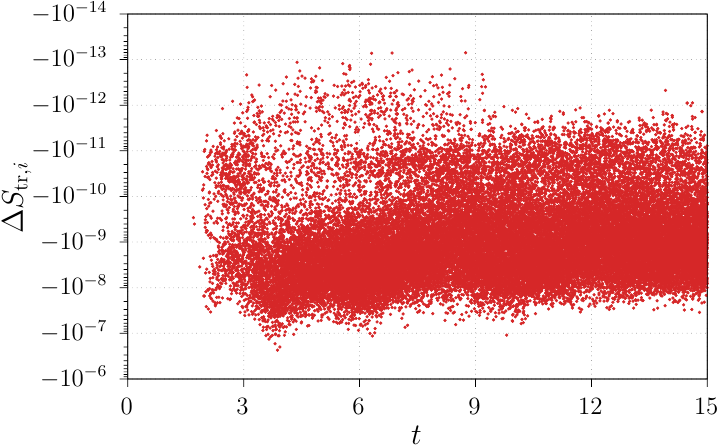}
\caption{Instantaneous change, \(\Delta S_{{\rm tr},i}\)}
\label{fig:khi_transfer_entropy_instantaneous}
\end{subfigure}
\hfill
\begin{subfigure}[t]{0.485\textwidth}
\centering
\includegraphics[width=\linewidth]{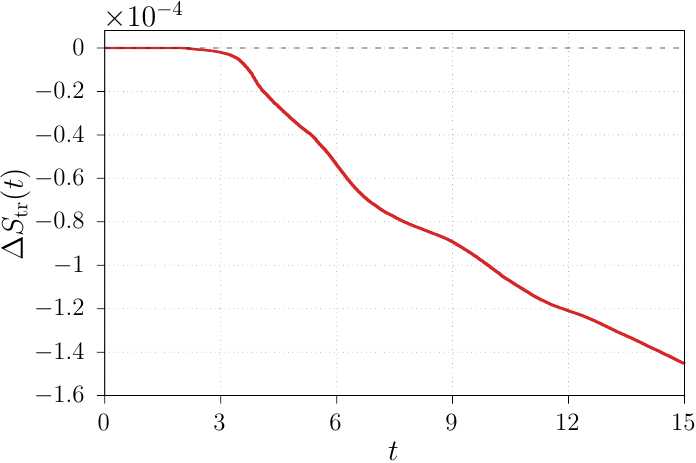}
\caption{Cumulative change, \(\Delta S_{\rm tr}(t)\)}
\label{fig:khi_transfer_entropy_cumulative}
\end{subfigure}
\caption{Entropy change caused by D-to-C transfers during the CD-RCM-AD
Kelvin--Helmholtz simulation with \(p=4\) and \(n_{\rm edge}=64\).}
\label{fig:khi_transfer_entropy}
\end{figure}

All \(43\,152\) recorded D-to-C projections decrease entropy, and
the maximum summed conservation defect is
\(\mathcal E_{\rm con,tr}\approx4.89\times10^{-15}\), consistent with
\cref{thm:cdsbp_state_transfer}. The cumulative transfer entropy change in
\cref{fig:khi_transfer_entropy} decreases monotonically to approximately
\(-1.45\times10^{-4}\) at \(t=15\). With \(S_h(0)\approx6.8\) and \(S_h(15)\approx5.1\) for this run, the cumulative transfer entropy decrease is less than \(0.01\%\) of
the net entropy decrease. Thus, we conclude that the transfer is conservative and
entropy-nonincreasing, with small cumulative entropy dissipation relative to the entropy dissipation by the spatial correction, LLF coupling, and volume AD.

\subsection{Two-dimensional Riemann problem}
\label{sec:numerics_riemann}
The four-quadrant Riemann problem tests the schemes on interacting shocks and
contact discontinuities. It is also useful for examining solution artifacts
reported for a mean-subtracted residual correction in a related high-order
calculation \cite[Sec.~5]{chan2025artificial}. Following the setup in
\cite{worku2026targeted}, we solve the Euler equations on the enlarged
periodic domain \(\Omega=[-1,1]^2\) until \(t_f=0.25\), with primitive initial
data
\begin{equation}\label{eq:riemann_numerical_initial_condition}
    (\rho,u,v,p)(x,y,0)
    =
    \begin{cases}
        (0.5313,\;0,\;0,\;0.4), & x\ge0,\;y\ge0,\\
        (1,\;0.7276,\;0,\;1), & x<0,\;y\ge0,\\
        (0.8,\;0,\;0,\;1), & x<0,\;y<0,\\
        (1,\;0,\;0.7276,\;1), & x\ge0,\;y<0.
    \end{cases}
\end{equation}
We compare C-2pt and C-RCM with volume AD of
\cite{worku2026targeted} against D-2pt, D-RCM, and CD-RCM without volume AD. For this test, the LLF wave speeds use the Davis estimate
\cite{davis1988simplified} rather than the Guermond--Popov bound, since the
no-AD D-RCM and CD-RCM calculations with the latter did not reach \(t_f\).
This choice does not imply that the Davis estimate is generally more robust
or positivity preserving; it also fails for other degrees and mesh
resolutions. For CD-RCM,
elements adjacent to the initial jumps, including across periodic boundaries,
are labeled D, and the
subsequent labels use the same \(\theta_k\) thresholds as in the
Kelvin--Helmholtz test case.

\begin{figure}[!htbp]
\centering
\begin{subfigure}[t]{0.35\textwidth}
\centering
\vspace{0pt}
\includegraphics[width=\linewidth]{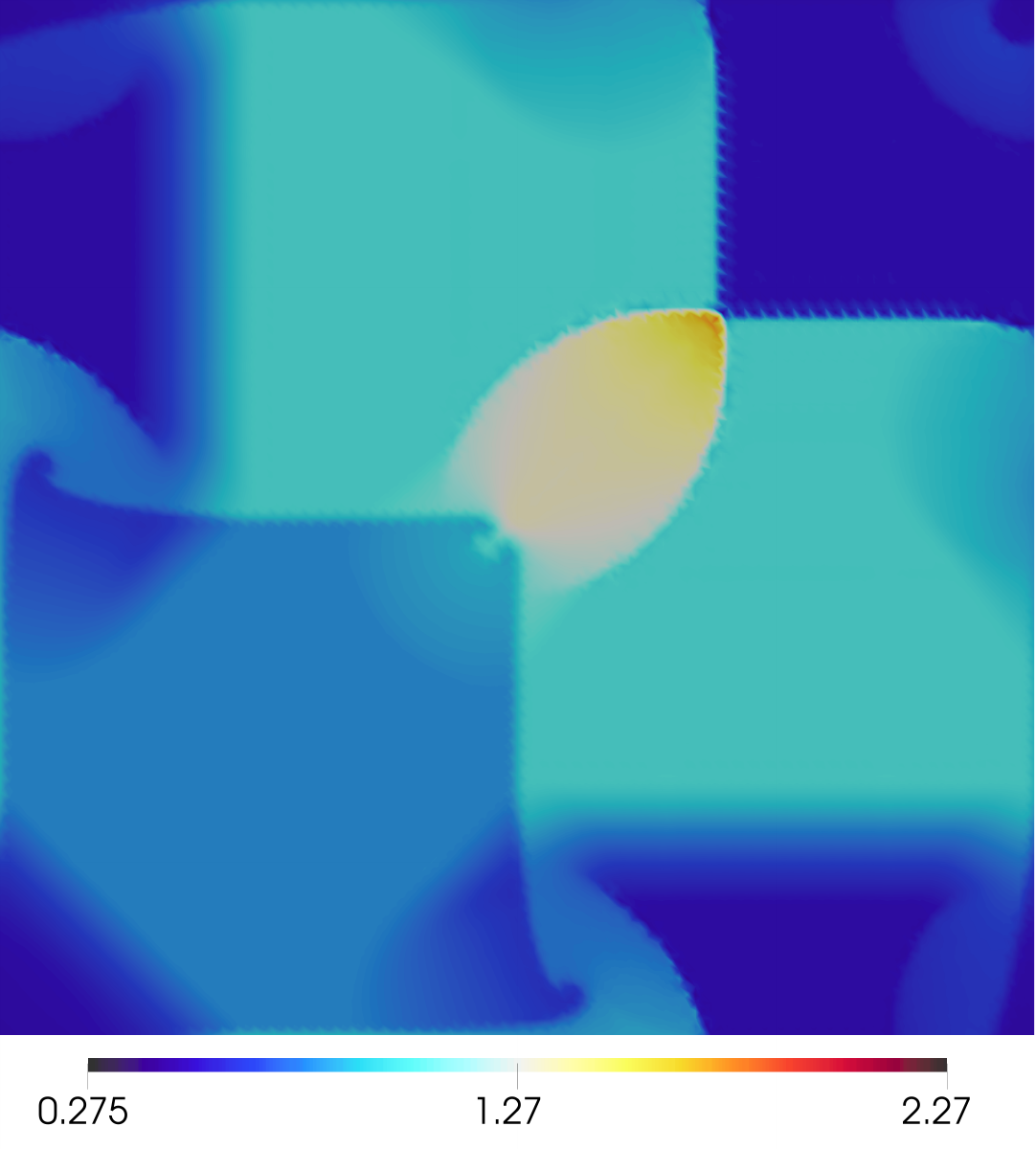}
\caption{C-2pt-AD}
\label{fig:riemann_c_2pt_ad}
\end{subfigure}
\hspace{0.04\textwidth}
\begin{subfigure}[t]{0.35\textwidth}
\centering
\vspace{0pt}
\includegraphics[width=\linewidth]{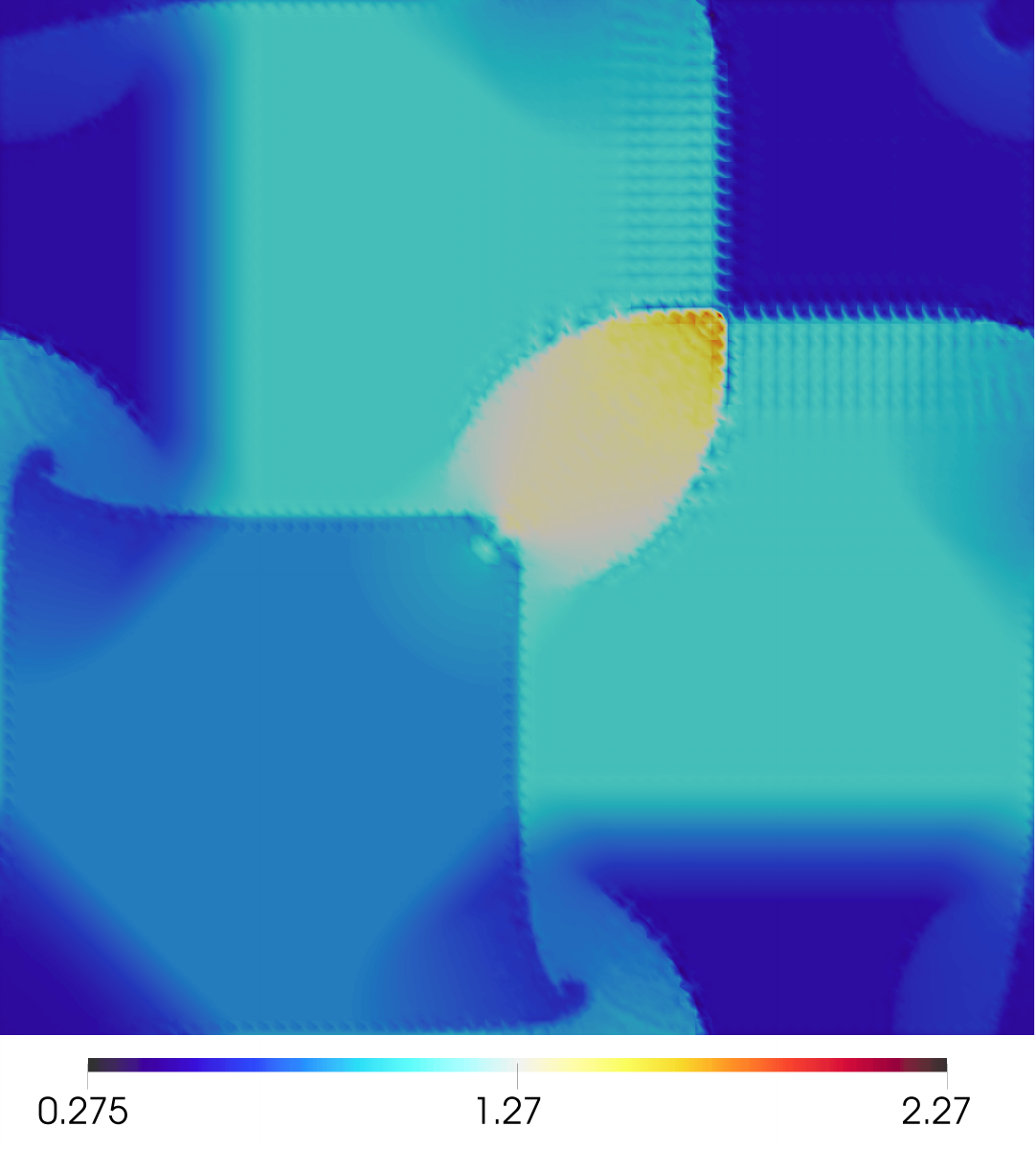}
\caption{D-2pt}
\label{fig:riemann_d_2pt}
\end{subfigure}

\smallskip
\begin{subfigure}[t]{0.35\textwidth}
\centering
\vspace{0pt}
\includegraphics[width=\linewidth]{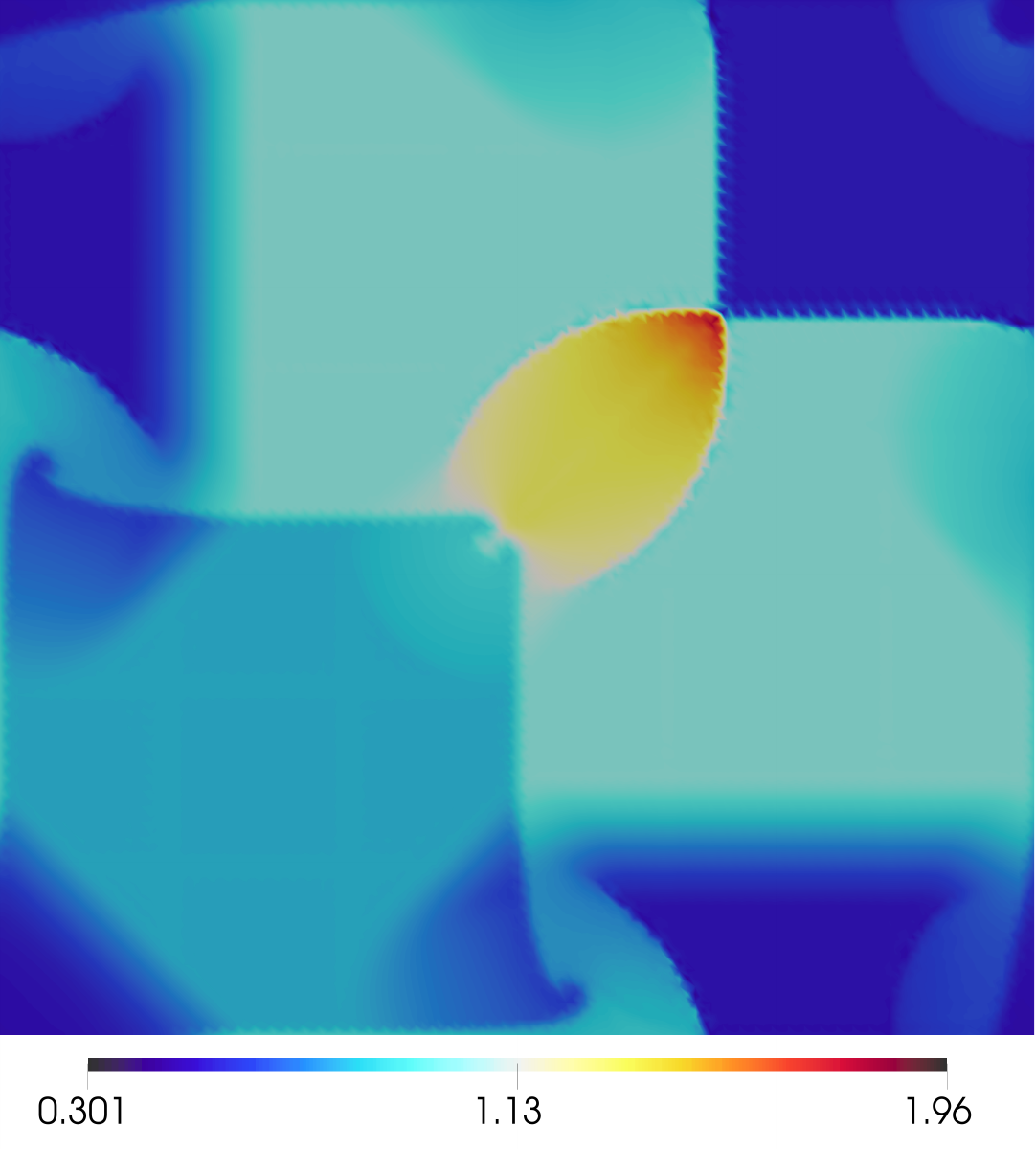}
\caption{C-RCM-AD}
\label{fig:riemann_c_rcm_ad}
\end{subfigure}
\hspace{0.04\textwidth}
\begin{subfigure}[t]{0.35\textwidth}
\centering
\vspace{0pt}
\includegraphics[width=\linewidth]{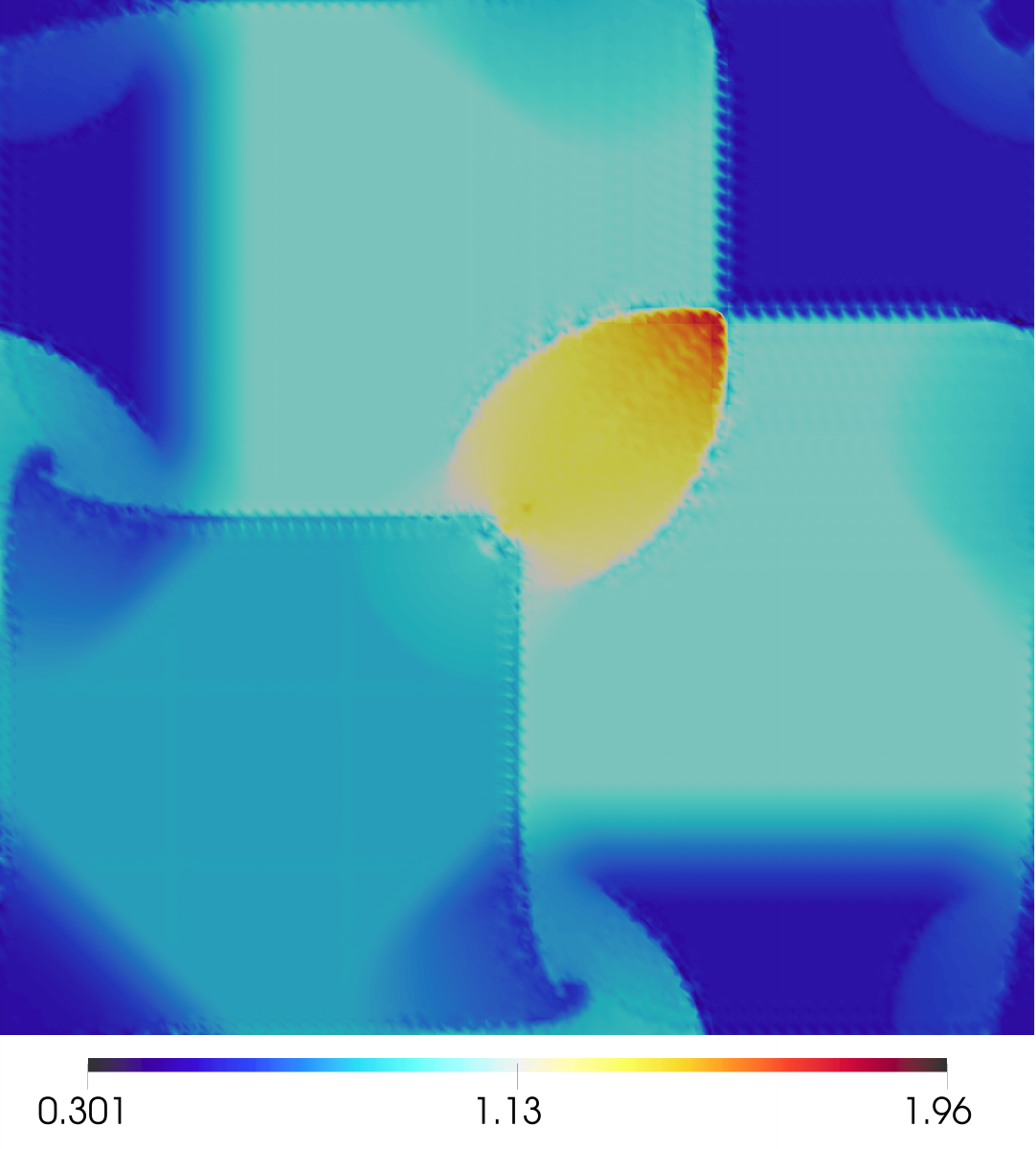}
\caption{D-RCM}
\label{fig:riemann_d_rcm}
\end{subfigure}

\smallskip
\begin{subfigure}[t]{0.35\textwidth}
\centering
\vspace{0pt}
\includegraphics[width=\linewidth]{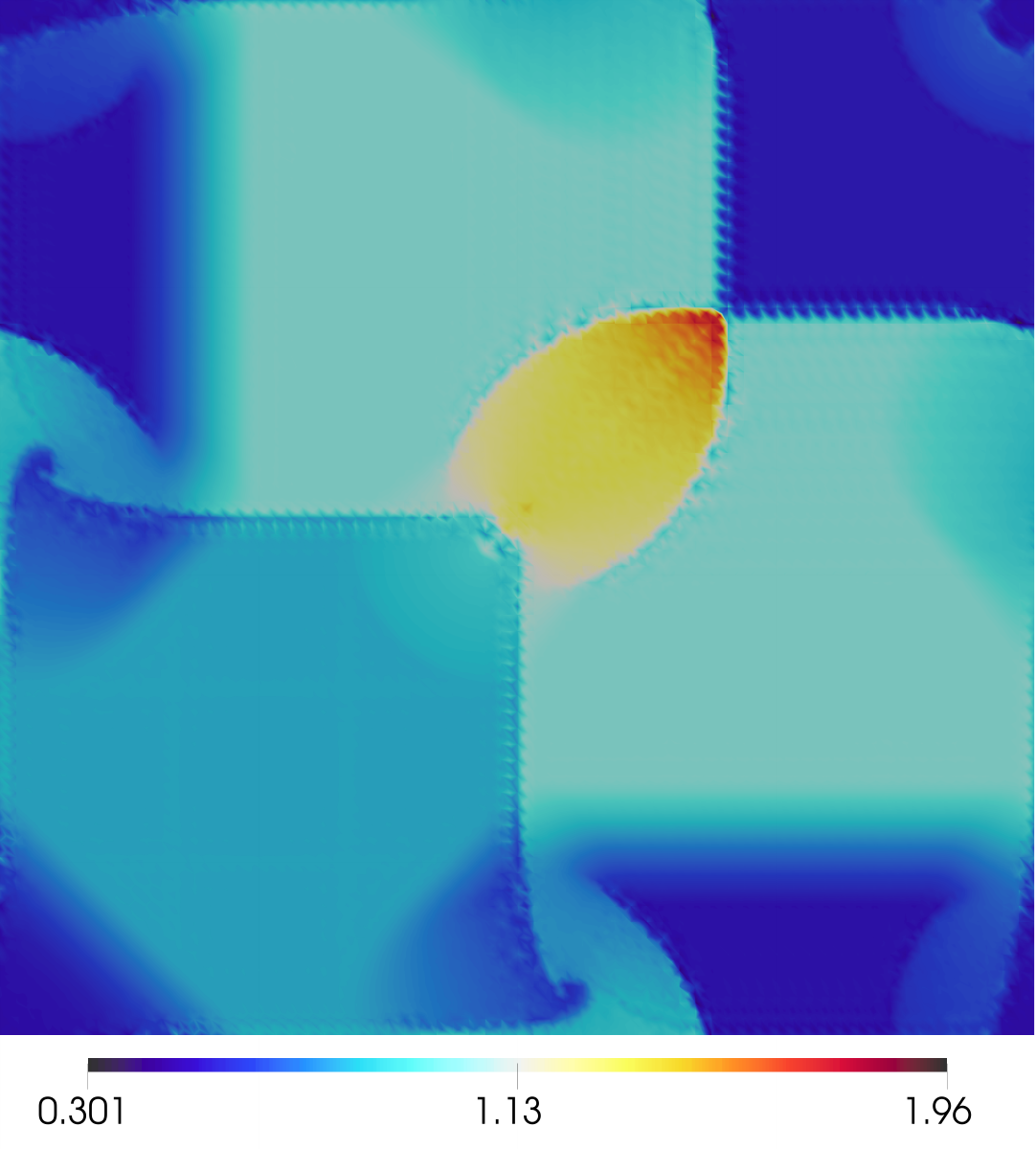}
\caption{CD-RCM}
\label{fig:riemann_cd_rcm}
\end{subfigure}
\hspace{0.04\textwidth}
\begin{subfigure}[t]{0.35\textwidth}
\centering
\vspace{0pt}
\includegraphics[width=\linewidth,trim=0 -57bp 0 0,clip]{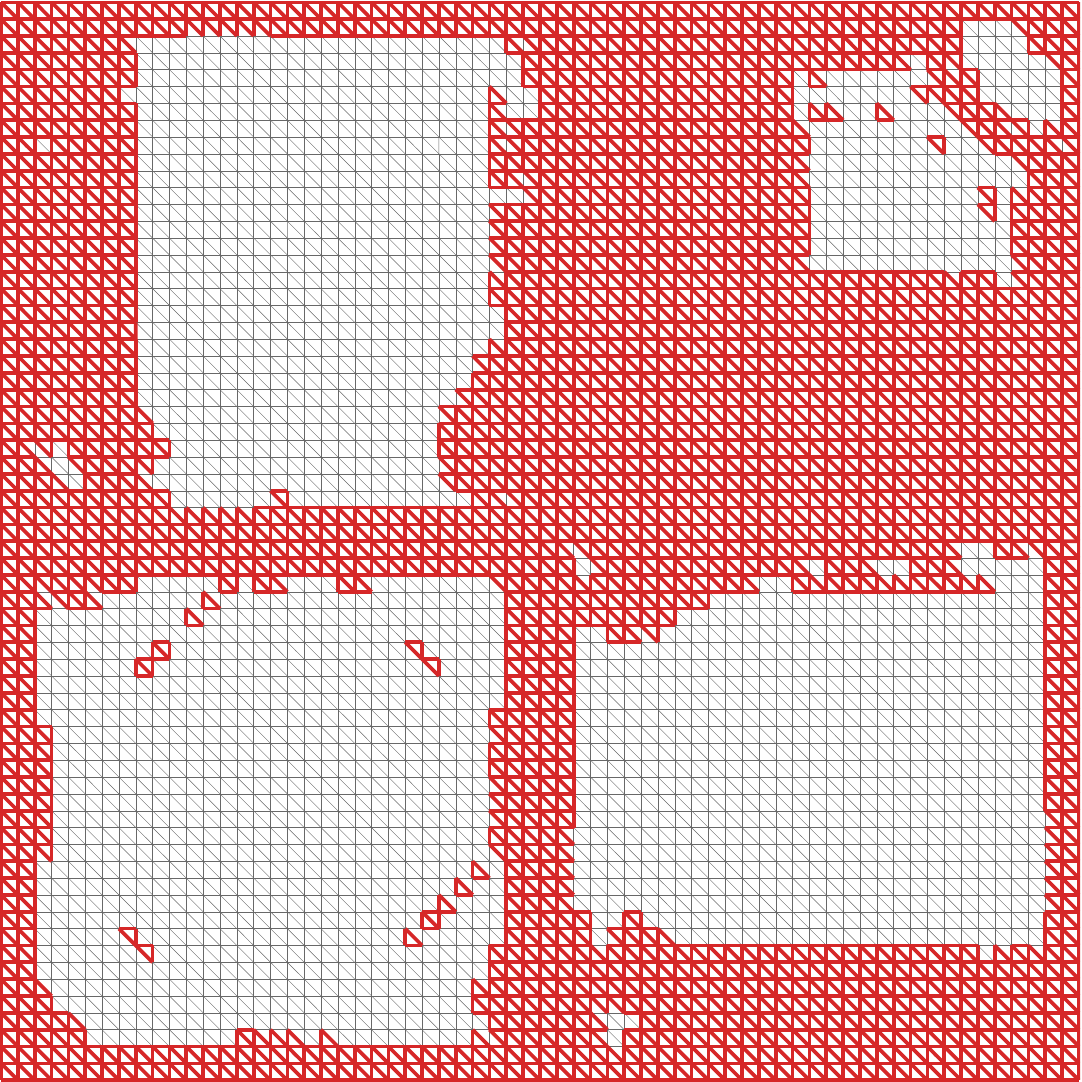}
\caption{CD-RCM facet partition}
\label{fig:riemann_cd_mesh}
\end{subfigure}
\caption{Two-dimensional Riemann problem at \(t=0.25\) with \(p=3\)
and \(n_{\rm edge}=64\). In (\subref{fig:riemann_cd_mesh}), C-SBP facets are gray
and D-SBP facets are red.}
\label{fig:riemann_hybrid_solution}
\end{figure}

\Cref{fig:riemann_hybrid_solution} shows the density fields and the
CD-RCM facet partition at \(t_f=0.25\). Compared with D-2pt in
(\subref{fig:riemann_d_2pt}), D-RCM and CD-RCM in
(\subref{fig:riemann_d_rcm}) and (\subref{fig:riemann_cd_rcm}) show less extensive oscillatory bands adjacent to
the approximately horizontal and vertical fronts. This qualitative
comparison is specific to the present mesh and degree: all three no-AD
solutions retain local oscillations near the fronts and in the central
wave-interaction region. The continuous schemes with AD in
(\subref{fig:riemann_c_2pt_ad}) and (\subref{fig:riemann_c_rcm_ad})
produce smoother density fields, although some near-front oscillations
remain. D-RCM and CD-RCM use only LLF coupling and the
\(\sigma_k^{\rm ed}\) residual correction for spatial damping; these terms
do not damp all under-resolved features at this resolution. The hybrid
density field closely resembles D-RCM while retaining C-SBP coupling away
from the marked regions in (\subref{fig:riemann_cd_mesh}).

\section{Conclusions}\label{sec:conclusions}
We have developed a hybrid C/D-SBP method with an element-local residual
correction, in which each interior facet of a conforming mesh may use either a C-SBP assembly or a D-SBP numerical-flux coupling. The same correction
term is used for every facet partition, while numerical fluxes and SATs are
evaluated only on D-SBP facets. The all-C and all-D partitions recover the
corresponding endpoint methods exactly.

For a fixed partition, each element residual satisfies conservation and entropy
identities, and the assembled method is exactly globally conservative and
satisfies a global entropy estimate with an explicit regularization defect,
while retaining design-order spatial consistency. For a changing partition, newly connected
nodal copies are transferred by a
mass-weighted average. The transfer is globally conservative,
entropy-nonincreasing, admissibility-preserving, and error non-amplifying in
the \(\overline\H_b\)-norm. These results allow the C/D partition to change
while retaining the conservation and entropy properties of the underlying
scheme.

Numerically, CD-RCM retains the smooth-solution accuracy of D-RCM,
with closely matching convergence rates and long-time behavior.
Conservation defects remain small, and the KHI transfer studies confirm
conservation near roundoff and entropy decrease for all recorded
projections. In the Burgers case, the hybrid scheme confines D-SBP coupling near
the moving shock while maintaining solution profiles and accepted-step
counts close to D-RCM. In the 2D Riemann test case, D-RCM and CD-RCM also show
less extensive near-front oscillatory bands than D-2pt without volume AD,
although local oscillations remain. Restricting SATs to D-SBP facets reduces interface work per residual evaluation; however, this does not necessarily reduce total cost,
particularly when hybrid runs require more time steps. The nonlinear tests support residual correction as a competitive alternative to Hadamard-form flux differencing, with the
clearest agreement in robustness when coupled with AD of
\cite{worku2026targeted}.

Future work will extend the construction to nonmatching \(hp\)-adaptation using D-SBP buffer regions and compatible mortar operators,
and to viscous systems for which restricting SATs to D-SBP facets may provide
larger computational savings.

\appendix
\crefalias{section}{appendix}
\Crefname{appendix}{Appendix}{Appendices}
\crefname{appendix}{Appendix}{Appendices}

\section{Entropy-defect superconvergence: Proof of
\texorpdfstring{\cref{lem:entropy_defect_scaling}}
{Lemma~\ref{lem:entropy_defect_scaling}}}
\label{app:entropy_defect_scaling}

\begin{proof}
    For diagonal-\(\E\) operators, the boundary decomposition and nodal
    extraction property, together with
    \(\Q_{x_i k}+\Q_{x_i k}^T=\E_{x_i k}\) and
    \(\psi^{(i)}=\fnb W^T\fnb F^{(i)}-\fn G^{(i)}\), give
    \begin{equation}\label{eq:common_entropy_defect}
        \delta_k
        =
        \sum_{i=1}^{d}
        \left(
            \bm w_k^T\overline{\Q}_{x_i k}\bm f_k^{(i)}
            -
            \bm1_k^T\E_{x_i k}\bm g_k^{(i)}
        \right)
        =
        \sum_{i=1}^{d}
        \left(
            -\bm w_k^T\overline{\Q}_{x_i k}^T\bm f_k^{(i)}
            +\bm1_k^T\E_{x_i k}\bm\psi_k^{(i)}
        \right).
    \end{equation}
    Let \(\bm x_{k,c}\in\Omega_k\) be fixed and apply Taylor expansions
    componentwise in the physical coordinates. For example, in two
    dimensions and for \(\ell=1,\ldots,n_c\), let
    \(\Delta x\coloneqq x-x_{k,c}\) and
    \(\Delta y\coloneqq y-y_{k,c}\). Then, we have
    \begin{equation}\label{eq:componentwise_taylor_expansion}
        \begin{aligned}
            \fn W_{\ell}(x,y)
            ={}&
            \fn W_{\ell}(x_{k,c},y_{k,c})
            +
            \frac{\partial\fn W_{\ell}}{\partial x}
            \Delta x
            +
            \frac{\partial\fn W_{\ell}}{\partial y}
            \Delta y
            \\
            &+
            \frac{1}{2}
            \left[
                \frac{\partial^2\fn W_{\ell}}{\partial x^2}
                \Delta x^2
                +
                2
                \frac{\partial^2\fn W_{\ell}}{\partial x\partial y}
                \Delta x\Delta y
                +
                \frac{\partial^2\fn W_{\ell}}{\partial y^2}
                \Delta y^2
            \right]
            +
            \cdots,
        \end{aligned}
    \end{equation}
    where all derivatives are evaluated at \((x_{k,c},y_{k,c})\). Let
    \(\fn W_{k,m,\ell}\) denote the sum of all terms of total degree \(m\)
    in this expansion and set
    \[
        \fnb W_{k,m}
        \coloneqq
        \left[
            \fn W_{k,m,1},
            \ldots,
            \fn W_{k,m,n_c}
        \right]^T.
    \]
    For example, \(\fn W_{k,2,\ell}\) is the quadratic contribution in
    \cref{eq:componentwise_taylor_expansion}. Defining
    \(\fnb F_{k,n}^{(i)}\) and \(\fn G_{k,r}^{(i)}\) analogously, and using
    \(\lVert\bm x-\bm x_{k,c}\rVert=\fn{O}(h)\), we obtain
    \begin{align*}
        \fnb W
        &=
        \sum_{m=0}^{2p}\fnb W_{k,m}
        +
        \fn{O}\!\left(h^{2p+1}\right),
        &
        \fnb W_{k,m}
        &=
        \fn{O}(h^m),
        \\
        \fnb F^{(i)}
        &=
        \sum_{n=0}^{2p}\fnb F_{k,n}^{(i)}
        +
        \fn{O}\!\left(h^{2p+1}\right),
        &
        \fnb F_{k,n}^{(i)}
        &=
        \fn{O}(h^n),
        \\
        \fn G^{(i)}
        &=
        \sum_{r=0}^{2p}\fn G_{k,r}^{(i)}
        +
        \fn{O}\!\left(h^{2p+1}\right),
        &
        \fn G_{k,r}^{(i)}
        &=
        \fn{O}(h^r).
    \end{align*}
    Let the corresponding lowercase bold letters denote nodal
    restrictions. Substitution into \cref{eq:common_entropy_defect} gives
    \[
        \begin{aligned}
        \delta_k
        ={}&
        \sum_{i=1}^{d}
        \Bigg[
            \left(
                \sum_{m=0}^{2p}\bm w_{k,m}
                +
                \fn{O}(h^{2p+1})
            \right)^T
            \overline{\Q}_{x_i k}
            \left(
                \sum_{n=0}^{2p}\bm f_{k,n}^{(i)}
                +
                \fn{O}(h^{2p+1})
            \right)
        \\
        &\hspace{30mm}
            -
            \bm1_k^T\E_{x_i k}
            \left(
                \sum_{r=0}^{2p}\bm g_{k,r}^{(i)}
                +
                \fn{O}(h^{2p+1})
            \right)
        \Bigg].
        \end{aligned}
    \]
    Consider a volume term with \(m+n\le2p\). If \(n\le p\), the
    derivative of \(\fnb F_{k,n}^{(i)}\) and the resulting volume
    quadrature are exact. Otherwise, \(m\le p-1\), and the SBP property
    transfers the derivative to \(\fnb W_{k,m}\), after which the volume
    and facet quadratures are exact. Therefore, we have
    \begin{align*}
        \bm w_{k,m}^T
        \overline{\Q}_{x_i k}
        \bm f_{k,n}^{(i)}
        &=
        \int_{\Omega_k}
        \fnb W_{k,m}^T
        \pder[\fnb F_{k,n}^{(i)}]{x_i}
        \dd\Omega,
        && m+n\le2p,
        \\
        \bm1_k^T\E_{x_i k}\bm g_{k,r}^{(i)}
        &=
        \int_{\Omega_k}
        \pder[\fn G_{k,r}^{(i)}]{x_i}
        \dd\Omega,
        && r\le2p.
    \end{align*}
    The physical mass-multiplied derivative and boundary operators scale as
    \(\fn{O}(h^{d-1})\). Thus, every term containing a Taylor remainder is
    \(\fn{O}(h^{2p+d})\). In addition, both the discrete volume terms and their corresponding integrals are \(\fn{O}(h^{m+n+d-1})\); hence contributions with \(m+n>2p\) can be absorbed into the \(\fn{O}(h^{2p+d})\) remainder, giving
    \[
        \begin{aligned}
        \delta_k
        ={}&
        \sum_{i=1}^{d}
        \int_{\Omega_k}
        \left[
            \left(
                \sum_{m=0}^{2p}\fnb W_{k,m}
            \right)^T
            \left(
                \sum_{n=0}^{2p}
                \pder[\fnb F_{k,n}^{(i)}]{x_i}
            \right)
            -
            \sum_{r=0}^{2p}
            \pder[\fn G_{k,r}^{(i)}]{x_i}
        \right]
        \dd\Omega
        +
        \fn{O}(h^{2p+d}).
        \end{aligned}
    \]
    The Taylor remainder estimates yield
    \begin{align*}
        \sum_{m=0}^{2p}\fnb W_{k,m}
        &=
        \fnb W+\fn{O}(h^{2p+1}),
        \\
        \sum_{n=0}^{2p}
        \pder[\fnb F_{k,n}^{(i)}]{x_i}
        &=
        \pder[\fnb F^{(i)}]{x_i}+\fn{O}(h^{2p}),
        \\
        \sum_{r=0}^{2p}
        \pder[\fn G_{k,r}^{(i)}]{x_i}
        &=
        \pder[\fn G^{(i)}]{x_i}+\fn{O}(h^{2p}).
    \end{align*}
    Since \(V_k=\fn{O}(h^d)\), substitution and the entropy compatibility
    identity give
    \[
        \begin{aligned}
        \delta_k
        &=
        \sum_{i=1}^{d}
        \int_{\Omega_k}
        \left[
            \fnb W^T
            \pder[\fnb F^{(i)}]{x_i}
            -
            \pder[\fn G^{(i)}]{x_i}
            +
            \fn{O}(h^{2p})
        \right]
        \dd\Omega
        +
        \fn{O}(h^{2p+d})
        =
        \fn{O}(h^{2p+d}),
        \end{aligned}
    \]
    which proves the result.
\end{proof}

\section{Sensitivity of noncompatible SAT-dependent corrections}
\label{app:sat_dependent_correction}
By \cref{eq:sat_defect_decomposition}, all neighboring-trace dependence
of \(\delta_k^{\rm SAT}\) enters through \(\mathcal E_{\gamma k}^{\star}\). We
quantify the effect of using noncompatible numerical flux pairs \((\bm f^*_{\gamma k}, \bm g^{\star}_{\gamma k})\), which do not lead to the cancellation in \cref{eq:sat_ec_cancellation}, on the SAT-dependent correction \(\bm c_{k,\epsilon}^{\rm SAT}\coloneqq \bm\alpha_{k,\epsilon}\delta_k^{\rm SAT}\). Consider the centered (symmetric) numerical entropy flux
\begin{equation}\label{eq:app_sat_centered_entropy_flux}
    \bm g_{\gamma k}^{\star}
    \coloneqq
    \frac12\sum_{i=1}^{d}\N_{x_i\gamma k}
    \left(\bm g_{\gamma k}^{(i)}+\bm g_{\gamma v}^{(i)}\right).
\end{equation}
To study the resulting dependence of \(\bm c_{k,\epsilon}^{\rm SAT}\) on the
neighboring trace, keep the state on \(\Omega_k\) fixed and set
\[
    \bm e_\gamma
    \coloneqq
    \bm u_{\gamma k}-\bm u_{\gamma v}.
\]
Since \(\bm u_{\gamma v}+\bm e_\gamma=\bm u_{\gamma k}\), the path
\(\bm u_{\gamma v}+\theta\bm e_\gamma\) joins the actual neighboring trace
to its matched value. Along this path, the facet mismatch in
\cref{eq:sat_defect_decomposition} is
\[
    \begin{aligned}
    \mathcal E_{\gamma k}^{\star}(\theta)
    ={}&
    \bm w_{\gamma k}^T\overline{\B}_\gamma
    \bm f_{\gamma k}^{*}
    \left(\bm u_{\gamma k},
    \bm u_{\gamma v}+\theta\bm e_\gamma\right)\\
    &-\frac12\sum_{i=1}^{d}
    \bm1_\gamma^T\B_\gamma\N_{x_i\gamma k}
    \left[
    \bm g^{(i)}(\bm u_{\gamma k})
    +\bm g^{(i)}
    \left(\bm u_{\gamma v}+\theta\bm e_\gamma\right)
    \right]
    -\bm1_\gamma^T\B_\gamma\bm\psi_{\gamma k}^{n},
    \end{aligned}
\]
where \(0\le\theta\le1\).
Thus, \(\mathcal E_{\gamma k}^{\star}(0)\) is the actual facet mismatch and
\(\mathcal E_{\gamma k}^{\star}(1)=0\) is its matched-trace value.  Assume that this
segment lies in a compact subset of the admissible set on which the physical
entropy fluxes and numerical facet flux are continuously differentiable.  The
mean-value theorem then gives, for some \(\theta_*\in(0,1)\),
\begin{equation}\label{eq:app_sat_mean_value}
    \begin{aligned}
    \Delta\mathcal E_{\gamma k}^{\star}
    &\coloneqq
    \mathcal E_{\gamma k}^{\star}(1)-\mathcal E_{\gamma k}^{\star}(0)
    =
    \left.
    \frac{\mathrm d\mathcal E_{\gamma k}^{\star}}
         {\mathrm d\bm u_{\gamma v}}
    \right|_{\bm u_{\gamma v}+\theta_*\bm e_\gamma}
    \bm e_\gamma.
    \end{aligned}
\end{equation}
On this set,
\(\mathrm d\bm g^{(i)}/\mathrm d\bm u=\fn{O}(1)\).  Since
\(\bm f_{\gamma k}^{*}\) includes the surface-scaled normal,
\(\mathrm d\bm f_{\gamma k}^{*}/\mathrm d\bm u_{\gamma v}
=\fn{O}(h^{d-1})\).  Together with
\(\B_\gamma=\fn{O}(1)\) and
\(\N_{x_i\gamma k}=\fn{O}(h^{d-1})\), these estimates show that
\(\mathrm d\mathcal E_{\gamma k}^{\star}/\mathrm d\bm u_{\gamma v}
=\fn{O}(h^{d-1})\).  Hence, \cref{eq:app_sat_mean_value} implies
\begin{equation}\label{eq:app_sat_facet_sensitivity}
    |\Delta\mathcal E_{\gamma k}^{\star}|
    \le
    C h^{d-1}\lVert\bm e_\gamma\rVert_2 .
\end{equation}

The remaining terms in \(\delta_k^{\rm SAT}\) are unchanged when only this
neighboring trace is perturbed; hence,
\(\Delta\delta_k^{\rm SAT}=\Delta\mathcal E_{\gamma k}^{\star}\).  Moreover,
\(\bm\alpha_{k,\epsilon}\) depends only on the state in \(\Omega_k\), which
has been held fixed. If the element-local state satisfies the smooth-state
assumptions of \cref{lem:correction_terms_scaling}, then
\(\lVert\bm\alpha_{k,\epsilon}\rVert_2=\fn O(h^{-1})\) by
\cref{thm:reference_normalization}. Therefore, we obtain
\begin{align}
    \Delta\!\left(
    \overline{\H}_k^{-1}\bm c_{k,\epsilon}^{\rm SAT}
    \right)
    &=
    \overline{\H}_k^{-1}\bm\alpha_{k,\epsilon}
    \Delta\delta_k^{\rm SAT},
    \label{eq:app_sat_correction_change}\\
    \left\lVert
    \Delta\!\left(
    \overline{\H}_k^{-1}\bm c_{k,\epsilon}^{\rm SAT}
    \right)
    \right\rVert_2
    &\le
    \lVert\overline{\H}_k^{-1}\rVert_2
    \lVert\bm\alpha_{k,\epsilon}\rVert_2
    |\Delta\mathcal E_{\gamma k}^{\star}|\notag\\
    &\le
    C h^{-d}h^{-1}h^{d-1}
    \lVert\bm e_\gamma\rVert_2
    =C h^{-2}\lVert\bm e_\gamma\rVert_2.
    \label{eq:app_sat_correction_sensitivity}
\end{align}
For the diagonal-\(\E\) SBP operators considered in this paper, the
nodal sampling of a smooth exact solution satisfies
\(\bm u_{\gamma k}=\bm u_{\gamma v}\), and hence
\(\bm e_\gamma=\bm0\). Consequently,
\(\mathcal E_{\gamma k}^{\star}=0\),
\(\delta_k^{\rm SAT}=\delta_k=\fn O(h^{2p+d})\), and
\begin{equation}
    \left\lVert
    \overline{\H}_k^{-1}\bm c_{k,\epsilon}^{\rm SAT}
    \right\rVert_2
    =\fn O(h^{2p-1}).
\end{equation}
Thus, the neighboring-trace dependence does not alter the formal consistency
estimate. For a general D-SBP numerical state, however, continuity is not imposed and \(\lVert\bm e_\gamma\rVert_2\) need not be small. In that case,
whenever the element-local smooth-state scaling used above holds,
\cref{eq:app_sat_correction_sensitivity} bounds the change in the
inverse-norm-scaled correction by \(C h^{-2}\lVert\bm e_\gamma\rVert_2\).
Hence, mismatches between neighboring D-SBP traces may induce large
changes in noncompatible SAT-dependent corrections after application
of the inverse norm. As in the ECAV construction of Chan
\cite{chan2025artificial}, the entropy defect \(\delta_k\) used in this work
does not depend on traces from neighboring elements and therefore avoids this
sensitivity mechanism.

For a SAT-dependent correction using a non-diagonal-\(\E\) SBP operator with the same facet-defect structure, suppose that the two restrictions of a smooth exact state satisfy \(\lVert\bm e_\gamma\rVert_2=\fn O(h^{p+1})\) and that the facet defect satisfies the bounded-derivative estimate used above. We also assume that \(\delta_k=\fn O(h^{2p+d})\) (note that \cref{lem:entropy_defect_scaling} assumes the use of diagonal-\(\E\) SBP operators). Then, using
\(\delta_k=\fn O(h^{2p+d})\),
\(\lVert\bm\alpha_{k,\epsilon}\rVert_2=\fn O(h^{-1})\), and
\(\lVert\overline{\H}_k^{-1}\rVert_2=\fn O(h^{-d})\), we have
\begin{equation}
    |\Delta\mathcal E_{\gamma k}^{\star}|=\fn O(h^{p+d}),
    \quad
    |\delta_k^{\rm SAT}|=\fn O(h^{p+d}),
    \qquad
    \left\lVert
    \overline{\H}_k^{-1}\bm c_{k,\epsilon}^{\rm SAT}
    \right\rVert_2
    =\fn O(h^{p-1}).
\end{equation}
This bound need not be sharp, but it does not guarantee design-order accuracy.
Abgrall \cite{abgrall2018general} reported a related one-order loss and showed
that it can be avoided with a compatible choice of the numerical entropy flux.
For collocated facet and volume nodes, the compatible entropy-conservative construction in
\cref{sec:relation_existing_corrections} gives the exact cancellation in
\cref{eq:sat_ec_cancellation}, eliminating both this possible order reduction and the
neighboring-trace sensitivity.

\section*{CRediT authorship contribution statement}
\textbf{Zelalem Arega Worku:} Conceptualization, Methodology, Formal analysis, Software, Investigation, Validation, Data Curation, Visualization, Writing -- original draft, Writing -- review \& editing. 

\textbf{David C. Del Rey Fern{\'a}ndez:} Conceptualization, Investigation, Validation, Supervision, Writing -- review \& editing.

\section*{Data availability}
Data are available from the corresponding author upon reasonable request.

\section*{Declaration of generative AI and AI-assisted
technologies in the manuscript preparation process}
OpenAI Codex was used for coding assistance and to identify and address writing and clarity issues in the manuscript. The authors wrote and reviewed the entire manuscript and take full responsibility for its content.

\bibliographystyle{elsarticle-num}
\bibliography{ref}

\end{document}